\documentclass[12pt]{amsart}
\usepackage{amscd,amssymb,graphicx,color,a4wide,hyperref,cite,amsthm,amsmath}

\usepackage{hyperref}
\usepackage{url}

\usepackage{mathrsfs}

\usepackage{epstopdf}

\usepackage[all]{xy}
\newtheorem{theorem}{Theorem}[section]
\newtheorem{lemma}[theorem]{Lemma}
\newtheorem{proposition}[theorem]{Proposition}
\newtheorem{corollary}[theorem]{Corollary}

\theoremstyle{definition}
\newtheorem{definition}[theorem]{Definition}
\newtheorem{construction}[theorem]{Construction}

\newtheorem{example}[theorem]{Example}

\theoremstyle{remark}
\newtheorem{remark}[theorem]{Remark}

\numberwithin{equation}{section}
\numberwithin{figure}{section}

\newcommand{\bbfamily}{\fontencoding{U}\fontfamily{bbold}\selectfont}
\newcommand{\textbb}[1]{{\bbfamily#1}}

\newcommand {\lfor} {\mbox{\textbb{[}}}
\newcommand {\rfor} {\mbox{\textbb{]}}}

\newcommand{\M} {\mathcal{M}}
\newcommand{\shM} {\mathcal{M}}
\newcommand{\shP} {\mathcal{P}}

\newcommand{\ZZ} {\mathbb{Z}}
\newcommand{\QQ} {\mathbb{Q}}
\newcommand{\RR} {\mathbb{R}}
\newcommand{\CC} {\mathbb{C}}

\renewcommand{\AA} {\mathbb{A}}

\newcommand {\T} {\mathfrak{T}}
\newcommand {\lra} {\longrightarrow}

\newcommand{\bigslant}[2]{{\raisebox{.2em}{$#1$}\left/\raisebox{-.2em}{$#2$}\right.}}

\newcommand\restr[2]{{
  \left.\kern-\nulldelimiterspace 
  #1 
  \vphantom{\big|} 
  \right|_{#2} 
  }}

\newcommand {\ol} {\overline}
  
\usepackage{tikz}
\usepackage{MnSymbol}
\usetikzlibrary{matrix,arrows,decorations.pathmorphing}
\usepackage{hyperref}
\usepackage{enumerate}
\usepackage{float}
\usepackage{tikz-cd}
\usepackage{relsize}
\usepackage{mathtools}

\newcommand\shO{\mathcal{O}}

\newcommand\A{\mathbb A}
\newcommand\C{\mathbb C}

\newcommand\LL{\mathbb L}
\newcommand\NN{\mathbb N}

\newcommand\PP{\mathbb P}
\newcommand\R{\mathbb R}
\newcommand\Q{\mathbb Q}

\newcommand\Z{\mathbb Z}

\newcommand\cN{\mathcal N}
\newcommand\cO{\mathcal O}

\newcommand\cX{\mathcal X}

\newcommand\Spec{\operatorname{Spec}\,}

\newcommand\Hom{\operatorname{Hom}}

\newcommand\Ker{\operatorname{Ker}\,}
\newcommand\Coker{\operatorname{Coker}\,}

\newcommand\Rea{\operatorname{Re}\,}

\newcommand\Ima{\operatorname{Im}\,}

\newtheoremstyle{cited}%
  {3pt}
  {3pt}
  {\itshape}
  {}
  {\bfseries}
  {.}
  {.5em}
  {\thmname{#1} \thmnumber{#2} \thmnote{\normalfont#3}}

\theoremstyle{cited}

\begin{document}

\title[Real Log Curves in Toric Varieties]{Real Log Curves in Toric Varieties, Tropical Curves, and Log Welschinger Invariants}
\author{H\"ulya Arg\"uz}
\address{Laboratoire de Math\'ematiques, Universit\'e de
Versailles St Quentin en Yvelines, France}
\email{nuromur-hulya.arguz@uvsq.fr}

\author{Pierrick Bousseau}
\address{Institute for Theoretical Studies\\ ETH Zurich \\ 8092 Zurich\\ Switzerland}
\email{pierrick.bousseau@eth-its.ethz.ch}

\date{\today}

\begin{abstract}
We give a tropical description of the counting of real log curves in toric degenerations of toric varieties. 
We treat the case of genus zero curves and all non-superabundant higher-genus situations. 
The proof relies on log deformation theory and is a real version of the 
Nishinou-Siebert approach to the tropical correspondence theorem for complex curves.
In dimension two, we use similar techniques to study the counting of real log curves with Welschinger signs 
and we obtain a new proof of Mikhalkin's tropical correspondence theorem for Welschinger invariants. 
\end{abstract}
\maketitle

\setcounter{tocdepth}{2}
\tableofcontents

\section{Introduction}
\label{sec:intro}

\subsection{Overview}
Tropical geometry provides a combinatorial approach to
complex and real enumerative geometry. Using a version of Viro's patchworking \cite{V1,V2,V3}, 
Mikhalkin \cite[Theorem 1]{Mi} proved a correspondence theorem 
between counts of complex curves in toric surfaces and counts of tropical curves in $\R^2$. 
In \cite{NS}, Nishinou-Siebert used toric degenerations and log deformation theory to prove 
a correspondence theorem between counts of genus $0$ complex curves in $n$-dimensional toric varieties and counts of tropical curves in $\R^n$. 
Under the assumption of non-superabundacy, this result was generalized by Nishinou \cite{nishinou2009correspondence} to higher-genus complex curves. 
Further works on the tropical correspondence theorem for complex curves in toric varieties include \cite{Tyo1, GroA, Ran, mandel2020descendant}.

In the present paper, we focus on counts of real curves in toric varieties. We start by recalling the main difference between complex and real algebraic geometry. 
In complex geometry, a space of configurations in general position is typically connected, as configurations 
which are not in general position form loci of complex codimension at least one and so of real codimension at least two. 
For example, counts of complex curves matching incidence conditions in general position are independent of the particular choice of general incidence conditions. 
By contrast, in real geometry, a space of configurations in general position is typically disconnected, 
as configurations which are not in general position typically form loci of real codimension one. 
It follows that counts of real curves matching incidence conditions in general positions depend in general of the particular choice of general incidence condition. 

Therefore, a correspondence theorem between counts of tropical curves and counts of real curves in a toric variety matching 
some real incidence conditions can only hold in general for a specific class of incidence conditions, those which are in some sense close to some ``tropical limit". 
In \cite[Theorem 3]{Mi}, Mikhalkin proved such a result for real curves in toric surfaces passing through configurations of real points. 
Our main result, Theorem \ref{main_theorem_intro}, generalizes this result to genus $0$ and higher-genus non-superabundant real curves in higher dimensional toric varieties, 
in complete analogy with how Nishinou-Siebert \cite{NS} generalized Mikhalkin's correspondence theorem \cite[Theorem 1]{Mi} for complex curves. 
More precisely, we consider a toric degeneration of a toric variety and of a set of incidence conditions.
Theorem \ref{main_theorem_intro}
describes how to compute tropically counts of real curves matching the incidence conditions in a fiber of the toric degeneration which is sufficiently close to the central fiber. 
It is in this precise sense that we are restricting ourselves to incidence conditions close to a ``tropical limit".

In \cite[Theorem 6]{Mi}, Mikhalkin also proved a tropical correspondence theorem for real curves in toric surfaces counted with Welschinger signs, 
that is with the sign $(-1)^s$ where $s$ is the number of real elliptic nodes of the curve. 
Using the techniques developed to prove Theorem \ref{main_theorem_intro}, 
we will give a new proof of this result, see Theorem \ref{Thm: log Welschinger equals tropical intro}.
In restriction to the case of rational curves in del Pezzo toric surfaces, counts with Welschinger signs are Welschinger invariants \cite{We,We1} and have the remarkable property of being invariant with respect to deformation of incidence conditions. The tropical correspondence theorem \cite[Theorem 6]{Mi}
has been a central tool in the study of Welschinger invariants
\cite{BM1, BM2, IKS03, IKS04, IKS09}.

\subsection{Main results}
We now state precisely our first main result, Theorem \ref{main_theorem_intro}, referring to the main body of the paper for more details on the objects involved. 

Let $M=\Z^n$ and $M_{\RR}=M \otimes_{\Z} \RR$.
Let $X$ be a $n$-dimensional proper toric variety over $\C$, defined by a complete fan $\Sigma$ in $M_{\RR}$. 
Remark that every toric variety over $\C$ is naturally defined over $\Z$ and so in particular over $\R$.
We fix a tuple 
$(g,\Delta,\mathbf{A},\mathbf{P})$ where:
\begin{itemize}
    \item $g$ is a nonnegative integer.
    \item $\Delta$ is a map $\Delta:M\setminus\{0\}\to \NN$ with support contained in the union of rays of $\Sigma$. 
    The choice of $\Delta$ specifies a degree and tangency conditions along the toric divisors for a curve in $X$, and a tropical degree for a tropical curve in $M_{\RR}$.
    We denote by $|\Delta|$ the number of $v \in M\setminus \{0\}$
    with $\Delta(v) \neq 0$, that is the number of contact points with the toric boundary divisor (union of toric divisors) for a corresponding algebraic curve, 
    or the number of unbounded edges for a corresponding tropical curve.
    \item  $\mathbf{A}=(A_1,\ldots,A_\ell)$ is a tuple of affine linear subspaces $A_j$ of $M_{\Q}=M \otimes_{\Z} \Q$ with codimension $d_j+1$ such that 
    \begin{equation}\label{Eq:vdim}
      \sum_{j=1}^{\ell} d_j=(n-3)(1-g)+|\Delta|\,.
    \end{equation}
    We denote by $L(A_j) \subset M_\Q$ the linear direction of $A_j$.
    \item  $\mathbf{P}=(P_1,\ldots, P_\ell)$ is a tuple of real points in the $n$-dimensional torus of $X$.
\end{itemize}

Let $\T_{g,\ell,\Delta}(\mathbf A)$ be the set of 
$\ell$-marked tropical curves $h \colon \Gamma \rightarrow M_{\RR}$ of genus $g$
and degree $\Delta$ and matching the tropical incidence conditions $\mathbf{A}$.
We assume that we are in a non-superabundant situation, that is that the set $\T_{g,\ell,\Delta}(\mathbf{A})$ 
consists of finitely many tropical curves $h \colon \Gamma \rightarrow M_{\RR}$, all
with $\Gamma$ trivalent, and which are all non-superabundant, that is having a space of deformations of the expected dimension
$(n-3)(1-g)+|\Delta|$. Such non-superabundant condition is automatic if $g=0$ 
by \cite[Proposition 2.4]{NS} or if $n=2$ by
\cite[Proposition 4.11]{Mi}. 

As in \cite{NS}, we construct a toric degeneration $\pi \colon \cX \rightarrow \A^1$ from a polyhedral decomposition of $M_{\RR}$ containing the images of all the tropical curves in $\T_{g,\ell,\Delta}(\mathbf A)$ and all their intersection points with the tropical constraints 
$\mathbf{A}$. For $t \in \A^1 \setminus \{0\}$, we have $X_t \coloneqq \pi^{-1}(t) \simeq X$. 
The central fiber $X_0 \coloneqq \pi^{-1}(0)$ is a union of toric varieties glued along their toric divisors. 
Using $\mathbf{A}$ and $\mathbf{P}$, we construct a family of real incidence conditions $\mathcal{Z}_{A_j,P_j} \subset \mathcal{X}$. 
For every $t \in \A^1$, we denote $Z_{A_j,P_j,t} \coloneqq \mathcal{Z}_{A_j,P_j} \cap X_t \subset X_t$
(see \eqref{Eq:tilde_Z}-\eqref{Eq:tilde_Z_0} for details).

For every $t \in \A^1(\R)\setminus \{0\} \simeq \R^{\times}$, let
$M_{(g,\Delta, \mathbf{A},\mathbf{P}),t}^{\RR-log}$
be the set of genus $g$ real stable map to $X_t$
with degree and tangency conditions along the toric divisors
prescribed by $\Delta$ and matching the incidence conditions $\mathbf{Z}_{\mathbf{A},\mathbf{P},t}$.
We denote 
\[ N_{(g,\Delta, \mathbf{A},\mathbf{P}),t}^{\RR-log}
\coloneqq \sharp M_{(g,\Delta, \mathbf{A},\mathbf{P}),t}^{\RR-log}\,.\]
In our main result, Theorem \ref{main_theorem_intro}, we explain how to tropically compute 
$N_{(g,\Delta, \mathbf{A},\mathbf{P}),t}^{\RR-log}$ for $t$ sufficiently close to $0$.

We define the real count of tropical curves by an explicit formula
\begin{equation}
    \label{Eq: the tropical count intro}
     N_{(g,\Delta,\mathbf{A},\mathbf{P})}^{\R-trop} \coloneqq
\sum_{(\Gamma,\mathbf E,h)\in \T_{g,\ell,\Delta}(\mathbf A)}
w^{\RR}(\Gamma,\mathbf{E})\cdot\mathcal{D}^{\RR}_{\mathcal{T}_h,\sigma} \cdot \prod_{j=1}^\ell \mathcal{D}^{\RR}_{\mathcal{A}_j} \,.
\end{equation}
We refer to the main body of the paper for the definition of the various factors: $w^{\RR}(\Gamma,\mathbf{E})$ is the total real weight defined in \eqref{Eq: Total real weight},  
$\mathcal{D}^{\RR}_{\mathcal{T},\sigma}$ is the twisted real lattice index of the map $\mathcal{T}_h$ 
given by \eqref{eq:NS map}, and $\mathcal{D}^{\RR}_{\mathcal{A}_i}$ is the real lattice index of the inclusion of lattices $\mathcal{A}_i$ defined in 
\eqref{Eq: lattice marked points}, the real lattice index being defined in Definition \ref{Def: real index} 
and the twisted real lattice index being defined in Theorem \ref{Thm: counts of prelog curves}.

\begin{theorem}[=Theorem \ref{main theorem}]\label{main_theorem_intro}
For every $(g,\Delta, \mathbf{A}, \mathbf{P})$ as above, we have 
\[ 
N_{(g,\Delta,\mathbf{A},\mathbf{P}),t}^\text{$\RR$-log}
=N_{(g,\Delta,\mathbf{A},\mathbf{P})}^\text{$\RR$-trop} \]
for all $t \in \A^1(\R) \setminus \{0\} \simeq \R^{\times}$ sufficiently close to $0$.
\end{theorem}

The proof of Theorem \ref{main_theorem_intro} follows the 
lines of the proof of the main result of \cite{NS}. In a first step, given a tropical curve $(\Gamma,\mathbf{E},h)$
in  $\T_{g,\ell,\Delta}(\mathbf A)$, we count how many ways there are to lift $(\Gamma,\mathbf{E},h)$ into a maximally degenerate real stable map to the central fiber $X_0$. In 
\cite{NS}, the count of such maximally degenerate complex stable maps is expressed as the index of a map of lattices 
$\mathcal{T}_h \colon M_1 \rightarrow M_2$. 
The lattice index, that is the cardinality of $\mathrm{Coker}(\mathcal{T}_h)$, is also the cardinality of the kernel of the map of complex tori $\mathcal{T}_h \otimes \C^{\times}
\colon M_1 \otimes_{\Z} \C^{\times} \rightarrow M_2 \otimes_{\Z} \C^{\times}$. 
In the proof of \cite{NS}, the map $\mathcal{T}_h \otimes \C^{\times}$ measures the obstruction to construct 
a maximally degenerate stable map by gluing of its irreducible components living inside the various toric components of $X_0$. 
So the kernel of $\mathcal{T}_h \otimes \C^{\times}$ is exactly in bijection with the relevant complex maximally degenerate stable maps. 
In the real case, we show similarly that the set of relevant maximally degenerate real stable maps can be identified with the kernel of the map of real tori 
$\mathcal{T}_h \otimes \R^{\times}
\colon M_1 \otimes_{\Z} \R^{\times} \rightarrow M_2 \otimes_{\Z} \R^{\times}$. 
By some elementary homological algebra (see Lemma \ref{Lem: kernel}), 
the cardinality of the kernel of  $\mathcal{T}_h \otimes \R^{\times}$ can be computed in terms of a cyclic decomposition of the group $\mathrm{Coker}(\mathcal{T}_h)$. 
Modulo the issue that  $\mathcal{T}_h \otimes \R^{\times}$ is not necessarily surjective (see Theorem \ref{Thm: counts of prelog curves} for details), 
this gives the definition of the factor $\mathcal{D}^{\RR}_{\mathcal{T}_h,\sigma}$ in 
\eqref{Eq: the tropical count intro}.

In a second step, we need to count how many ways there are to lift a maximally degenerate real stable map to $X_0$
into a real stable log map to $X_0$.
This is done in Theorem 
\ref{Thm: lifting md curves to log maps} and this produces the factor $w^{\RR}(\Gamma,\mathbf{E})$ in 
\eqref{Eq: the tropical count intro}. 
Finally, the counting of real stable log maps to the central fiber $X_0$ agrees with the counting of real stable log maps 
on a neighbour fiber $X_t$ by log smooth deformation theory, as in 
\cite{NS}.

We remark that for $n=2$, Mikhalkin 
\cite[Theorem 3]{Mi} gives an alternative recursive description of 
$N_{(g,\Delta,\mathbf{A},\mathbf{P})}^{\RR-trop}$. This is similar to how Mikhalkin
\cite[Theorem 1]{Mi} expresses the complex multiplicity of a tropical curve as a product of local multiplicities attached to vertices, 
whereas the complex multiplicity of 
Nishinou-Siebert \cite{NS} is a priori a globally defined lattice index. 
We also remark that, for $g=0$, a version of Theorem \ref{main_theorem_intro} is proved by Tyomkin in \cite[Section 5.1]{Tyo2}. 
Tyomkin's proof uses explicit rational parametrizations of rational curves in toric varieties and so is fundamentally limited to the $g=0$ case. 
Thus, one can view Theorem \ref{main_theorem_intro} as providing a common generalization of \cite[Theorem 3]{Mi}
and \cite[Section 5.1]{Tyo2}. Finally, we point out that, while we prove Theorem \ref{main_theorem_intro} via a 
real version of the log arguments of Nishinou-Siebert 
\cite{NS}, it should also be possible to write a proof using a real version of the stacky approach to the correspondence theorem due to Tyomkin
in \cite{Tyo1}.


Our second main result, Theorem \ref{Thm: log Welschinger equals tropical intro}, is a tropical correspondence theorem for counts of real curves with 
Welschinger signs in toric surfaces, recovering with a different proof 
\cite[Theorem 6]{Mi} (see also \cite{Sh0}).
We stay in the setup of Theorem \ref{main_theorem_intro}
but specialized to $n=2$ and to the case of $0$-dimensional affine constraints $\mathbf{A}$.

For $t \in \A^1(\R)\setminus \{0\} \simeq \R^{\times}$
general, all the singularities of the image 
$\varphi(C)$ of a real stable map 
$\varphi \colon C \rightarrow X_t$ defining an element of 
$M_{(g,\Delta, \mathbf{A},\mathbf{P}),t}^{\RR-log}$ are nodes, that is ordinary double points.
The Welschinger sign of $\varphi$ is defined by 
\[ \mathcal{W}^{log}(\varphi) \coloneqq (-1)^{m(\varphi)}\]
where $m(\varphi)$ is the number of real elliptic nodes of 
$\varphi(C)$. We define log Welschinger numbers by 
\[ \mathcal{W}^{\RR-log}_{(g,\Delta,\mathbf{A},\mathbf{P}),t} \coloneqq \sum_{ \varphi \in M_{(g,\Delta, \mathbf{A},\mathbf{P}),t}^{\RR-log}}\mathcal{W}^{log}(\varphi)\,.\]
In Theorem \ref{Thm: log Welschinger equals tropical intro}, we explain how to tropically compute 
$\mathcal{W}^{\RR-log}_{(g,\Delta,\mathbf{A},\mathbf{P}),t}$ for $t$ sufficiently close to $0$.

Let $h\colon \Gamma \to M_{\R}$ be a tropical curve defining an element of $\mathcal{T}_{g,\ell,\Delta}(\mathbf{A})$. 
For $V$ a vertex of $\Gamma$, let $\Delta_V \subset N_{\RR}$ denote the dual triangle to $V$: for every edge $E_i$ of 
$\Gamma$ adjacent to $V$, the corresponding side of $\Delta_V$ has integral length equal to the weight 
$w(E_i)$ of $E_i$. We set \[\mathrm{Mult}_\RR(V)\coloneqq (-1)^{I_{\Delta_V}} \,,\] 
where 
$I_{\Delta_V}$ is the number of integral points in the interior of $\Delta_V$.
We define the multiplicity of a tropical curve by 
\begin{equation}
\mathrm{Mult}_{\RR}(h) \coloneqq \begin{cases} 
      0 & \mathrm{if}~\Gamma~\mathrm{contains~a~bounded~edge~of~even~weight} \\
      \prod_V \mathrm{Mult}_{\R}(V)  & \mathrm{else}.
   \end{cases}
\end{equation}
Finally, we define 
\[\mathcal{W}^{\RR-trop}_{(g,\Delta,\mathbf{A},\mathbf{P})} \coloneqq \sum_{(\Gamma,\mathbf{E},h) \in 
\mathcal{T}_{g,\ell,\Delta}(\mathbf{A})} \mathrm{Mult}_\RR(h)\,. \]

\begin{theorem}[=Theorem \ref{Thm: log Welschinger equals tropical}]
\label{Thm: log Welschinger equals tropical intro}
For all $t \in \A^1(\R) \setminus \{0\} 
\simeq \R^{\times}$ sufficiently close to $0$, we have
\begin{equation}
\mathcal{W}^{\RR-log}_{(g,\Delta,\mathbf{A},\mathbf{P}),t}
=
\mathcal{W}^{\RR-trop}_{(g,\Delta,\mathbf{A},\mathbf{P})}\,.
\end{equation}
\end{theorem}

Theorem \ref{Thm: log Welschinger equals tropical intro} recovers \cite[Theorem 6]{Mi}, 
proved by Mikhalkin using degeneration of the complex structure and a version of Viro's patchworking. 
A purely algebraic proof is also given by Shustin in \cite{Sh0} (see also the exposition given in \cite{IMS}). 
Our proof of Theorem 
\ref{Thm: log Welschinger equals tropical intro} relies on the log geometric framework 
used in the proof of Theorem \ref{main_theorem_intro}, and as such differs from the proofs given by \cite{Mi} and \cite{Sh0}. 
Nevertheless, we will use one of the key idea of \cite{Sh0}, described in \cite{Sh0} as a refinement of the tropicalization by a shift operation, and reinterpreted in our language as a non-toric blow-up of the central fiber $X_0$ (see also \cite{ShT1, ShT2}). In deforming a real stable log map
$\varphi_0 \colon C_0 \rightarrow X_0$ to the central fiber $X_0$ into a real stable log map $\varphi \colon C \rightarrow X_t$, the nodes of $\varphi_0(C_0)$ away
from the double locus of $X_0$ deform into locally isomorphic nodes of $\varphi(C)$, 
but the nodes of $\varphi_0(C_0)$ on the double locus of $X_0$ deform in general into several nodes of $\varphi(C)$, see Figure \ref{Fig: Nodes}. 
We perform a non-toric blow-up of $X_0$ in order to study the real nature of these nodes.

\begin{figure}
\center{\scalebox{1.0}{\input{Nodes.pspdftex}}}
\caption{A maximally degenerate real log curve $\varphi_0 \colon C_0 \to X_0$ and its 
deformation $\varphi \colon C \to X$ with some new nodes generated in the image.}
\label{Fig: Nodes}
\end{figure}

For $X$ a toric del Pezzo surface and $d \in H_2(X,\Z)$ such that $c_1(X) \cdot d -1 >0$, let 
$\Delta_d \colon M \setminus \{0\} \rightarrow \NN$ be the 
tropical degree defined by $\Delta_d(v)=d \cdot D_v$ 
if $v$ is the primitive generator of the ray of the fan of $X$
corresponding to the toric divisor $D_v$, and 
$\Delta_d(v)=0$ else. Then, for $g=0$ and 
$\Delta=\Delta_d$, the log Welschinger number $\mathcal{W}^{\R-log}_{(0,\Delta_d,\mathbf{A},\mathbf{P}),t}$
agrees with the Welschinger invariant defined symplectically in 
\cite{We,We1}, see Corollary \ref{Cor:W_tropical}. 
In particular, it follows from \cite{We,We1} that $\mathcal{W}^{\R-log}_{(0,\Delta_d,\mathbf{A},\mathbf{P}),t}$ is independent of $t$ in this case. 

\subsection{Future directions}
In \cite{NS}, Nishinou-Siebert stressed in the complex case the robustness of the log geometric framework. 
This has been remarkably confirmed by the development of 
log Gromov-Witten theory \cite{AbramovichChen, GSlogGW}.
In the present paper, we make a first step in the study of real stable log maps.
We mention briefly possible future directions to test the robustness of 
the real log geometric framework.

We are considering ``purely real" constraints: our incidence conditions 
 are isomorphic to toric varieties with their standard 
real structures. It would be interesting to study pairs of complex conjugated constraints. For real toric del Pezzo surfaces, Shustin \cite{Sh}
(see also \cite{GMS}) has given a tropical description of Welschinger invariants for real curves passing through arbitrary real configuration of points 
(real and complex conjugated pairs).

One should be able to apply degeneration techniques to non-toric situations.
A natural plan is to start with the study of non-toric log Calabi-Yau surfaces: 
we expect the formulation of a real version of the tropical vertex of Gross-Pandharipande-Siebert \cite{GPS}, compatible with its $q$-deformed version 
\cite{FS, bousseau2018quantum}, and with possible applications to a real version of mirror symmetry and to real Gromov-Witten invariants 
(see for example 
\cite[Conjecture 6.2.1]{bousseau2019proof}).

Finally, a natural question is the tropical interpretation of higher-dimensional versions of Welschinger invariants. 
Welschinger \cite{We2, Wespinor} defined 3-dimensional invariants, later generalized by Solomon \cite{So, ST} in the setting of open Gromov-Witten theory.
For $\PP^3$, an interpretation in terms of floor diagrams appears in 
\cite{BM1}. We also refer to Georgieva--Zinger for positive-genus analogues of Welschinger’s invariants for several real symplectic manifolds, such as the odd-dimensional projective spaces and the quintic threefold \cite{GZ,GZ2,GZ3}. It is natural to ask if real log geometric techniques can shed some light on these invariants in a more general setup. 
This will be the focus of future work.

\subsection{Plan of the paper}
In \S\ref{Sec: toric and tropical}, we review the construction of toric degenerations and basic facts on tropical curves in $\R^n$.
In \S\ref{Sec:real_stable_log}, we introduce the notion of real stable log map.
In \S\ref{Sec. counts of md curves}, 
we prove Theorem \ref{Thm: counts of prelog curves} counting ways to lift a tropical curve to a maximally degenerate real stable map to the central fiber $X_0$.
In \S\ref{Sec: from real md to real log}, we prove Theorem 
\ref{Thm: lifting md curves to log maps} counting ways to lift a 
maximally degenerate real stable to $X_0$
to a real stable log map. In \S\ref{sec:defo}, we use log deformation theory 
to study deformation of a real stable log map to $X_0$
in a real stable log map to a neighbour fiber $X_t$.
In \S\ref{sec:The Main Theorem}, we prove our main result, 
Theorem \ref{main theorem} (=Theorem \ref{main_theorem_intro})
by combining the results proved in 
\S\ref{Sec. counts of md curves}-\S\ref{Sec: from real md to real log}-\S\ref{sec:defo}.
In \S\ref{Sec: tropical W signs}, we discuss the counting of tropical curves
in $\R^2$ with 
Welschinger signs. 
In \S\ref{Sec: Log W signs}, we prove Theorem \ref{Thm: log Welschinger equals tropical} (=Theorem \ref{Thm: log Welschinger equals tropical intro})
computing tropically the counts of real stable log maps in toric surfaces with 
Welschinger signs. We conclude \S\ref{Sec: Log W signs} by a discussion of the relation with the symplectically defined Welschinger invariants.

\subsection{Acknowledgements}
We thank Mark Gross, Bernd Siebert, Eugenii Shustin, Ilya Tyomkin, Tom Coates, Ilia Itenberg and Dimitri Zvonkine for useful discussions. 
A major part of this paper was written during our visit to Kyoto University for the 5'th KTGU Mathematics Workshop for Young Researchers. 
We thank Hiroshi Iritani and all the other organisers for their hospitality.
Finally, we thank the anonymous referee for their careful reading and the many suggestions to improve the exposition.

During the preparation of this paper the first author has received funding
from the European Research Council (ERC) under the European Union’s Horizon 2020
research and innovation programme (grant agreement No. 682603), and from Fondation Math\'ematique Jacques Hadamard. 
The second author acknowledges the support of Dr. Max R\"ossler, the Walter Haefner Foundation and the ETH Zurich Foundation. 

\section{Tropical curves and toric degenerations}
\label{Sec: toric and tropical}

In this section, we review basic facts about tropical curves and toric degenerations, mainly following \cite[\S 1-3]{NS}.

\subsection{Tropical curves in $\R^n$}
\label{sec: Real Tropical Curves}

Throughout this paper we fix the lattice $M=\ZZ^n$ and we denote by
$M_{\RR}=M\otimes_{\ZZ}\RR$ 
the associated $n$-dimensional real vector space. 

Let $\overline{\Gamma}$ be a weighted, connected finite graph without divalent vertices. 
Denote the set of vertices and
edges of $\overline{\Gamma}$ by $\overline{\Gamma}^{[0]}$ and $\overline{\Gamma}^{[1]}$
respectively, and let
$$w_{\overline{\Gamma}}:\overline{\Gamma}^{[1]}
\rightarrow \NN\setminus \{0\}$$ be the weight function. 
Denote the set of adjacent vertices to an edge
$E\in \overline{\Gamma}^{[1]}$ by $\partial E=\{V_1,V_2\}$. 
The set of 1-valent vertices is denoted by $\overline{\Gamma}_{\infty}^{[0]}\subseteq \overline{\Gamma}^{[0]}$. We set
\[
\Gamma=\overline{\Gamma}\setminus\overline{\Gamma}_{\infty}^{[0]}
\]
We denote the set of vertices and edges of $\Gamma$ as $\Gamma^{[0]}$,
$\Gamma^{[1]}$, and let $$w_{\Gamma}:\Gamma^{[1]}
\rightarrow \NN\setminus\{0\}$$
be the restriction of the weight function to $\Gamma^{[1]}$. We call the non-compact edges of $\Gamma$ \emph{unbounded edges} 
and denote the set of unbounded edges by $\Gamma_{\infty}^{[1]}\subseteq
\Gamma^{[1]}$. We call the compact edges of $\Gamma$ \emph{bounded edges}. 
The set of bounded edges of $\Gamma$ is $\Gamma^{[1]} \setminus \Gamma^{[1]}_{\infty}$.

\begin{definition}
\label{tropcurve}
A \emph{parameterized tropical curve} in $M_{\RR}$ is a proper map
$h \colon \Gamma\rightarrow M_{\RR}$ satisfying the following conditions.
\begin{enumerate}
\item For every edge $E\subseteq\Gamma^{[1]}$, the restriction $h|_E$ is
an embedding with image $h(E)$ contained in an affine line with rational
slope.
\item For every vertex $V\in\Gamma^{[0]}$, the following \emph{balancing
condition} holds. Let $$E_1,\ldots,E_m\in \Gamma^{[1]}$$
 be the edges
adjacent to $V$, and let $m_i\in M$ be the primitive integral vector
emanating from $h(V)$ in the direction of $h(E_i)$. Then
\begin{equation}\label{Eq:balancing_condition}
\sum_{j=1}^m w_\Gamma(E_j)m_j=0.
\end{equation}
\end{enumerate}
An \emph{isomorphism} of parametrized
tropical curves $h \colon \Gamma\rightarrow N_{\RR}$
and $h' \colon \Gamma'\rightarrow M_{\RR}$ is a homeomorphism $\Phi:\Gamma
\rightarrow\Gamma'$ respecting the weights of the edges and such that
$h=h'\circ\Phi$. A \emph{tropical curve} is an isomorphism class
of parameterized tropical curves. The \emph{genus} of a tropical curve
$h \colon \Gamma\rightarrow M_{\RR}$ is the first Betti number of $\Gamma$.
\end{definition}

Let 
\[F(\Gamma)=\{(V,E)\,|\, E\in \Gamma^{[1]}~\mathrm{and}~V\in\partial E\}\] 
be the set of \emph{flags} of $\Gamma$. Let
\begin{align*}
u \colon F(\Gamma)~ & \lra M \\
F(\Gamma) \ni (V,E)  & \longmapsto  u_{V,E}
\end{align*}
be the map sending a flag
$(V,E)$ to the primitive integral vector $u_{(V,E)} \in M$ emanating from $h(V)$ in the direction of $h(E)$.

\begin{definition}
\label{Daef: type of a tropical curve}
The \emph{type of a tropical curve} $h \colon \Gamma \to M_\RR$ is the pair $(\Gamma, u).$
\end{definition}

\begin{definition}
\label{Def: degree}
The \emph{degree} of a type $(\Gamma,u)$ of tropical curves is the map 
\[ \Delta(\Gamma,u) \colon M\setminus \{0\}\to \NN\]
defined by
\[
\Delta(\Gamma,u)(v)
\coloneqq \sharp\big\{ (V,E)\in F(\Gamma)
\,\big|\, E \in \Gamma_\infty^{[1]}\,, \, w(E)\cdot u_{(V,E)}=v\big\}.
\]
The degree of
a tropical curve is the degree of its type. 
\end{definition}

\begin{figure} 
\center{\scalebox{1.0}{\input{TwoConics.pspdftex}}}
\caption{Two tropical curves $h_1: \Gamma \to \RR^2$ and $h_2: \Gamma \to \RR^2$ having different types but the same degree.}
\label{Fig: TwoConics}
\end{figure}

For every $\Delta \colon M \setminus \{0\} \rightarrow \NN$ with finite support, 
we define
\begin{equation}
\label{Eq: cardinality Delta}    
    |\Delta| \coloneqq \sum_{v\in M\setminus\{0\}} \Delta(v).
\end{equation}
Note that a tropical curve of
degree $\Delta$ has $|\Delta|$ unbounded edges.

\begin{definition}
\label{Def:l-marked}
An \emph{$\ell$-marked tropical curve}
$(\Gamma, \mathbf{E}, h)$ consists of a tropical curve 
$h \colon \Gamma \rightarrow M_{\R}$ together with a choice of 
$\ell$ edges
\[ \mathbf{E} = (E_1, .\ldots , E_{\ell}) \subset { \big( 
\Gamma^{[1]}\big)}^{\ell} \,.\]
\end{definition}

For a given type $(\Gamma, u)$ of degree $\Delta$, we denote by $\T_{(\Gamma,u)}$ the set of
isomorphism classes of marked tropical curves of type $(\Gamma,u)$. If $\Gamma$ trivalent of genus $g$, 
then, by \cite[Proposition 2.13]{Mi} (see also
\cite[Proposition 1.17]{Gross}), $\T_{(\Gamma,u)}$ is an open convex polyhedral domain in a real affine space of dimension 
greater or equal to
$(n-3)(1-g)+|\Delta|$.

\begin{example}
For $(\Gamma,u)$ the type of either one of the two tropical curves in Figure \ref{Fig: TwoConics}, we have 
$\T_{(\Gamma,u)}=\R^2 \times (\R_{>0})^3$, where the first factor $\R^2$ is the position of one of the vertices and the three factors $\R_{>0}$ are the lengths of the three bounded edges. Note that in this example, we have $n=2$, 
$g=0$, $|\Delta|=6$, so $(n-3)(1-g)+|\Delta|=5$.
\end{example}

\begin{definition}
\label{Def:non-superabundant}
Let $(\Gamma,u)$ be a type of tropical curves of degree $\Delta$ with 
$\Gamma$ trivalent and of genus $g$. We say that $(\Gamma,u)$ is \emph{non-superabundant} if 
\[ \dim \T_{(\Gamma,u)}=(n-3)(1-g)+|\Delta| \,.\]
A (marked) tropical curve is non-superabundant if its type is non-superabundant.
\end{definition}

\begin{example}
A tropical curve $h: \Gamma \to M_{\RR}$ is superabundant if some of the cycles of $\Gamma$ map into smaller-dimensional affine-linear subspaces of $M_{\RR}$, as illustrated in Figure \ref{Fig: Nonsuper}. For a more comprehensive discussion on super-abundancy we refer to \cite[\S 2.6]{Mi}.

\begin{figure} 
\center\scalebox{0.8}{\input{Nonsuper.pspdftex}}
\caption{Image of a non-superabundant tropical curve $h: \Gamma \to \RR^3$.}
\label{Fig: Nonsuper}
\end{figure}
\end{example}

\begin{definition}
\label{Def: constraints}
An \emph{affine
constraint} is an $\ell$-tuple
$\mathbf{A}=(A_1,\ldots, A_\ell)$ of affine subspaces $A_j\subset M_\QQ$. 
An $\ell$-marked tropical curve
$(\Gamma,\mathbf{E}, h)$ \emph{matches} the affine constraint
$\mathbf{A}$ if 
\[  h(E_j)\cap A_j \neq \emptyset \quad \text{for all} \quad j=1,\ldots, \ell\,. \]
\end{definition}

\begin{example}
In Figure \ref{Fig: Marked curve} we illustrate a marked tropical curve 
$(\Gamma,\mathbf{E},h)$, with $\mathbf{E}=(E_1,\ldots,E_5)$, matching an affine constraint given by $5$ points $(p_1,\ldots,p_5)$. 

\begin{figure} 
\center{\input{conic.pspdftex}}
\caption{A marked tropical curve matching $5$ points in $\RR^2$.}
\label{Fig: Marked curve}
\end{figure}
\end{example}

\begin{definition}
\label{Def: tropical spaces}
For $g \in \NN$ and $\Delta \colon M \setminus \{0\} \rightarrow \NN$, 
the set of $\ell$-marked tropical curves of genus $g$ and degree
$\Delta$
is denoted $\T_{g,\ell,\Delta}$.
We denote by $\T_{g,\ell,\Delta}(\mathbf{A})$
the subset of tropical curves matching an affine constraint 
$\mathbf{A}$.
\end{definition}

\begin{definition}
\label{Def: general affine}
Let $g \in \NN$ and $\Delta \colon M \setminus \{0\} \rightarrow \NN$ with 
finite support. 
An affine constraint $\mathbf{A}=(A_1,\ldots,A_{\ell})$ is \emph{general for
$(g,\Delta)$} if:
\begin{itemize}
\item writing $\mathrm{codim} A_j=d_j+1$, we have 
\[ \sum_{j=1}^{\ell} d_j=(n-3) (1-g) +|\Delta|\,; \]
\item no translation of $M_{\RR}$ preserves $\bigcup_{j=1}^{\ell} A_j$, that is, there does not exist 
$v \in M_\RR$ such that $\bigcup_{j=1}^{\ell} A_j + \R v 
\subset \bigcup_{j=1}^{\ell} A_j$;
\item for any $\ell$-marked tropical curve 
$(\Gamma, \mathbf{E},h)$ of genus $g$, degree $\Delta$, and matching 
$\mathbf{A}$,
the following hold:
\begin{itemize}
\item  $\Gamma$ is trivalent;
\item $(\Gamma, \mathbf{E},h)$ is non-superabundant;
\item  $h(\Gamma^{[0]}) \cap \bigcup_{j=1}^{\ell} A_j=\emptyset$;
\item  $h$ is injective for $n > 2$. For $n=2$, it is at least injective on the subset of
vertices, and all fibers are finite.
\end{itemize}
\end{itemize}
\end{definition}

\begin{proposition}(\cite[Proposition 4.11]{Mi},\cite[Proposition 2.4]{NS})
Let $g \in \NN$, $\Delta \colon M \setminus \{0\} \rightarrow \NN$ with 
finite support, and $\mathbf{A}=(A_1,\ldots,A_{\ell})$ an affine constraint such that, writing $\mathrm{codim} A_j=d_j+1$, 
we have 
\[ \sum_{j=1}^{\ell} d_j=(n-3) (1-g) +|\Delta|\,. \]
Denote by $\mathfrak{A} \coloneqq \prod_{j=1}^{\ell} M_{\Q}/L(A_j)$ the
space of affine constraints that are parallel to $\mathbf{A}$, where $L(A_j) \subset M_\Q$ is the linear subspace associated to the affine subspace $A_j$. 
Assume that $n=2$ or $g=0$.
Then the subset
\[ \mathfrak{Z} \coloneqq \{ \mathbf{A}' \in \mathfrak{A}\,|\,
\mathbf{A}' \,\text{is nongeneral for} \,(g,\Delta) 	\}\]
of $\mathfrak{A}$ is nowhere dense.
\end{proposition}

\begin{proof}
For $n=2$, this is Proposition 4.11 of \cite{Mi}. For $g=0$, this is 
Proposition 2.4 of \cite{NS}.
\end{proof}

\begin{proposition}\label{prop_finite_tropical}
Let $g \in \NN$, $\Delta \colon M \setminus \{0\} \rightarrow \NN$
with finite support, and $\mathbf{A}$ an affine constraint general for 
$(g,\Delta)$. Then the set 
$\T_{g,\ell,\Delta}(\mathbf{A})$ of $\ell$-marked tropical curves of 
genus $g$ and degree $\Delta$ matching $\mathbf{A}$ is finite.
\end{proposition}

\begin{proof}
By \cite[Proposition 2.1]{NS}, there are only finitely many types of
tropical curves of genus $g$ and degree $\Delta$. The set of marked tropical curves
of a given non-superabundant type is a convex polyhedron $P$ in 
$\R^{(n-3)(1-g)+|\Delta|}$. The set of  marked tropical curves of this type matching 
$\mathbf{A}$ is the intersection of $P$ with some affine subspace $Q$ of $\R^{(n-3)(1-g)+|\Delta|}$. 
If $Q$ intersects the boundary of $P$ in $\R^{(n-3)(1-g)+|\Delta|}$, 
then the tropical curve is not trivalent and this contradicts the assumption that $\mathbf{A}$ is general. 
Therefore, the intersection of $Q \cap P$ is entirely contained in $P$, and so is an affine subspace entirely contained in $P$. If $Q \cap P$ is not zero-dimensional, this is only possible 
if $Q \cap P$ contains a one-dimensional family of translated tropical curves, 
which again contradicts the assumption that $\mathbf{A}$ is general (more precisely that no translation preserves $\bigcup_{j=1}^{\ell} A_j$). Hence $Q \cap P$ is zero-dimensional, and so is either empty or consists of a single point.
\end{proof}

\subsection{Toric degenerations of toric varieties from polyhedral decompositions}
\label{Sec: toric degenerations}
In this section we review how to produce
from a polyhedral decomposition of $M_{\Q}$
a toric degeneration over $\C$. For details we refer to \cite[\S 3]{NS}.
\begin{definition}
A \emph{polyhedral decomposition} of $M_{\Q}$ is a covering $\mathscr{P} = \{ \Xi \}$
of $M_{\Q}$ by a finite number of strongly convex polyhedra satisfying the following two properties:
\begin{itemize}
    \item[i)] If $\Xi \in \mathscr{P} $ and $\Xi' \subset \Xi$ is a face, then $\Xi' \in \mathscr{P}$.
    \item[ii)] If $\Xi,\Xi' \in \mathscr{P} $, then $\Xi \cap \Xi'$ is a common face of $\Xi$ and $\Xi'$.
\end{itemize}

\end{definition}

Let $\mathscr{P}$ be a polyhedral decomposition $M_{\Q}$. 
We will denote by $\mathscr{P}^{[k]}$ the set of 
$k$-dimensional cells of $\mathscr{P}$.
For each $\Xi\in\mathscr{P}$, let
$C(\Xi)$ be the closure of the cone spanned by $\Xi \times
\{1\}$ in $M_\RR\times\RR$:
\begin{equation}
\label{cone over Xi}
C(\Xi) =  \overline{ \big\{a \cdot (n, 1)\,\big|\, a \ge 0, n\in
\Xi\big\} } \,.
\end{equation}
Note that taking the closure here will be important while talking about the asymptotic cone in cases $\Xi$ is unbounded. 
We use the convex polyhedral cones $C(\Xi)$ to define a fan presented by its faces as
\[
\widetilde{\Sigma}_\mathscr{P}\coloneqq \big\{ \sigma\subset C(\Xi) \text{
face}\,\big|\, \Xi\in \mathscr{P}\big\}
\]
which we refer to as the fan associated to $\mathscr{P}$.
By construction, the projection onto the second factor \[M_{\RR}\times\RR\to(M_{\RR}\times\RR)/_{M_\RR}=\RR\] 
defines a non-constant map of fans from the fan $\widetilde{\Sigma}_\mathscr{P}$ to the fan $\{ 0,\RR_{\geq 0}\}$ of $\AA^1$, 
hence a flat toric morphism 
\[
\pi \colon \mathcal{X} \longrightarrow \AA^1 \,,
\]
where $\mathcal{X}$ is the $(n+1)$-dimensional toric variety over $\C$ defined by the fan $\widetilde{\Sigma}_{\mathscr{P}}$. 
Throughout this paper we will say that
\emph{$\pi \colon \mathcal{X} \rightarrow \A^1$ is the toric degeneration 
obtained from the polyhedral decomposition $\mathscr{P}$ of $M_{\Q}$.}

Note that $\pi \colon \mathcal{X}\to \AA^1$ is a degeneration of toric varieties. 
To describe the general fiber of $\pi$, we first identify $M_\RR$ with $M_\RR\times\{0\}\subset M_\RR\times\RR$ and we define the fan
\[
\Sigma_\mathscr{P}= \big\{ \sigma \cap (M_\RR\times\{0\})\,\big|\,
\sigma\in\widetilde\Sigma_\mathscr{P}\big\}\,.
\]
By Lemma $3.3$ in \cite{NS}, $\Sigma_\mathscr{P}$ is the asymptotic fan of the polyhedral decomposition $\mathscr{P}$.
Let $X$ be the $n$-dimensional toric variety over $\C$ defined by the fan $\Sigma_\mathscr{P}$. By Lemma $3.4$ of \cite{NS} we have
\begin{equation}
\nonumber
\label{degeneration1}
\pi^{-1}(\AA^1\setminus \{0\}) = X \times (\AA^1\setminus \{0\})
\end{equation}
and the closed fibers of $\pi$ over $\AA^1\setminus \{ 0\}$ are all pairwise isomorphic to $X$.  

We now describe the central fiber 
\[X_0 \coloneqq \pi^{-1}(0)\,.\] 
For $\Xi\in\mathscr{P}$ the rays emanating from $\Xi$ through
adjacent $\Xi'\in\mathscr{P}$ define a fan $\Sigma_\Xi$ by
\begin{equation}
\label{fan}
\Sigma_\Xi=\big\{ \RR_{\ge 0}\cdot (\Xi'-\Xi)\subset M_\RR/L(\Xi)
\,\big|\, \Xi'\in\mathscr{P}, \Xi\subset\Xi' \big\}.
\end{equation}
in $M_\RR/L(\Xi)$, where $L(\Xi)\subset M_\RR$ is the linear subspace associated to $\Xi$. 
Let $X_\Xi$ be the toric variety over $\C$ associated to the fan $\Sigma_\Xi$. By Proposition $3.5$ in \cite{NS}, there exist closed
embeddings $X_\Xi\hookrightarrow X_0$, $\Xi\in\mathscr{P}$ compatible with morphisms $X_{\Xi}\to X_{\Xi'}$ for $\Xi'\subset \Xi$, 
inducing an isomorphism \[\displaystyle
X_0 \simeq \lim_{\displaystyle
\genfrac{}{}{0pt}{1}{\longrightarrow}{\Xi\in\mathscr{P}} } X_\Xi \,.\]
In particular, the central fiber $X_0$ is a union of toric varieties, glued along toric boundary divisors.

\begin{example} \label{ex_polyhedral_decomposition}
The polyhedral decomposition $\mathscr{P}$ illustrated in Figure \ref{Fig: P} defines a toric degeneration of the toric variety $\PP^2$, whose associated fan is the asymptotic fan $\Sigma_{\mathscr{P}}$, into the union of the $7$ copies of $\PP^2$ and a copy of $\PP^1\times \PP^1$. The neighbourhoods of vertices of $\mathscr{P}$, labelled by red circles in the figure, correspond to the fans of the $7$ irreducible components of the central fiber.

\begin{figure} 
\center\scalebox{0.8}{\input{P.pspdftex}}
\caption{A polyhedral decomposition $\mathscr{P}$ of $M_{\Q}$ and the associated asymptotic fan $\Sigma_{\mathscr{P}}$.}
\label{Fig: P}
\end{figure}
\end{example}

\subsection{Degeneration of incidence conditions}
In the remaining of the paper, we will study curves in toric degenerations 
$\pi \colon \mathcal{X} \rightarrow \A^1$
matching some incidence conditions. We now explain how to construct these incidence conditions.

Let $A$ be an affine subspace of $M_{\Q}$ and let $P$ be a point in the $(n+1)$-dimensional torus orbit of $\mathcal{X}$.
Denoting $LC(A)$ the linear closure of $A \times \{1\}$
in $M_{\Q} \times \Q$, we introduce the closure in
$\mathcal{X}$ of the orbit of the torus 
$(LC(A) \cap (M \times \Z)) \otimes_{\Z} \C^{\times}$ passing through the point $P$:
\begin{equation} 
\label{Eq:tilde_Z}
\mathcal{Z}_{A,P} \coloneqq \overline{((LC(A) \cap (M \times \Z)) \otimes_{\Z} \C^{\times})\cdot P}\subset \mathcal{X}\,.
\end{equation}
For every $t \in \A^1$, we define
\begin{equation}
\label{Eq:tilde_Z_0}
Z_{A,P,t} \coloneqq \mathcal{Z}_{A,P} \cap X_t \,,
\end{equation}
the intersection of $\mathcal{Z}_{A,P}$ with the fiber $X_t \coloneqq \pi^{-1}(t)$. In particular, 
$Z_{A,P,0}$ is a subvariety of the central fiber $X_0$.

\subsection{Good polyhedral decompositions}

\begin{definition}
\label{Def: good decomposition}
Let $g \in \NN$, 
$\Delta \colon M \setminus \{0\} \rightarrow \NN$ with finite support, and
$\mathbf{A} = (A_1,\ldots,A_\ell)$ an affine constraint that is general for $(g,\Delta)$. 
Let $\mathscr{P}$ be an integral polyhedral
decomposition of $M_\QQ$ such that, for every 
$(\Gamma, \mathbf{E}, h) \in \T_{g,\ell,\Delta}(\mathbf{A})$, we have the following.
\begin{itemize}
\item[(i)] The image $h(\Gamma)$ is contained in the one-skeleton of $\mathscr{P}$, that is, $h(\Gamma^{[0]})\subset
\bigcup_{\Xi\in\mathscr{P}^{[0]}} \Xi $ and $h(\Gamma^{[1]})\subset
\bigcup_{\Xi\in\mathscr{P}^{[1]}} \Xi $.
\item[(ii)] Intersection points of $h(\Gamma)$ with the constraints are vertices of 
$\mathscr{P}$: $h(\Gamma)\cap A_j\subset \mathscr{P}^{[0]}$ for every 
$j=1,\ldots,\ell$.
\item[(iii)] For every bounded edge $E \in \Gamma^{[1]}$, the weight $w(E)$ divides the integral affine length of $h(E)$. 
\end{itemize}
Then, we say $\mathscr{P}$ is a \emph{polyhedral decomposition good
for $(g,\Delta,\mathbf{A})$.}
\end{definition}

\begin{proposition} \label{prop:good_decomp}
Let $g \in \NN$, 
$\Delta \colon M \setminus \{0\} \rightarrow \NN$ with finite support, and
$\mathbf{A} = (A_1,\ldots,A_\ell)$ an affine constraint that is general for $(g,\Delta)$. 
Up to a rescaling of $M \subset M_{\Q}$, there exists a
polyhedral decomposition 
of $M_{\Q}$ which is good for $(g,\Delta, \mathbf{A})$.
\end{proposition}

\begin{proof}
By Proposition \ref{prop_finite_tropical}, the set 
$\T_{g,\ell,\Delta}(\mathbf{A})$ is finite. Therefore the result follows from
part (1) of the proof of Theorem 8.3 of \cite{NS} (which relies on 
\cite[Proposition 3.9]{NS}).
\end{proof}

\section{Real stable log maps}
\label{Sec:real_stable_log}

Log geometry is an efficient tool to work systematically with 
normal crossings divisors in algebraic geometry, and so to deal with many questions involving degenerations and compactifications. 
In this paper, extending the work of \cite{NS} considering complex curves, we use log geometry to study real curves in special fibers of toric degenerations and their deformations. 

Throughout this section we assume basic familiarity with log geometry \cite{K,O}. 
For an introductory review see \cite{A_KN}. 
Nevertheless, we start reviewing a couple of definitions to fix our notation.
Then, we introduce real stable log maps. 

\subsection{Log schemes}
\label{Sec:log schemes}

All the monoids in this paper are assumed to be commutative.

\begin{definition}
\label{Def: log structure}
Let $X$ be a scheme. A \emph{pre log structure} on $X$ is a sheaf of monoids $\shM$ on $X$ together with a homomorphism of monoids 
$\alpha \colon \shM \lra \mathcal{O}_X$ where we consider the structure sheaf $\mathcal{O}_X$ as a monoid with respect to multiplication. 
A pre log structure on $X$ is called a \emph{log structure} if $\alpha$ induces an isomorphism
\[
\restr{\alpha}{\alpha^{-1}(\mathcal{O}_X^{\times})}:{\alpha}^{-1}(\mathcal{O}_X^{\times})\lra \mathcal{O}_X^{\times}.
\]
We call a scheme $X$ endowed with a log structure a \emph{log scheme}, and denote the log structure on $X$ by $(\mathcal{M}_X, \alpha_X)$, 
or sometimes by omitting the structure homomorphism from the notation, just by $\mathcal{M}_X$. 
We denote a scheme $X$, endowed with a log structure by $(X, \mathcal{M}_X)$.
\end{definition}
Let $X$ be a log scheme. The \emph{ghost sheaf} of $X$ is the sheaf on $X$ defined by
\[
\overline{\M}_{X} \colon = \M_X /{\alpha^{-1}(\mathcal{O}_{X}^{\times})}\,.
\]
\begin{example} 
\label{The divisorial log structure}
Let $D\subset X$ be a divisor. Let $j \colon X\setminus D\hookrightarrow X$ denote the inclusion map. 
The \emph{divisorial log structure} on $X$ is the pair $(\shM_{(X,D)},\alpha_X)$, 
where $\shM_{(X,D)}$ is  the sheaf of regular functions on $X$, that restrict to units on $X\setminus D$. 
That is,
\[\shM_{(X,D)}\coloneqq j_*(\mathcal{O}_{X\setminus D}^\times)\cap \mathcal{O}_X\] 
The structure homomorphism $\alpha_X$ is given by the inclusion $\alpha_X:\shM_{(X,D)} \hookrightarrow \mathcal{O}_X$.

\begin{figure} 
\center\scalebox{0.6}{\input{P113.pspdftex}}
\caption{The fan $\Sigma_{\PP_{(1,1,3)}}$ for the toric variety $\PP_{(1,1,3)}$ on the left and its dual $\Delta_{\PP_{(1,1,3)}}$ on the right, referred to as the moment map image of $\PP_{(1,1,3)}$ \cite[Chapter 4]{Fulton}, along with some points labelled on different types of fibers of the moment map $\PP_{(1,1,3)} \to \Delta_{\PP_{(1,1,3)}}$.}
\label{Fig: P113}
\end{figure}

\end{example}
Analogous to assigning a sheaf to a presheaf, we can assign a log structure to a pre log structure, using a fibered coproduct, as follows.
\begin{definition}
\label{Def: Log pre log}
Let $\alpha \colon \shP\lra \mathcal{O}_X$ be a prelog structure on $X$.  
We define the \emph{log structure associated to the prelog structure $(\shP,\alpha)$} on $X$ as follows. Set 
\[\shP^a \coloneqq \bigslant{\shP\oplus \mathcal{O}_X^{\times}}
{\{(p,\alpha(p)^{-1})\,\big|\, p\in \alpha^{-1}(\mathcal{O}^\times_X)\big\}}\]
and define the structure homomorphism 
$\alpha^a \colon  \shP^a  \longrightarrow  \mathcal{O}_X$ by
\begin{eqnarray}
\nonumber
\alpha^a(p,h) & = h\cdot \alpha(p)
\end{eqnarray}
One can easily check that $(\shP^a,\alpha^a)$ is a log structure on $X$.
\end{definition}

Let $Y$ be a log scheme, and let $f\colon X\to Y$ be a scheme theoretic morphism. Then, the log structure on $Y$, given by $\alpha_Y\colon \shM_Y\lra \mathcal{O}_Y$,
induces a log structure on $X$ defined as follows. First define a prelog structure,
by considering the composition
\[  f^{-1}\shM_Y\to f^{-1}\mathcal{O}_Y \to \mathcal{O}_X \,.\]
Then, we endow $X$ with the log structure associated to this prelog structure, as in Definition \ref{Def: Log pre log}. We refer to this log structure as the \emph{pull back log structure} or the \emph{induced log structure} on $X$, and denote it by $\shM_X = f^* \shM_Y$.

\begin{example}
\label{Ex: trivial log point}
For every scheme $X$, $\mathcal{M}_X \coloneqq \cO_X^{\times}$ define a log structure on $X$, called the \emph{trivial log structure}. In particular, taking $X= \Spec \C$, we call 
$(\Spec \C,\C^{\times})$ the \emph{trivial log point}.
\end{example}

\begin{example}
\label{Ex: standard log point}
Let $X \coloneqq \Spec \CC$. Define $\mathcal{M}_X \coloneqq \CC^{\times} \oplus \NN$, and $\alpha_X\colon \M_X \to \CC$ as follows.
\begin{center}
$\alpha_X ( x , n ) \coloneqq \left\{
	\begin{array}{ll}
		x  & \mbox{if } n = 0 \\
		0 & \mbox{if } n \neq 0
	\end{array}
\right.
$
\end{center}
The corresponding log scheme $O_0 \coloneqq (\Spec \CC, \C^{\times} \oplus \NN)$ is called the \emph{standard log point}. 
One can check that the log structure on the standard log point is the same as the pull-back of the divisorial log  structure 
on $\A^1$ with the divisor $D=\{0\}\subset \A^1$.
\end{example}
\begin{definition}
A log structure $\shM$ on a scheme $X$ is called {\it coherent} 
if \'etale locally on $X$ there exists a finitely generated monoid $\shP$ and,
denoting  $\shP_X$ the constant sheaf corresponding to $\shP$,
a prelog structure $\shP_X\to\mathcal{O}_X$ whose associated log structure, 
defined as in Definition \ref{Def: Log pre log}, is isomorphic to $\shM$. 
Recall that a monoid $\shP$ is called {\it integral} if the canonical map 
$\shP \rightarrow \shP^{\mathrm{gp}}$ from $\shP$ to its Grothendieck group $\shP^{\mathrm{gp}}$ 
is injective.
The sheaf $\shM$ is called {\it integral} if $\shM$ is a sheaf of integral monoids. If $\shM$ is both coherent and integral then it is called {\it fine}. 
\end{definition}

\begin{definition}
\label{Def:chart}
For a scheme $X$ with a fine log structure $\shM$ a {\it chart for $\shM$ } is a homomorphism $\shP_X \to \shM$ 
for a finitely generated integral monoid $\shP$ which induces $\shP^a\cong \shM$ over an \'etale open subset of $X$. 
\end{definition}
\begin{remark}
\label{Rem: toric log structure}
Let $X$ be a toric variety defined by a fan $\Sigma$ in $N_\RR$.
The \emph{toric log structure} on $X$ is the divisorial 
log structure $\M_X=\M_{(X,D)}$ defined by the toric boundary divisor 
$D \subset X$, that is, the union of toric divisors of $X$.
Moreover this log structure is fine. Note that for any $\sigma \in \Sigma$,
and recalling that $N \coloneqq \Hom(M,\Z)$,
we have the affine toric subset 
\[U_\sigma= \Spec\CC[\sigma^\vee\cap N]\] 
of $X$, where $\sigma^\vee$ is the dual cone of 
$\sigma$, and the divisorial log structure $\M_X$ is locally generated by the monomial functions on this open subset. That is, the canonical map
\begin{eqnarray}
\label{toric chart}
\sigma^\vee\cap N & \lra & \CC[\sigma^\vee\cap N]\\
\nonumber
n & \longmapsto & z^n
\end{eqnarray}
is a chart for the log structure on $U_\sigma$. 
\end{remark}
\begin{proposition}
\label{toric charts and ghost sheaves}
For any point $x$ in the interior of the toric stratum of $X$ associated to $\sigma$, the toric chart~\eqref{toric chart} induces a canonical isomorphism
\[
\sigma^\vee\cap N/\sigma^\perp\cap N \stackrel{\sigma}{\lra} \ol \M_{X,x}.
\]
\qed
\end{proposition}
\begin{proof}
The proof is straightforward and can be found in \cite[Prop A.30]{MyThesis}. 
\end{proof}

\begin{example} \label{Ex: P113}
Consider the weighted projective plane 
\[X = \PP_{(1,1,3)} = (\CC^{3} \setminus \{ 0\}) / \CC^* \,, \]
where the action of $\CC^*$ is given by $\xi \cdot (x_0,x_1,x_2) = (\xi x_0, \xi x_1, \xi^3 x_2)$, with associated toric fan and moment image $\Delta_{\PP_{(1,1,3)}}$ as in Figure \ref{Fig: P113}. Recall that $ \PP_{(1,1,3)}$ admits a torus fibration onto $\Delta_{\PP_{(1,1,3)}}$, with fibers illustrated on the right hand side of the figure. We describe below the toric log structure locally around points on different types of fibers. Let $D=D_1 \cup D_2 \cup D_3$ be the toric boundary divisor.
Denote by $\M_X$ be the toric log structure as in Remark \ref{Rem: toric log structure}, given by the sheaf of regular functions which are units in $X \setminus D$. 

For a point $p_i$ contained in the $2$-dimensional torus orbit of $\PP_{(1,1,3)}$, that is, whose image under the moment map maps to the interior of $\Delta_{\PP_{(1,1,3)}}$, the stalk of $\M_X$ is generated by functions that are all units on $X \setminus D$. Hence, the stalk of the ghost sheaf at such a point is trivial; $\ol\M_{X,p_i}=0$. Let $p_b$ be a generic point in $D \subset \PP_{(1,1,3)}$. Then the image of $p_b$ under the moment map lies in the boundary of $\Delta_{\PP_{(1,1,3)}}$. The stalk of $\M_X$ is generated by functions which are of the form $ h \cdot x^a$, where $a \in \NN$, $h$ is a unit on $X \setminus D$, and the local equation of $D$ in a neighbourhood of $p_b$ is given by $(x=0)$. Then, the stalk of the ghost sheaf is 
\begin{align*}
    \ol\M_{X,p_b} & \cong \NN \\
    x^a & \mapsto a
\end{align*}
Note that one can describe the stalk at such points $p_b \in D_i \subset D$, for $i=1,2,3$, as the integral points on the dual of the ray in $\Sigma$ corresponding to the divisor $D_i$. Similarly, for the two smooth torus fixed points $p_{(0,0)}$ and $p_{(3,0)}$,
we have $\ol\M_{X,p_{(0,0)}}
=\ol\M_{X,p_{(3,0)}}
=\NN^2$. Finally, for the singular torus fixed point 
$p_{(0,1)}$, the stalk of the ghost sheaf $\ol\M_{X,p_{(0,1)}}$ is the monoid of integral points contained in the dual cone $\sigma^{\vee}$ of the cone
$\sigma$ spanned by $(-1,0)$ and $(1,3)$ in the fan.
\end{example}

\begin{figure} 
\center\scalebox{0.6}{\input{2P113.pspdftex}}
\caption{A polyhedral decomposition and the moment map of the central fiber of the associated toric degeneration.}
\label{Fig: 2P113}
\end{figure}

Let $\pi \colon \mathcal{X} \to \AA^1 =\Spec \CC[t]$ be a toric degeneration of toric varieties as in \S \ref{Sec: toric degenerations}. 
The total space $\mathcal{X}$ is toric and so can be endowed with the divisorial 
toric log structure $\shM_\mathcal{X}$.
We endow the central fiber $X_0 \subset \mathcal{X}$ with the pull-back log structure
\begin{equation}  \label{Eq: pullback_log}
\shM_{X_0} = \iota^* \shM_\mathcal{X} \end{equation}
induced from $\mathcal{X}$ by the inclusion $\iota \colon X_0 \hookrightarrow \mathcal{X}$. 
Note that the toric affine cover of $\mathcal{X}$ induces an affine cover of the central fiber 
$X_0$ by restriction to $t=0$, and this defines a chart for the log structure $\shM_{X_0}$.

\begin{example} \label{Ex: 2P113}
In Figure \ref{Fig: 2P113}, we illustrate on the left a polyhedral decomposition of $\RR^2$ and on the right the moment map of the central fiber $X_0$ of the associated toric degeneration $\mathcal{X}$ as in 
\S\ref{Sec: toric degenerations}. The central fiber $X_0$ has two irreducible components, both isomorphic to the weighted projective 
plane $\PP_{(1,1,3)}$ (see Example \ref{Ex: P113}). We endow $X_0$ with the pullback log structure 
\eqref{Eq: pullback_log} $\M_{X_0}$. 
For $p_i$
a point contained in the $2$-dimensional torus orbit of one of the two irreducible components, we have $\ol\M_{X_0,p_i}=\NN$.
From the point of view of the degeneration 
$\pi \colon \cX \rightarrow \A^1=\Spec \C[t]$, $\ol\M_{X_0,p_i}=\ol\M_{\mathcal{X},p_i}=\NN$ is generated by 
$t^n$, $n \in \NN$. Note that the restriction of the log structure
on $X_0$ to the toric components is not the toric log structure
such as described in Example \ref{Ex: P113}: the ghost sheaf of the toric log structure is generically trivial, whereas the ghost sheaf of the log structure restricted from $X_0$ is generically $\NN$.
A more general comparison of these two log structures can be found in 
\cite[Lemma 5.13]{GS1}.

For a point $p_d$ in the double locus of $X_0$, we have $\ol\M_{X_0,p_d}=S_e \coloneqq \NN^2 \oplus_\NN \NN$, where 
$e$ is the integral length of the bounded vertical edge of the polyhedral decomposition, and where in the fibered coproduct 
$\NN^2 \oplus_\NN \NN$
one uses the maps $\NN \rightarrow \NN^2$, $1 \mapsto (1,1)$, and
$\NN \rightarrow \NN$, $1 \mapsto e$. Transversally to the double locus, the log structure near $p_d$
is locally isomorphic to the restriction over $t=0$ of the toric log structure on the toric surface of equation $zw=t^e$, whose algebra of regular functions is exactly the monoid algebra of the monoid $S_e$.
\end{example}

\begin{definition}
\label{Def: log morphism}
A \emph{morphism of log schemes} $f \colon (X,\mathcal{M}_X) \to (Y,\mathcal{M}_Y)$ is a morphism of schemes $f \colon X \to Y$
along with a homomorphism of sheaves of monoids $f^\sharp \colon f^{-1}\mathcal{M}_Y \to \mathcal{M}_X$ such that
the diagram
\[\begin{CD}
f^{-1}\M_Y @>{f^\sharp}>> \M_X\\
@V{\alpha_Y}VV@VV{\alpha_X}V\\
f^{-1}\shO_Y@>{f^*}>>\shO_X.
\end{CD}\]
is commutative. Here, $f^*$
is the usual pull-back of regular functions defined by the
morphism $f$. Given a morphism of log spaces $f: (X,\M_X) \to (Y,\M_Y)$, we denote by $\underline{f}\colon X \to Y$ 
the underlying morphism of schemes. By abuse of notation, the underlying morphism on topological spaces is also denoted by $\underline{f}$. 
\end{definition}

\begin{definition} \label{def_strict}
Let  $f \colon (X,\mathcal{M}_X) \to (Y,\mathcal{M}_Y)$ be a log morphism and let $x$ be a point of $X$. We say that $f$ is \emph{strict} at the point $x$
if $f^\sharp$ induces an isomorphism $f^{-1} \shM_{Y,f(p)} \simeq \shM_{X,p}$.
\end{definition}

We refer to \cite[Defn. $3.23$]{Gross} for the notion of \emph{log smooth morphism}.
We recall that a toric variety is log smooth over the trivial log point. 
If $\pi \colon \mathcal{X} \rightarrow \A^1$ is a toric degeneration of toric varieties, then 
$\pi$ can be naturally viewed as a log smooth morphism. In restriction to $t=0$, we get a log 
smooth morphism $X_0 \rightarrow O_0$, that is, $X_0$ is log smooth over the standard log 
point $O_0$.
We also refer to \cite[Remark 3.25]{Gross} for the notion of 
\emph{integral} morphism of log schemes.

\subsection{Stable log maps}
\label{Sec: Log curves}

We recall that a $\ell$-marked stable map with target a scheme $X$ is 
a map $f \colon C \rightarrow X$, where $C$ is a proper nodal curve
with $\ell$ marked smooth points $\mathbf{x}=(x_1,\ldots,x_\ell)$,
such that the group of automorphisms of $C$ fixing $\mathbf{x}$
and commuting with $f$ is finite.

The notion of stable map has a natural generalization to the setting of logarithmic 
geometry \cite{AbramovichChen, GSlogGW,ACGS17}. 

\begin{definition}
\label{logCurve}
Let $X \rightarrow B$ be a log morphism between log schemes.
A \emph{$\ell$-marked stable log map with target $X \rightarrow B$}
is a commutative diagram of log schemes
\[\begin{CD}
C @>{f}>> X\\
@V{\pi}VV@VV{}V\\
W@>{}>>B\,,
\end{CD}\]
together with a tuple of sections $\mathbf{x} = (x_1,\dots,x_{\ell})$
of $\pi$,
where
$\pi$ is a proper log smooth and integral morphism of
log schemes, such
that, for every geometric point $w$ of $W$, the restriction of $\underline{f}$ to $w$
with the marked points $\underline{\mathbf{x}}(w)$ is an ordinary stable map, and such that, if $U \subset C$ is the non-critical locus of $\underline{\pi}$, we have 
$\overline{\shM}_C|_U \simeq \underline{\pi}^{*} \overline{\shM}_W
\oplus \bigoplus_{j=1}^{\ell} (x_j)_{*}\NN_W$.
\end{definition}

\begin{remark} \label{rem_basic}
When the scheme underlying $W$ in Definition \ref{logCurve}
is a point, then $W$ is a log point $(\Spec \C, \C^{*} \oplus Q)$, defined as the standard log point $O_0=(\Spec \C,\C^{*} \oplus \NN)$ as in Example \ref{Ex: standard log point}, but with $\NN$ replaced by a 
possibly more complicated monoid $Q$.
The general enumerative theory of stable log maps, called log Gromov--Witten theory and developed in \cite{GSlogGW, AbramovichChen}, is based on the notion of \emph{basic stable log map}.
For a basic stable log map over a log point 
$(\Spec \C, \C^{*} \oplus Q)$, the monoid $Q$ is the so-called 
\emph{basic monoid}. The basic monoid is uniquely determined by the combinatorial type of the stable log map and has a natural tropical interpretation: the dual cone $\Hom(Q,\RR_{\geq 0})$
is the base of a universal family of tropical curves with given combinatorial type. 
For the present paper, we will mostly not use this general theory because we will end up with a finite list of unobstructed stable log maps over the standard log point $O_0$. In particular, we will have a one-parameter smoothing log maps and there is always a distinguished morphism
$O_0 \rightarrow (\Spec\C,\C^{*}\oplus Q)$, where $Q$ is the basic monoid, just coming from our degeneration situation.
One exception is in the final step of the proof of 
Theorem \ref{main theorem}, where we explain how a computational proof of \cite{NS} can be replaced by a conceptual argument based on the stable reduction theorem for basic stable log maps of 
\cite{GSlogGW}.
\end{remark}

The following result is the specialization to the case of a log smooth curve over the standard log point of the general theorem of Fumiharu Kato \cite [p.222]{Katof} describing explicit local charts for a log smooth curve.
We denote by $\sigma \colon \NN \rightarrow \C^{\times}$ the chart on the standard 
log point $O_0 \coloneqq (\Spec \C, \C^{\times} \oplus \NN)$ defined by 
\begin{center}
$\sigma(n) \coloneqq \left\{
	\begin{array}{ll}
		1  & \mbox{if } n = 0 \\
		0 & \mbox{if } n \neq 0
	\end{array}
\right.$
\end{center}
(see Example \ref{Ex: standard log point}).

\begin{theorem}(\cite [p.222]{Katof} ) 
\label{Thm: structure of log curves}
Let $\pi \colon C \rightarrow O_0$ be a log smooth and integral curve over the 
standard log point $O_0$.
Then,
\'etale locally, $(C,\mathcal{M}_C)$ is isomorphic to one of
the following log schemes $V$ over $O_0$.
\begin{enumerate}
\item[(i)]
$\Spec(\CC[z])$ with the log structure induced by the chart
\[
\NN\lra \mathcal{O}_V,\quad q\longmapsto \pi^{*} \sigma(q).
\]
\item[(ii)]
$\Spec(\CC[z])$ with the log structure induced by the chart
\[
\NN\oplus\NN\lra \mathcal{O}_V,\quad (a,q)\longmapsto z^a \pi^{*}\sigma(q).
\]
\item[(iii)]
$\Spec(\CC[z,w]/(zw))$ with the log structure
induced by the chart
\[
S_e \coloneqq \NN^2\oplus_\NN \NN\lra \mathcal{O}_V,\quad
\big((a,b),q\big)\longmapsto  z^a w^b \pi^{*}\sigma(q) \,.
\]
Here $\NN\to\NN^2$ is the diagonal embedding and $\NN\to \NN$,
$1\mapsto e$ is some homomorphism uniquely defined by $\pi \colon C\to O_0$.
Moreover, $e \neq 0$.
\end{enumerate}
In each case, the log morphism $\pi \colon C \to O_0$ is represented by the
canonical maps of charts $\NN\to \NN$, $\NN\to \NN\oplus\NN$
and $\NN\to \NN^2\oplus_\NN \NN$, respectively where we identify the domain $\NN$ 
with the second factor of the image. 
\end{theorem}

In Theorem \ref{Thm: structure of log curves}, cases $\mathrm{(i),(ii),(iii)}$ 
correspond to neighbourhoods of general points, marked points and nodes of $C$ respectively.

\begin{remark}
\label{Rem:nonstandard charts}
Let $(C,\mathcal{M}_C)$ be a log smooth curve over the standard log point 
$O_0$ and let $p$ be a nodal point of $C$. 
By Theorem 
\ref{Thm: structure of log curves} (iii), the log structure of $C$ at $p$ is induced by a chart of the form
\[ \beta \colon S_e
=\NN^2 \oplus_{\NN} \NN \to \mathcal{O}_{C,q}\]
\[
 ((a,b),q) \longmapsto \left\{
	\begin{array}{ll}
		z^a w^b  & \mbox{if } q = 0 \\
		0 & \mbox{if } q \neq 0
	\end{array}
\right.
\]
For every $\zeta \in \C^{\times}$, this chart becomes after the change of variables 
$z \mapsto \zeta z$
the chart 
\[ \beta_{\zeta} \colon S_e
=\NN^2 \oplus_{\NN} \NN \to \mathcal{O}_{C,q} \] 
\[
 ((a,b),q) \longmapsto \left\{
	\begin{array}{ll}
		(\zeta^{-1}z^a)w^b  & \mbox{if } q = 0 \\
		0 & \mbox{if } q \neq 0
	\end{array}
\right.
\]
We will use the charts $\beta_{\zeta}$ in \S \ref{Sec: from real md to real log}.
\end{remark}

\begin{figure} 
\center\scalebox{0.6}{\input{2P1.pspdftex}}
\caption{ A map $\underline{\varphi}_0\colon \PP^1 \coprod_q \PP^1 \to \PP_{(1,1,3)} \coprod_{\PP^1} \PP_{(1,1,3)}$ from two copies of $\PP^1$ glued along a nodal point $q$, into two copies of the weighted projective plane $\PP_{(1,1,3)}$ glued along a copy of $\PP^1$.}
\label{Fig: 2P1}
\end{figure}

\begin{example} \label{Ex: 2P1}
In Figure \ref{Fig: 2P1}, we illustrate a map 
$\underline{\varphi}_0$ from a curve $C$ to the central fiber $X_0$
of the toric degeneration considered in Example 
\ref{Ex: 2P113}. 
Using the pullback log structure \eqref{Eq: pullback_log}, $X_0$ is a log scheme log smooth over the standard log point $O_0$. If $\M_C$ is a log structure on $C$ such that 
$(C,\M_C)$ is log smooth over $O_0$ and $\underline{\varphi}_0$
extends into a log morphism $\varphi_0$, then we have $\ol\M_{C,\xi}=\NN$ for a general point $\xi \in C$ by Theorem \ref{Thm: structure of log curves}(i), $\ol\M_{C,p_i}=\NN \oplus \NN$ for 
a point $p_i\in C$ which maps to a divisor of $X_0$
not contained in the double locus by Theorem \ref{Thm: structure of log curves}(ii), and for the node $q \in C$ we have $\ol\M_{C,q}=S_{e'}$, where the monoid $S_{e'}$ is defined analogously to $S_e$ in Theorem \ref{Thm: structure of log curves}(iii). Here we use $e'\in \NN$, as $e\in \NN$ is taken in Example 
\ref{Ex: 2P113} while defining the log structure on the target. 
In \S \ref{Sec: from real md to real log}, we will give a detailed description of the set of log morphisms $\varphi_0$ with underlying morphism $\underline{\varphi}_0$ at the level of schemes.

\end{example}

\subsection{Real stable log maps}
\label{Subsec: real log schemes}
Given a scheme $X$ over $\mathbb{C}$, a \emph{real structure} on $X$ is an anti-holomorphic involution 
$\iota \colon X\to X$ on the set of complex points of $X$. We call a pair $(X,\iota)$ a
\emph{real scheme}. By abuse of notation we usually omit $\iota$ when talking about real schemes. The real locus $X(\RR)$ of a real scheme 
$(X,\iota)$ is the set of fixed points of $\iota$ acting on the set of complex 
points of $X$. 

Toric varieties are naturally
defined over $\Z$: this follows directly from the explicit description of
toric varieties in terms of fans. 
In particular, a toric variety over $\C$ is naturally defined over 
$\R$. We will refer to this real structure as the \emph{standard real structure} 
on a toric variety.
In this paper, we will always consider standard real structures on toric varieties over 
$\C$.
If $X$ is a $n$-dimensional toric variety over $\C$ with its standard real structure, of fan $\Sigma \subset M$, then the intersection of the $n$-dimensional torus orbit 
$M \otimes_{\Z} \C^{\times} \simeq (\C^{\times})^n$ in $X$ with the real locus $X(\RR)$
is $M \otimes_{\Z} \R^{\times} \simeq (\R^{\times})^n$. We refer to \cite[\S4]{Fulton}
for more details on the real locus of toric varieties with standard real structures.

We review real structures in the setting of logarithmic algebraic geometry, as introduced in \cite{AS2}. 
For a more comprehensive study of real log schemes, we refer to \cite[\S $5$]{AS2}.
\begin{definition}
Let $(X,\M_X)$ be a log scheme over $\CC$ with a real structure
$\iota_{X} \colon X\to X$ on the underlying scheme. Then a \emph{real
structure} on $(X,\M_X)$ (\emph{lifting $\iota_{X}$)} is an
involution
\[
\tilde\iota_{X}= (\iota_X,\iota_X^\flat): (X,\M_X)\lra (X,\M_X)
\]
of log schemes over $\RR$ with underlying scheme-theoretic morphism
$\iota_{X}$. The data consisting of $(X,\M_X)$ and the involutions
$\iota_X$, $\iota_X^\flat$ is called a \emph{real log scheme}.
\end{definition}

The standard real structure $\iota_X$ on a toric variety $X$ preserves the toric divisors and so lifts by \cite[Proposition 5.4]{AS2} to a real structure $\tilde{\iota}_X$ on $(X,\M_X)$ where $\M_X$ denotes the toric log structure.  
We get a real log structure on the central fiber $X_0$ of a toric degeneration
$\pi \colon \mathcal{X} \rightarrow \A^1$ by restriction of the real log
structure on the toric total space $\cX$. We similarly get a real log structure 
on the standard log point $O_0 =(\Spec \C, \C^{\times} \oplus \NN)$ by restriction
to $\{0\} \subset \A^1$ of the 
real log structure of $\A^1$.
The restricted real log structure makes sense due to the following result which appears as Proposition $5.9$ in \cite{AS2}.
\begin{proposition}\label{Prop: compatibility with base change}
Cartesian products exist in the category of real log schemes.
\end{proposition}

Real morphisms of real log schemes are defined as follows.

\begin{definition}
\label{Def: Category of real log schemes}
Let $(X,\M_X)$ and $(Y,\M_Y)$ be real log schemes. A morphism
$f:(X,\M_X)\to (Y,\M_Y)$ of real log schemes is called \emph{real} if the
following diagram is commutative.
\[\begin{CD}
f^{-1}\iota_Y^{-1}\M_Y @>{f^{-1}\iota^\flat_Y}>> f^{-1}\M_Y\\
@V{\iota_X^{-1}f^\flat}VV@VV{f^\flat}V\\
\iota_X^{-1}\M_X@>{\iota^\flat_X}>>\M_X.
\end{CD}\]
Here the left-hand vertical arrow uses the identification
$\iota_Y\circ f= f\circ\iota_X$.
\end{definition}

Given a toric degeneration of toric varieties 
$\pi \colon \mathcal{X} \rightarrow \A^1$, $\pi$ is naturally a 
real log morphism. Similarly, its restriction to the central fiber 
$X_0 \rightarrow O_0$ is naturally a real log morphism.

\begin{definition}
\label{log map}
Let $X \rightarrow B$ be a real log morphism between real log schemes. 
A \emph{real $\ell$-marked stable log map with target $X \rightarrow B$} is a commutative diagram of real log schemes
\[\begin{CD}
C @>{f}>> X\\
@V{\pi}VV@VV{}V\\
W@>{}>>B\,,
\end{CD}\]
together with a tuple of real sections $\mathbf{x} = (x_1,\dots,x_{\ell})$
of $\pi$, such that the underlying diagram and sections of log schemes over 
$\C$ define a stable log map (in the sense of Definition \ref{logCurve}).
\end{definition}

\begin{remark}
In Definition \ref{log map}, we assume that the sections defining the marked points
are real. This will be enough for the purposes of the present paper.
A more general definition should allow pairs of complex conjugated marked 
points.
\end{remark}

\section{Counts of maximally degenerate real curves}
\label{Sec. counts of md curves}

In this section, we study stable maps to the central fiber $X_0$
of a toric degeneration.
In \S\ref{Sec:prelog}, we review following
\cite[\S 4]{NS} pre-log stable maps to $X_0$ and the associated tropical curve.
In \S \ref{Sec:md_real_stable_maps}, we introduce the notion of
maximally degenerate real stable map to $X_0$.
In \S \ref{Sec:count_md_stable_maps}, we prove the main result of this
section, Theorem \ref{Thm: counts of prelog curves}, 
counting the number of ways to lift a tropical curve in $M_{\R}$ to a maximally degenerate real stable map to $X_0$.

\subsection{Pre-log stable maps and tropical curves}
\label{Sec:prelog}
Following \cite[Defn 4.1]{NS}, we first recall the notion of torically transverse curve in a toric variety. 

\begin{definition}\label{Def: torically transverse}
Let $X$ be a $n$-dimensional toric variety.  An algebraic curve $C \subset X$ is called a \emph{torically transverse
curve} if it is disjoint from all toric strata of codimension greater than $1$. 
A stable map $\varphi: C\to X$ over a scheme $S$ is a \emph{torically transverse stable map} if the following holds:
\begin{enumerate}
\item[(i)] For the restriction $\varphi_s$ of
$\varphi$ to every geometric point $s\to S$, 
denoting $\mathrm{Int} X$ the $n$-dimensional torus orbit of $X$,
$\varphi_s^{-1}(\mathrm{Int} X) \subset C_s$ is dense.
\item[(ii)] $\varphi_s(C_s)\subset X$ is a torically transverse curve.
\end{enumerate}
\end{definition}

Let $\mathscr{P}$ an integral polyhedral decomposition of $M_{\Q}$, and let 
$\pi \colon \mathcal{X} \rightarrow \A^1$ be the corresponding toric degeneration, with special fiber $X_0$ (see \S \ref{Sec: toric degenerations}).
The irreducible components $X_v$ of $X_0$ are indexed by the vertices $v$ of 
$\mathscr{P}$.

As in \cite[Defn.~4.3]{NS}, we next define pre-log stable maps to $X_0$.
They are stable maps to $X_0$ obtained by gluing together torically transverse curves in the various components of $X_0$ which satisfy a 
kissing condition along the double locus of $X_0$, 
where exactly two irreducible components of $X_0$ intersect.
\begin{definition}
\label{Def: prelog}
A stable map $\underline{\varphi}_0 \colon C_0 \to X_0$ is called a \emph{pre-log stable map} if the following holds.
\begin{enumerate}
\item[(i)] For every $v$, the restriction $C_0\times_{X_0}X_v \to
X_v$ is torically transverse.
\item[(ii)] If $P\in C_0$ maps to the singular locus of $X_0$, then $C_0$
has a node at $P$, and $\underline\varphi_0$ maps the two branches $(C_0',P)$,
$(C_0'',P)$ of $C_0$ at $P$ to different irreducible components
$X_{v'}$, $X_{v''}\subset X_0$. Moreover, if $w'$ is the intersection
index with the toric boundary $D'\subset X_{v'}$ of the restriction
$(C_0',P)\to (X_{v'},D')$, and $w''$ accordingly for $(C_0'',P)\to
(X_{v''},D'')$, then $w'=w''$. This condition is referred to as the kissing condition at a nodal point $P$. 
\end{enumerate}
\end{definition}

\begin{figure} 
\center\scalebox{0.6}{\input{TropP1s.pspdftex}}
\caption{ The associated tropical curve to the map $\underline{\varphi}_0$ in Figure \ref{Fig: 2P1}.}
\label{Fig: TropP1q}
\end{figure}

To every pre-log stable map $\underline\varphi_0$, there is an associated tropical curve $h \colon \Gamma \lra M_{\RR}$, 
contained in the one-skeleton of 
$\mathscr{P}$, and constructed as the \emph{dual intersection
graph} of $\underline{\varphi}_0$. The details of this construction which we summarise below can be found in \cite[Construction 4.4]{NS}. 

\begin{construction}\label{Constr. tropical curve associated to prelog curve}
Let $\underline{\varphi}_0 \colon C_0 \rightarrow X_0$ be a pre-log stable map. 
We first define an open graph $\tilde \Gamma$. 
Two irreducible components of $C_0$ are called
\emph{indistinguishable} if they intersect in a node \emph{not}
mapping to the singular locus of $X_0$ \footnote{In the following sections we focus attention on what we call maximally degenerate curves, which in particular have no indistinguishable components.}.  Define $\tilde\Gamma$ to be the graph whose set of vertices equals the quotient of the set of
irreducible components of $C$ modulo identification of
indistinguishable ones. Any nodal point $P_E\in C_0$ corresponds to a bounded edge $E$ of $\tilde\Gamma$. 
To define the set of unbounded edges, first set $D\subset X_0$ to be the union of
toric prime divisors of the $X_v$ \emph{not} contained in the
singular locus of $X_0$. Then the set of unbounded edges is
$\underline{\varphi}_0^{-1}(D)$. 
An unbounded edge $E$ labelled by $Q_E\in
\underline{\varphi}_0^{-1}(D)$ attaches to $V\in \tilde\Gamma^{[0]}$ if $Q_E\in
C_V$.  Now, we define the map
\[ h \colon \tilde\Gamma \to M_{\RR}  \]
as follows. First note that by Definition \ref{Def: prelog}, if $C_V\subset C$ denotes the irreducible
component indexed by a vertex $V$ then $h(V)=v$ for the unique
$v\in\mathscr{P}^{[0]}$ with $\underline{\varphi}_0(C_V)\subset X_v$. 
Under $h$ a bounded edge $E$ corresponding to a nodal point $P_E$ maps to the line segment joining $h(V')$, $h(V'')$ if $P_E\in C_{V'}\cap C_{V''}$. 
To determine the images of unbounded edges of $\tilde \Gamma$ under $h$, 
let $E$ be an unbounded edge labelled by $Q_E \in \underline{\varphi}_0^{-1}(D)$. 
Assume for a $1$-cell $e\in \Sigma^{[1]}$ of the toric fan of $X$, $D_e\subset X_{h(V)}$ is the corresponding toric divisor with $\varphi(Q_E)\subset D_e$. 
Then, $h$ maps $E$ homeomorphically to $h(V)+e\subset M_\QQ$.
Finally define the weights
of the edges of $\tilde{\Gamma}$ adjacent to a vertex $V$ by the intersection numbers of
$\underline{\varphi}_0 |_{C_V}$ with the toric prime divisors of $X_{h(V)}$. This is
well-defined by the definition of a pre-log stable map.
While $\tilde\Gamma$ may have divalent vertices the kissing condition in the definition of a pre-log curve assures that the two weights at
such a vertex agree. We may thus remove any divalent vertex and
join the adjacent edges into a single edge. The resulting weighted
open graph $\Gamma$ has the same topological realization as
$\tilde\Gamma$ and hence $h$ can be interpreted as a map 
\[h \colon \Gamma\to  M_\QQ \,.\] 

Let $Y$ be a proper toric variety and let $u_1,\ldots, u_k\in M$
be the primitive generators of the rays of the toric fan of $Y$.
If $\varphi \colon C \to Y$ is a torically transverse stable map, we have by 
\cite[Prop 4.2]{NS} that 
\begin{equation}
\label{Eq: balancing}
    \sum_{i=1}^k w_i u_i=0 ~ \mathrm{for} ~
w_i=\deg \varphi^*(D_i)
\end{equation}
the intersection number of $C$ with the
toric divisor corresponding to $u_i$. 
This result ensures that $h$ satisfies the balancing condition
\eqref{Eq:balancing_condition} at each of its vertices, and hence is a tropical curve.
\end{construction}

\begin{figure} 
\center\scalebox{1.0}{\input{1to7.pspdftex}}
\caption{A maximally degenerate stable map $\underline{\varphi}_0:C_0\to X_0$,
where $X_0$ is the union of $7$ copies of $\PP^2$
and one copy of $\PP^1\times \PP^1$, and
where each irreducible component of $X_0$ is identified with its moment map image.}
\label{Fig: P1}
\end{figure}

\begin{figure} 
\center\scalebox{0.8}{\input{Trop17.pspdftex}}
\caption{The tropical curve associated to the maximally degenerate stable map in Figure \ref{Fig: P1}.}
\label{Fig: Trop17}
\end{figure}

\begin{example}
Figure \ref{Fig: 2P1} discussed in Example \ref{Ex: 2P1} represents a maximally degenerate stable map to the special fiber $X_0$
of the toric degeneration of Example \ref{Ex: 2P113}. We illustrate 
in Figure \ref{Fig: TropP1q} the associated tropical curve given by 
Construction \ref{Constr. tropical curve associated to prelog curve}.
\end{example}

\begin{example}
In Figure \ref{Fig: P1}, we illustrate a maximally degenerate stable map $\underline{\varphi}_0 \colon C_0 \rightarrow X_0$
to the special fiber $X_0$ of the degeneration of $\PP^2$
defined by the polyhedral decomposition described in Example 
\ref{ex_polyhedral_decomposition}. The special fibers $X_0$ is the union of $7$ copies of $\PP^2$ and of one copy of $\PP^1 \times \PP^1$. The curve $C_0$ has $9$ irreducible components, 
all isomorphic to $\PP^1$. The $7$ components of $C_0$ labeled $1$ to $7$ are mapped by $\underline{\varphi}_0$ to the $7$ copies of $\PP^2$ in $X_0$, whereas the two components of $C_0$
drawn in red are mapped to the same component $\PP^1\times \PP^1$ and their images intersect, creating a node in the image curve $\underline{\varphi}_0(C_0)$ contained in the smooth locus of $X_0$. The dual graph $\tilde{\Gamma}$ of $C_0$ constructed in Construction \ref{Constr. tropical curve associated to prelog curve} is represented on the left of Figure \ref{Fig: Trop17}. The graph $\Gamma$ is obtained from $\tilde{\Gamma}$ by removing the two red bivalent vertices corresponding to the two red irreducible components of $C_0$.
Note that the image of the tropical curve 
$h \colon \tilde{\Gamma} \rightarrow M_\RR$ represented on the right of Figure \ref{Fig: Trop17} is exactly the one-skeleton of the polyhedral decomposition in Figure \ref{Fig: P}.
\end{example}

\subsection{Maximally degenerate real stable maps}
\label{Sec:md_real_stable_maps}

Following \cite[Definition 5.1]{NS} in the complex case, we define 
\emph{real lines} and we prove some of their basic properties. 

\begin{definition}\label{def lines}
Let $X$ be a complete toric variety and $D\subset X$ the toric
boundary, that is, the union of toric divisors of $X$. A \emph{line} on $X$ is a non-constant, torically
transverse map $\varphi: \mathbb{P}^1\to X$ such that for every irreducible component $D_j$ of $D$, 
we have $\sharp \varphi^{-1}(D_j) \leq 1$, and there are at most $3$ irreducible components $D_j$ of $D$ with 
$\sharp \varphi^{-1}(D_j) \neq 0$.
We call a line \emph{real} if $\varphi$ is a real map, where both 
$\PP^1$ and $X$ are equipped with their standard real structure.
\end{definition}

\begin{remark}
We follow Definition 5.1 of \cite{NS} for the use of the terminology ``line" in Definition \ref{def lines}. However, this terminology might be misleading: one should note for example that a linear embedding of $\PP^1$ in $\PP^n$
is not a line in the sense of Definition \ref{def lines} if 
$n \geq 3$. 
\end{remark}

\begin{lemma}\label{lem_real_intersection}
Let $\varphi \colon \PP^1 \to X$ be a real line as in Definition \ref{def lines}. 
Then each intersection point of $\varphi(\PP^1)$
with the toric boundary of $X$ is real.
\end{lemma}

\begin{proof}
As $\varphi \colon \PP^1 \to X$ is a real map, the intersection of 
$\varphi(\PP^1)$ with each toric divisor consists of some real points and some pairs of complex conjugated points. 
By definition of a line, the intersection of $\varphi(\PP^1)$ with a toric divisor is either empty or consists of a single point. 
Therefore, each intersection point of $\varphi(\PP^1)$ with a toric divisor is necessarily real.
\end{proof}

Following \cite{NS}, if the intersection of the image of a line with the toric boundary $D$ consists of two points, 
we call it a \emph{divalent line}, and if it consists of three points we call it a \emph{trivalent line}. 
Note that since a line is torically transverse by definition, the associated tropical curve satisfies the balancing condition \eqref{Eq: balancing}.
For a line $\varphi \colon \PP^1 \rightarrow X$, let
$u_i \in M$ be the primitive generators of the rays
corresponding to the divisors of $X$ being intersected, and let $w_i$ be the intersection numbers
with $\varphi$. Writing $(\mathbf{u},\mathbf{w})
=((u_i)_i,(w_i)_i)$, we say that the line $\varphi$
is of type $(\mathbf{u},\mathbf{w})$.

\begin{lemma}
\label{Lem: torsor}
Let $X$ be a toric variety of dimension $n$ and $a\in\{2,3\}$.
Let
\[(\mathbf u, \mathbf w)=((u_i)_{1 \leq i \leq a},(w_i)_{1 \leq i \leq a})\in M^a\times(\NN\setminus\{0\})^a\]
with $u_i$ primitive and $\sum_{i=1}^a w_i u_i=0$.
Denote by $\LL_{(\mathbf u, \mathbf w)}^{\RR}$
the moduli space of real lines whose type is given by $(\mathbf u, \mathbf w)$. 
There is a transitive action of $M \otimes_{\ZZ} \R^{\times}$ on $\LL_{(\mathbf u, \mathbf w)}^{\RR}$. 
Moreover, in the trivalent case this action is simply transitive,
while in the divalent case the action factors over a simply
transitive action of $(M/\ZZ u_1) \otimes_{\ZZ} \R^\times= (M/\ZZ u_2) \otimes_{\ZZ} \R^\times$.
\end{lemma}

\begin{proof}
Let $S$ be the one- or two-dimensional toric variety defined by the complete fan
with rays $\mathbb{Q}u_i$. Then up to a toric birational transformation $X$ is a product $S \times Y$ with $Y$ a
complete toric variety, so that the composition of any line $\mathbb{P}^1 \to S \times Y$ 
with the projection to the second factor is constant by \cite[Lemma 5.2]{NS}. 
Therefore, it suffices to consider the case where $\dim X \leq 2$. 

In the divalent case, we can assume $\dim X=1$, and there is only one isomorphism class of stable real maps 
$\mathbb{P}^1 \to X$ that is totally branched over $0$ and $\infty$. 

In the trivalent case, we can assume that $X=S$ is the complete toric surface whose toric fan is given by the rays 
$\mathbb{Q} u_i$. Let $D_1$, $D_2$, and $D_3$
be the toric divisors of $S$ respectively dual to the rays 
$\mathbb{Q} u_1$, $\mathbb{Q} u_2$ and $\mathbb{Q} u_3$.
Let $\varphi \colon \PP^1 \rightarrow S$ be a line in $S$ of type 
$(\mathbf u, \mathbf w)$ and let 
$q_1$, $q_2$, and $q_3$ be the points of $\PP^1$ respectively mapped by 
$\varphi$ on $D_1$, $D_2$, and $D_3$. Let $y$ be the unique coordinate on 
$\PP^1$ such that $y(q_1)=-1$, $y(q_2)=0$, and $y(q_3)=1$. For every 
$n \in N$, the monomial $z^n$ is a rational function on $S$ with vanishing order 
$(u_i,m)$ along $D_i$, where $(-,-)$ denote the natural duality pairing between 
$M$ and $N$. It follows that the rational function 
$\varphi^{*}(z^n)$ on $\PP^1$ is necessarily of the form 
\[ \varphi^{*}(z^n)=\chi_\varphi(n) \prod_{i=1}^3 (y-y(q_i))^{(w_i u_i,n)}\,,\]
for some $\chi_\varphi(n) \in \C^{\times}$. As $z^{n+n'}=z^n \cdot z^{n'}$ for every 
$n, n' \in N$, the map $\chi_\varphi \colon n \mapsto \chi_\varphi(n)$ is a character of $N$.
Therefore, the map $\varphi \mapsto \chi_\varphi$ identifies the set of lines in $S$
of type $(\mathbf u, \mathbf w)$ with the set $\Hom(N,\C^{\times})=M \otimes \C^{\times}$
of characters of $N$, which is obviously a torsor under $M \otimes \C^{\times}$.

A line $\varphi \colon \PP^1 \rightarrow S$ is real if and only if $\chi_\varphi(n) \in \R^\times$ 
for every $n \in N$. Therefore, the map $\varphi \mapsto \chi_\varphi$ identifies the set of real lines in $S$
of type $(\mathbf u, \mathbf w)$ with the set $\Hom(N,\R^{\times})=M \otimes \R^{\times}$
of real-valued characters of $N$, which is obviously a torsor under $M \otimes \R^{\times}$.
\end{proof}

Let $\mathscr{P}$ an integral polyhedral decomposition of $M_{\Q}$, and let 
$\pi \colon \mathcal{X} \rightarrow \A^1$ be the corresponding toric degeneration, with special fiber $X_0$.
The irreducible components $X_v$ of $X_0$ are indexed by the vertices of 
$\mathscr{P}$. The notion of maximally degenerate stable map is formulated in 
\cite[Definition 5.6]{NS}.

\begin{definition}
\label{Def: maximally degenerate}
Let $\underline{\varphi}_0 \colon C_0 \to X_0$ be a pre-log stable map. 
If for every $v\in\mathscr{P}^{[0]}$ the
projection $C_0\times_{X_0}X_v \to X_v$ is a line, or, for $n=2$, the
disjoint union of two divalent lines intersecting disjoint toric
divisors, then $\underline{\varphi}_0$ is called \emph{maximally degenerate}.
\end{definition}
\begin{figure} 
\center{\input{MD.pspdftex}}
\caption{A maximally degenerate stable map $\underline{\varphi}_0:C_0\to X_0\cong\prod_4\PP^2$, 
where each $\PP^2$ is identified with its moment map image, and the associated tropical curve.}
\label{Fig: MD}
\end{figure}

We refer to Figure \ref{Fig: MD} for an illustration of a maximally degenerate stable map.

The following Lemma ensures that the two natural ways to define a 
\emph{maximally degenerate real stable map}, either as a real stable map which happens to be maximally degenerate, 
or as a real stable map obtained from gluing real lines, agree.

\begin{lemma}
Let $\underline{\varphi}_0 \colon C_0 \rightarrow X_0$ be a real stable map which is maximally degenerate. 
Then, for every $v \in \mathscr{P}^{[0]}$, 
the
projection $C_0\times_{X_0}X_v \to X_v$ is a real line, or, for $n=2$, the
disjoint union of two divalent real lines intersecting disjoint toric
divisors.
\end{lemma}

\begin{proof}
Let $\iota_{C_0}$ and $\iota_{X_0}$ be the real structures on $C_0$ and $X_0$. 
As $X_0$ is endowed with the real structure induced by the standard toric real structure on $\mathcal{X}$, 
we have $\iota_{X_0}(X_v)=X_v$ for every $v \in \mathscr{P}^{[0]}$.
As $\underline{\varphi}_0$ is a real map, it follows that for every $v \in 
\mathscr{P}^{[0]}$ we have $\iota_{C_0}(C_0 \times_{X_0} X_v)
= C_0 \times_{X_0} X_v$. Therefore, when $C_0 \times_{X_0} X_v$ is a line, we deduce that $C_0 \times_{X_0} X_v$ is a real line.
When $C_0 \times_{X_0} X_v$ is a disjoint union of two divalent lines intersecting disjoint toric divisors, 
then each of these two lines is also real because
the restriction of 
$\iota_{X_0}$ to $X_v$ preserves the toric divisors of $X_v$.
\end{proof}

\subsection{Counts of maximally degenerate real stable maps}
\label{Sec:count_md_stable_maps}

Let $g \in \NN$, 
$\Delta \colon M \setminus \{0\} \rightarrow \NN$ with finite support, and
$\mathbf{A} = (A_1,\ldots,A_\ell)$ an affine constraint that is general for $(g,\Delta)$ in the sense of Definition \ref{Def: general affine}. 
Recall from Proposition \ref{prop_finite_tropical} that the set 
$\T_{g,\ell,\Delta}(\mathbf{A})$ of $\ell$-marked genus $g$ tropical curves of
degree $\Delta$ matching $\mathbf{A}$ is finite. 
Let $\mathscr{P}$ be a polyhedral decomposition of $M_{\Q}$
which is good for $(g,\Delta,\mathbf{A})$ in the sense of 
Definition \ref{Def: good decomposition}.

Let $\pi \colon \mathcal{X} \rightarrow \A^1$ be the toric degeneration 
defined by $\mathscr{P}$ (see \S \ref{Sec: toric degenerations}).
Let $X_t \coloneqq \pi^{-1}(t)$ be the fiber of $\pi$ over 
$t \in \A^1$. In particular, we denote by $X_0$ the central fiber of 
$\pi$. We fix $\mathbf{P}=(P_1,\ldots,P_\ell)$ a $\ell$-tuple of real points in the $n$-dimensional torus orbit of $X=X_1$. 
By \eqref{Eq:tilde_Z}-\eqref{Eq:tilde_Z_0}, the affine subspaces 
$A_j$ and the points $P_j$
define incidence conditions 
\[\mathcal{Z}_{A_j,P_j} \subset \mathcal{X}\]
and 
\[Z_{A_j,P_j,t} = \mathcal{Z}_{A_j,P_j} \cap X_t \subset X_t\,.\] 
We denote 
$\mathbf{\mathcal{Z}}_{\mathbf{A},\mathbf{P}}
\coloneqq (\mathcal{Z}_{A_1,P_1}, \ldots, \mathcal{Z}_{A_\ell,P_\ell})$
and 
$\mathbf{Z}_{\mathbf{A},\mathbf{P},t} 
\coloneqq (Z_{A_1,P_1,t},\ldots,Z_{A_{\ell},P_{\ell},t})$. In particular, 
$\mathbf{Z}_{\mathbf{A},\mathbf{P},0}$ is a tuple of incidence conditions in the 
central fiber $X_0$.
Because the points $P_j$ are taken to be real, the subvarieties 
$\mathcal{Z}_{A_j,P_j}$ are real subvarieties of 
$\mathcal{X}$, and for $t \in \A^1(\RR)$, 
$Z_{A_j,P_j,t}$ is a real subvariety of $X_t$.

Let $(\Gamma,\mathbf{E},h) \in \T_{g,\ell,\Delta}(\mathbf{A})$.
We will study maximally degenerate real stable maps 
$\underline{\varphi}_0 \colon C_0 \rightarrow X_0$
matching the incidence conditions $\mathbf{Z}_{\mathbf{A},\mathbf{P},0}$ 
and
with associated tropical curve $(\Gamma,\mathbf{E},h)$
(in the sense of Construction 
\ref{Constr. tropical curve associated to prelog curve}).

It is shown in \cite[Prop 5.7]{NS} that the number of maximally degenerate
stable maps to $X_0$ matching the incidence conditions $\mathbf{Z}_{\mathbf{A},\mathbf{P},0}$ and
with associated tropical curve $(\Gamma, \mathbf{E},h)$  
equals the lattice index of
the map of lattices 
\begin{eqnarray}
\label{eq:NS map}
\,\,\,\,\,\,\,\,\,
\mathcal{T}_h \colon \mathrm{Hom}(\Gamma^{[0]},M)
& \longrightarrow &
\prod_{E\in \Gamma^{[1]} \setminus \Gamma^{[1]}_\infty} M/\ZZ u_{(\partial^-E,E)} 
\times
\prod_{j=1}^{\ell} M/\big((\QQ u_{(\partial^-E_j,E_j)} + L(A_j))\cap M\big) \\
\nonumber
\phi  &  \longmapsto & \big((\phi(\partial^+E)-\phi(\partial^-E))_E,(\phi(\partial^- E_j))_j\big).
\end{eqnarray}
In \eqref{eq:NS map}, the quotient space $M/\ZZ u_{(\partial^-E,E)}$ can be viewed as a space of directions transverse to the direction $u_{(\partial^-E,E)}$, and the map 
$\phi \mapsto \phi(\partial^+ E)-\phi(\partial^-E)$
is therefore a measure of how far from $u_{(\partial^-E,E)}$
is the direction of the line segment connecting $\phi(\partial^+E)$ and $\phi(\partial^-E)$. Similarly, the quotient space 
$M/\big((\QQ u_{(\partial^-E_j,E_j)} + L(A_j))\cap M\big)$ can be viewed as a space of directions transverse to the subspace spanned by the direction of $u_{(\partial^-E_j,E_j)}$ and the directions parallel to $A_j$, and so the map $\phi \mapsto \phi(\partial^- E_j)$ is a measure of how far is the direction of $\phi(\partial^- E_j)$ from being contained in this subspace. 
Recall that the lattice index of a map of lattices is the order of its cokernel, 
which is also equal to the order of the kernel of the induced map obtained by tensoring both sides with $\CC^{\times}$. 
However, this is no longer true once restricting to $\RR^{\times}$. 
So, to obtain the analogue of this result for maximally degenerate real stable maps, we first need the following definition.

\begin{figure} 
\center{\input{GluingLines.pspdftex}}
\caption{The tropical curve in Figure \eqref{Fig: MD} obtained by gluing $4$ lines.}
\label{Fig: Gluing lines}
\end{figure}
\begin{definition}
\label{Def: real index}
Let $\Psi: M_1\to M_2$ be an inclusion of lattices of finite index. Let 
\begin{equation}
\label{Eq:Coker}
\mathrm{Coker}(\Psi)=\ZZ/_{(p_1)^{e_1}\ZZ} \times \ldots \times  \ZZ/_{(p_n)^{e_n}\ZZ} 
\end{equation}
be the primary decomposition of the free abelian group $\mathrm{Coker}(\Psi)$. We define the \emph{real index} of $\Psi$ as
\begin{equation}
\label{Eq: real index}
\mathcal{D}^{\RR}_{\Psi}\coloneqq 2^{\#  \{i~|~p_i=2   \}} 
\end{equation}
\end{definition}
\begin{lemma}
\label{Lem: kernel}
Let $\Psi: M_1\to M_2$ be an inclusion of lattices as in Definition \ref{Def: real index}, and let
\begin{equation}
\label{Eq: sequence tensered with R}    
    \Psi_{\RR}: M_1 \otimes_{\ZZ} \RR^{\times} \lra M_1 \otimes_{\ZZ} \RR^{\times}
\end{equation}
be the map obtained from $\Psi$ by tensoring with $\RR^{\times}$. Then, $\# \{ \mathrm{Ker} \Psi_{\RR} \} = \mathcal{D}^{\RR}_{\Psi}$. 
\end{lemma}

\begin{proof}
The proof follows from elementary homological algebra of abelian groups 
(see for example \cite[Chapter 3]{MR1269324}). Consider the short exact sequence
\[ 0 \lra M_1 \lra M_2 \lra \mathrm{Coker}(\Psi) \lra 0\,,  \]
where $\mathrm{Coker}(\Psi)$ is as in \eqref{Eq:Coker}. Since $M_1$ and $M_2$ are free abelian, 
this is a free resolution of $\mathrm{Coker}(\Psi)$, and tensoring it with $\RR^{\times}$, we obtain an exact sequence
\[ 0 \lra \mathrm{Tor}_1(\R^{\times},\mathrm{Coker}(\Psi)) \lra M_1 \otimes_{\ZZ} \RR^{\times} \overset{~~\Psi_{\RR}~~}{\longrightarrow} M_2 \otimes_{\ZZ} \RR^{\times} \lra 0 \]
and so 
\[\mathrm{Ker}(\Psi_{\RR})=\mathrm{Tor}_1(\R^{\times},\mathrm{Coker}(\Psi))\,.\]
By additivity of the $\mathrm{Tor}$ bifunctor, we have  \[\mathrm{Tor}_1(\R^{\times},\mathrm{Coker}(\Psi))
=\prod_{j=1}^{\ell} \mathrm{Tor}_1(\R^{\times},\ZZ/_{(p_j)^{e_j}\ZZ})\,.\]
For every positive integer $m$, the free resolution
\[  0 \lra \ZZ \overset{~~\cdot m~~}{\longrightarrow} \ZZ \lra \ZZ/_{n\ZZ} \lra 0 \] induces the short exact sequence
\[ 0 \lra \mathrm{Tor}_1(\RR^{\times},\ZZ/_{m\ZZ}) \lra \RR^{\times}  \lra \RR^{\times} \lra 0  \]
where the right hand side arrow is given by
\begin{eqnarray}
\nonumber
\RR^{\times} \lra \RR^{\times} \\
\nonumber
r \longmapsto r^{m}
\end{eqnarray}
Therefore, 
\[ \mathrm{Tor}_1(\RR^{\times},\ZZ/_{m\ZZ}) = \begin{cases} 
      0 & \mathrm{if~} m ~ \mathrm{is~odd} \\
      \ZZ/_{2 \ZZ} & \mathrm{if~} m ~ \mathrm{is~even}
   \end{cases}
\]
Hence, the result follows.
\end{proof}

Now we are ready to define the real analogue of the lattice index of
\cite[Prop 5.7]{NS}.

We denote by
\begin{eqnarray}
 \nonumber
M_{2,1}^{\tau_h} & \colon= &
\prod_{E\in \Gamma^{[1]} \setminus \Gamma^{[1]}_\infty} M/\ZZ u_{(\partial^-E,E)}  \\
\label{Eq: M22}
M_{2,2}^{\tau_h} & \colon= &
\prod_{j=1}^{\ell} M/\big((\QQ u_{(\partial^-E,E)} + L(A_j))\cap M\big)
\end{eqnarray}
the components of the target of the map $\mathcal{T}_h$ in \eqref{eq:NS map}. We 
denote by $Z_{A_j,P_j,0}(\RR)$ the real locus of $Z_{A_j,P_j,0}$.
We denote by
\[ Z_{j,\RR} \coloneq  Z_{A_j,P_j}^0(\RR) 
\cap (M \otimes_{\ZZ} \R^{\times})\]
the intersection of  $Z_{A_j,P_j,0}(\RR)$
with the real torus $M \otimes_{\Z} \R^{\times}$
of the toric component of $X_0$
containing $Z_{A_j,P_j,0}$, and 
\[  \mathbf{Z}_{\RR} \colon = (Z_{1,\RR},\ldots, Z_{j,\RR}) \subset \prod_{j=1}^{\ell} M\otimes_{\Z} \RR^{\times} \,.\]
We denote by 
\begin{equation}
    \label{Eq: [Z]}
[\mathbf{Z}_{\RR}]    
\end{equation}
the image in $M_{2,2}^{\tau_h} \otimes_{\Z} \RR^{\times}$ of $\mathbf{Z}_{\RR}$ by the quotient map 
\[ \prod_{j=1}^{\ell} M\otimes_{\Z} \RR^{\times} \lra  M_{2,2}^{\tau_h} \otimes_{\Z} \RR^{\times} \,.  \]
In fact, $[\mathbf{Z}_{\RR}]$ is a point in 
 $ M_{2,2}^{\tau_h} \otimes_{\Z} \RR^{\times}$, since the quotient mods out the linear directions of the affine constraints. 

\begin{definition}
\label{Def:twisted real index}
We define  $\sigma \coloneqq (0 \otimes 1, [\mathbf{Z}_{\RR}])
\in M_{2,1}^{\tau_h} \otimes_{\Z} \RR^{\times} \times  M_{2,2}^{\tau_h} \otimes_{\Z} \RR^{\times}$. 
We define the \emph{real index of $\mathcal{T}_h$ twisted by $\sigma$} as
\[ \mathcal{D}^{\RR}_{\mathcal{T}_h,\sigma} \colon = \begin{cases} 
       \mathcal{D}^{\RR}_{\mathcal{T}_h}, & \mathrm{if ~ the ~ image ~ of ~} \sigma \mathrm{~in~} 
\Coker(\mathcal{T}_h) \otimes_{\Z} \R^{\times} \mathrm{~ is~} 0 \otimes 1.
 \\
      0, & \mathrm{otherwise.} 
   \end{cases}
\]
 where $\mathcal{D}^{\RR}_{\mathcal{T}_h,\sigma}$ is the real lattice index of the map $\mathcal{T}_h$ in \eqref{eq:NS map}.
\end{definition} 
 
\begin{theorem}
\label{Thm: counts of prelog curves}
The number 
$N^{\RR-\mathrm{prelog}}_{\mathbf{Z}_{\mathbf{A},\mathbf{P},0},h}(X_0)$
of isomorphism classes of  
maximally degenerate real stable maps to $X_0$ matching the incidence conditions $\mathbf{Z}_{\mathbf{A},\mathbf{P},0}$
and with associated tropical curve $(\Gamma, \mathbf{E},h) 
\in \T_{g,\ell,\Delta}(\mathbf{A})$ is equal to 
the real lattice index twisted by $\sigma$ 
of the map $\mathcal{T}_h$ in \eqref{eq:NS map}, that is,   
\[N^{\RR-\mathrm{prelog}}_{\mathbf{Z}_{\mathbf{A},\mathbf{P},0},h}(X_0) =\mathcal{D}^{\RR}_{\mathcal{T}_h,\sigma}\,.\] 
\end{theorem}
\begin{proof}
After tensoring with $\mathbb{Q}$,  the map of lattices $\mathcal{T}_h$ in \eqref{eq:NS map} agrees with the map
\begin{eqnarray}
\label{Eq: affine linear map}
\nonumber
\quad\quad\mathcal{T}_{h,\QQ}: \mathrm{Hom}(\Gamma^{[0]},M_\QQ)&\lra&
\prod_{E\in\Gamma^{[1]} \setminus \Gamma_\infty^{[1]}} M_\QQ/\QQ
u_{(\partial^-E,E)} \times \prod_{j=1}^{\ell} M_\QQ/ \big(\QQ u_{(\partial^-
E_j,E_j)}+L(A_j)\big),\\
h &\longmapsto& \big((h(\partial^+E)-h(\partial^-E))_E,
(h(\partial^-E_j))_j\big).\nonumber
\end{eqnarray}
Fixing a general constraint $\mathbf{A}$, we ensure by Corollary \cite[Corollary 2.5]{NS} 
that $\mathcal{T}_{h,\QQ}$ is an isomorphism, and hence has finite real index.
Analogously as in the proof of \cite[Prop 5.7]{NS}, we will obtain maximally degenerate real curves by gluing real lines. 
Denote by $\tilde \Gamma$ the graph obtained from $\Gamma$ by inserting vertices at all points 
$h^{-1}(\mathscr{P}^{[0]})\setminus \Gamma^{[0]}$ corresponding to the points of intersection of the image of 
$h$ with the polyhedral decomposition $\mathscr{P}$, as in \cite[Constr. 4.4]{NS} 
to ensure the connectivity of the image of $\underline{\varphi}_0$. 
It follows from Lemma \ref{Lem: torsor} that the moduli space 
$\prod_{V\in
\tilde\Gamma^{[0]}} \LL^{\RR}_{(\mathbf u(V), \mathbf w(V))}$ 
of real lines can be identified with the real torus
\[ \mathrm{Hom}(\Gamma^{[0]},M) \otimes_{\ZZ} \R^{\times} \times \prod_{V\in \tilde\Gamma^{[0]}
\setminus \Gamma^{[0]}} (M/\ZZ u_{(V,E^-(V))}) \otimes_{\ZZ} \R^{\times} \,,\]
where the first factor is given by the real lines corresponding to trivalent vertices 
and the second one is given by the real lines corresponding to divalent vertices of $\tilde{\Gamma}$.

On the other hand, to glue real lines, which generically lie inside the $n$-dimensional torus orbits of irreducible components of $X_0$ pairwise together, 
we want them to intersect along points at the double locus of $X_0$. So, given two the real lines $L^-$ and $L^+$, dual to the vertices 
$\partial^- E$ and $\partial^+ E$ of a bounded edge $E$, we want them to
intersect at a point on the toric variety $X_{h(E)}$ in real points 
$P^-$ and $P^+$. By  Lemma \ref{lem_real_intersection}, $P^+$ and $P^-$ lie in the real $(n-1)$-dimensional torus orbit 
of $X_{h(E)}$. As the real $(n-1)$-dimensional torus orbit of $X_{h(E)}$ can be identified with $M/\ZZ u_{(\partial^-E,E)} \otimes_{\Z} \RR^{\times}$, 
and so the difference between $P^+$ and $P^-$ can be viewed as an element of  $M/\ZZ u_{(\partial^-E,E)} \otimes_{\Z} \RR^{\times}$. 
So, consider the map of real tori 
\begin{equation}
 \label{Eq: obstructed map}
\mathrm{Hom}(\Gamma^{[0]},M) \otimes_{\ZZ} \R^{\times} \times \prod_{V\in \tilde\Gamma^{[0]}
\setminus \Gamma^{[0]}} (M/\ZZ u_{(V,E^-(V))}) \otimes_{\ZZ} \R^{\times}
\longrightarrow
\prod_{E\in \tilde{\Gamma}^{[1]} \setminus 
\tilde{\Gamma}^{[1]}_\infty}
M/\ZZ u_{(\partial^-E,E)} \otimes_{\Z} 
\RR^{\times}\,.
\end{equation}
obtained by tensoring with 
$\RR^{\times}$ the map of lattices
\begin{eqnarray}
\nonumber
\mathrm{Hom}(\Gamma^{[0]},M) 
\times \prod_{V\in \tilde\Gamma^{[0]}
\setminus \Gamma^{[0]}} M/\ZZ u_{(V,E^-(V))}
& \longrightarrow & 
\prod_{E\in \Gamma^{[1]} \setminus 
\Gamma^{[1]}_\infty}
M/\ZZ  u_{(\partial^-E,E)} \\
\nonumber
h & \longmapsto &  \big((h(\partial^+E)-h(\partial^-E))_E\big) \,.
\end{eqnarray}
The map \eqref{Eq: obstructed map} can be thought of assigning to a set of real lines the obstructions to their gluing. 
Therefore, the set of real maximally degenerate stable maps to 
$X_0$ with associated tropical curve 
$h \colon \Gamma \rightarrow M_{\RR}$ can be identified 
with the kernel of the map \eqref{Eq: obstructed map}.

It remains to impose the constraints
$\mathbf{Z}_{\mathbf{A},\mathbf{P},0}$. Consider the map of lattices
\begin{eqnarray}
\nonumber
\mathrm{Hom}(\Gamma^{[0]},M)
\times \prod_{V\in \tilde\Gamma^{[0]}
\setminus \Gamma^{[0]}} (M/\ZZ u_{(V,E^-(V))})
& \longrightarrow &
\prod_{j=1}^{\ell}
M/((\Q u_{(\partial^-E,E)} +L(A_j)) \cap M)
\\
\nonumber
h & \longmapsto & (h(\partial^- E_j)_j)
\end{eqnarray}
and the induced map of real tori
\begin{equation}
    \label{Eq: map of real tori for constraints}
    \mathrm{Hom}(\Gamma^{[0]},M) \otimes_{\Z} \RR^{\times}
\times \prod_{V\in \tilde\Gamma^{[0]}
\setminus \Gamma^{[0]}} (M/\ZZ u_{(V,E^-(V))}) \otimes_{\ZZ} \R^{\times}
\longrightarrow
\prod_{j=1}^{\ell}
M/((\Q u_{(\partial^-E,E)} +L(A_j)) \cap M) \otimes_{\Z} \RR^{\times}
\end{equation}
obtained by tensoring with $\RR^{\times}$. 
For every line $C_{V_j}$ attached to a bivalent vertex 
$V_j$ of $\tilde{\Gamma}$ defined by the intersection of 
$\Gamma$ with the constraint $A_j$, we can consider the real locus $C_{V_j}(\RR)$ and then the intersection 
\[ C_{V_j,\RR} \coloneqq C_{V_j}(\RR) \cap 
(M \otimes_{\ZZ} \RR^{\times} )\]
with the real torus $M \otimes_{\Z} \R^{\times}$
of the toric component of $X_0$ containing 
$C_{V_j}$.
Let $[C_{V_j,\RR}]$ be the image of $C_{V_j,\RR}$ by the quotient map 
$M \otimes_{\Z} \R^{\times} \rightarrow M/((\Q u_{(\partial^-E,E)} +L(A_j)) \cap M) \otimes_{\Z} \RR^{\times}$: 
it is a point measuring the position of $C_{V_j}$ transversely to the torus $(\Q u_{(\partial^-E,E)} +L(A_j)) \cap M) \otimes_{\Z} \RR^{\times}$.
The map \eqref{Eq: map of real tori for constraints} assigns to a tuple $(C_V)_V$ of real lines the tuple $([C_{V_j}])_{j=1,\ldots,\ell}$.
On the other hand, the constraints $Z_{A_j,P_j,0}$ 
similarly define points  
$[Z_{j,\R}]$ in
$
M/((\Q u_{(\partial^-E,E)} +L(A_j)) \cap M) \otimes_{\Z} \RR^{\times}$. A line $C_{V_j}$ matches the constraint 
$Z_{A_j,P_j,0}$ if and only of 
$[C_{V_j,\R}]=[Z_{j,\RR}]$. 
Therefore, the preimage of $[\mathbf{Z}_{\RR}]$
by the map \eqref{Eq: map of real tori for constraints} 
is exactly the set of real lines matching the constraints $\mathbf{Z}_{\mathbf{A},\mathbf{P},0}$. 

The map \eqref{Eq: obstructed map}
 restricted to factors associated to the bivalent vertices of 
$\tilde{\Gamma}$ is an isomorphism onto the factors associated to the bounded edges of $\tilde{\Gamma}$ attached to bivalent vertices, and so removing these factors from the domain and target of \eqref{Eq: obstructed map} does not change the kernel of \eqref{Eq: obstructed map}.   
We conclude that the set of maximally degenerate real stable maps to 
$X_0$ with associated tropical curve 
$h \colon \Gamma \rightarrow M_{\RR}$ and matching the  constraints 
$\mathbf{Z}_{\mathbf{A},\mathbf{P},0}$
is the preimage of
$\sigma \coloneqq (0 \otimes 1, [\mathbf{Z}_{\R}])$
by the map of real tori $\mathcal{T}_{h,\RR^{\times}}$ obtained from the map of lattices $\mathcal{T}_h$ in \eqref{eq:NS map} 
by tensoring with $\RR^{\times}$. Thus, this set is a torsor under $\Ker(\mathcal{T}_{h,\RR^{\times}})$
if
$\sigma$
is in the image of 
$\mathcal{T}_{h,\RR^{\times}}$, and empty if
$\sigma$
is not in the image of 
$\mathcal{T}_{h,\RR^{\times}}$.
The result then follows from Lemma \ref{Lem: kernel}.
\end{proof}

Let $\underline{\varphi}_0 \colon C_0 \rightarrow X_0$ be a maximally 
degenerate real stable map to $X_0$ matching the 
incidence conditions $\mathbf{Z}_{\mathbf{A},\mathbf{P},0}$
and with associated tropical curve $(\Gamma, \mathbf{E},h) 
\in \T_{g,\ell,\Delta}(\mathbf{A})$.
As explained in \cite[Remark 5.8]{NS}, the proof of \cite[Prop 5.7]{NS} in the complex case establishes a bijection between 
$Z_{A_j,P_j,0}\, \cap \, \underline{\varphi}_0 (C_0)$ 
and the intersection of the two subtori 
$(\ZZ u_{(\partial^- E_j,E_j)}) \otimes_{\Z} \C^{\times}$ and $(L(A_j) \cap M) \otimes_{\Z} \C^{\times}$ in 
$M \otimes_{\Z} \C^{\times}$. This
latter number of intersection points is the covering degree of
\[
(\ZZ u_{(\partial^- E_j,E_j)})\otimes_{\Z} \C^{\times} 
\times 
(L(A_j) \cap M) \otimes_{\Z} \C^{\times}
\lra 
\big((\QQ u_{(\partial^- E_j,E_j)}+L(A_j))\cap M\big)\otimes_{\Z} 
\C^{\times}\,,
\]
which equals the index 
\begin{equation}
\nonumber
[\ZZ u_{(\partial^- E_j,E_j)} + L(A_j)\cap M:(\QQ u_{(\partial^-
E_j,E_j)}+L(A_j))\cap M].
\end {equation}
The following proposition is the real analogue of this result.
\begin{proposition}
\label{Prop: number of intersection points}
Let $\underline{\varphi}_0 \colon C_0 \rightarrow X_0$ be a maximally 
degenerate real stable map to $X_0$ matching the 
incidence conditions $\mathbf{Z}_{\mathbf{A},\mathbf{P},0}$
and with associated tropical curve $(\Gamma, \mathbf{E},h) 
\in \T_{g,\ell,\Delta}(\mathbf{A})$.
Then, the number of real intersection points of the image of $\underline{\varphi}_0$ with the incidence $Z_{A_j,P_j,0}$
equals the real lattice index of the inclusion
\begin{equation}
\label{Eq: lattice marked points}
\mathcal{A}_j: \ZZ u_{(\partial^- E_j,E_j)} + L(A_j) \cap M \lra (\QQ u_{(\partial^-
E_j,E_j)}+L(A_j))\cap M  \,.
\end{equation}
\end{proposition}
\begin{proof}
The set of real intersection points between 
$Z_{A_j,P_j,0}$ and the image curve $\underline{\varphi}_0(\underline{C}_0)$
is in bijection with the intersection of the two subtori 
$(\ZZ u_{(\partial^- E_j,E_j)}) \otimes_{\Z} \RR^{\times}$ and $(L(A_j)\cap M) \otimes_{\Z} \RR^{\times}$ in 
$M \otimes_{\Z} \RR^{\times}$.
Therefore, the number of these real intersection points is the cardinality of the kernel of the map 
\begin{eqnarray}
 \nonumber
 (\ZZ u_{(\partial^- E_j,E_j)})\otimes_{\Z} \RR^{\times} 
\times 
(L(A_j)\cap M) \otimes_{\Z} \RR^{\times}
 & \lra & 
\big((\QQ u_{(\partial^- E_j,E_j)}+L(A_j))\cap M\big)\otimes_{\Z} 
\RR^{\times} \\
\nonumber
 (z_1, z_2) & \longmapsto & z_1 z_2^{-1}
\end{eqnarray}
The result then follows from Lemma 
\ref{Lem: kernel}.
\end{proof}

\section{From real maximally degenerate curves to real log curves}
\label{Sec: from real md to real log}
In this section, we stay in the setup of \S \ref{Sec:count_md_stable_maps}.
We compute the numbers of ways to lift a maximally degenerate real stable map 
$\underline{\varphi}_0 \colon \underline{C}_0 \rightarrow X_0$
to a real stable log map 
$\varphi_0 \colon C_0 \rightarrow X_0$ in the sense of
Definition \ref{log map}. Recall from 
\S \ref{Sec:log schemes} that we view $X_0$ as a log scheme log smooth over the 
standard log point $O_0$. We denote by $\pi_0 \colon X_0 \rightarrow O_0$
the corresponding log morphism. 

A $\ell$-marked maximally degenerate real stable map 
$(\underline{\varphi}_0 \colon \underline{C}_0 \rightarrow X_0, 
\mathbf{x}_0)$ is a maximal degenerate real stable map $\underline{\varphi}_0 \colon \underline{C}_0 \rightarrow X_0$ with a $\ell$-tuple 
$\mathbf{x}_0=(x_{0,1},\ldots,x_{0,\ell})$ of real marked points. 
We say that 
$(\underline{\varphi}_0 \colon \underline{C}_0 \rightarrow X_0, 
\mathbf{x}_0)$ matches incidence conditions 
$\mathbf{Z}_{\mathbf{A},\mathbf{P},0}$ if $\underline{\varphi}_0(x_{0,j})
 \in Z_{A_j,P_j,0}$ for every $j=1,\ldots,\ell$.

\begin{theorem}
\label{Thm: lifting md curves to log maps}
Let $(\underline{\varphi}_0 \colon \underline{C}_0 \to X_0, \textbf{x}_0)$ be a marked maximally
degenerate real stable map matching incidence conditions 
$\mathbf{Z}_{\mathbf{A},\mathbf{P},0}$. 
Let $(\Gamma,\mathbf{E},h)$ be the associated tropical curve, 
where $\mathbf{E}=(E_1,\dots,E_\ell)$.
 For any $E \in \Gamma^{[1]}$ set $w^{\R}(E) =2$ if $w(E)$ is even, and $w^{\R}(E) =1$ if $w(E)$ is odd.
Let $w^{\R}(\Gamma,\mathbf{E})$ be the total real weight of
$(\Gamma,\mathbf{E})$ defined by
\begin{equation}
\label{Eq: Total real weight}    
    w^{\R}(\Gamma,\mathbf{E}) \coloneqq 
\prod_{E \in \Gamma^{[1]}\setminus \Gamma^{[1]}_\infty}
w^{\R}(E) \cdot\prod_{j=1}^{\ell} w(E_j) \,.
\end{equation}
Then there are exactly $w^{\RR}(\Gamma,\mathbf{E})$ pairwise non-isomorphic pairs 
$(\varphi_0 \colon C_0/O_0 \to X_0, \mathbf{x}_0)$ of real stable maps over the standard log 
point $O_0$ with underlying marked stable map isomorphic to $\underline{\varphi}_0$, and such that $\varphi_0$ is strict at $p \in \underline{C}_0$ as in Definition \ref{def_strict}
if $p \notin \mathbf{x}$ and 
$\pi_0 \colon X_0 \rightarrow O_0$ is strict at $\underline{\varphi}_0(p)$.
\end{theorem}
Before proving Theorem \ref{Thm: lifting md curves to log maps}, we will
analyse the action of the real involution on nodal points of log curves.
\begin{lemma}
\label{Lem: involution on ghosts}
Let $q$ be a nodal point on a log smooth curve $(C,\mathcal{M}_C)$ over the standard log point. Let $\beta_{\zeta}: S_e \to \mathcal{O}_{C,q}$, be a chart defined as in Remark \ref{Rem:nonstandard charts} for a root of unity $\zeta \in \C^{\times}$. Let $(\iota,\iota^{\flat})$ on $(C,\mathcal{M}_C)$ be a real structure on $(C,\M_C)$ lifting the standard real involution $\iota$ on $\underline{C}$. Then, \'etale locally, on a
neighbourhood of $q$, the involution $\iota^{\flat}$ is defined as
\begin{eqnarray}
\nonumber
\phi: S_e \oplus_{\beta^{-1}_{\zeta}(\mathcal{O}_C^{\times})} \mathcal{O}_C^{\times} & \longmapsto &  S_e \oplus_{\beta^{-1}_{\zeta}(\mathcal{O}_C^{\times})} \mathcal{O}_C^{\times} \\
\nonumber
(p, h) & \longmapsto & (\phi_1(p), \phi_2(p)\cdot \iota(h))
\end{eqnarray}
where $\phi_1: S_e \to S_e$ is the identity, and $\phi_2: S_e \to  \mathcal{O}_C^{\times}$ is given by \[s_{((1,0),0)}  \mapsto  \zeta^2, \,\  s_{((0,1),0)} \mapsto  1, \,\ s_{((0,0),1)}
\mapsto  1 \,. \] 
\end{lemma}
\begin{proof}
In order to use uniformly multiplicative notation for monoids, we denote
$t^p h$ for $(p,h)$ where $p\in S_e$ and $h\in \mathcal{O}_C^{\times}$.
Here $t^p$ is simply a formal multiplicative notation for $p \in S_e$.
We have
$t^{p} \cdot t^{p'}=t^{p+p'}$ for every $p, p' \in S_e$.

The real log structure being a lift of the real structure $\iota$ on 
$\underline{C}$, we necessarily have $\phi(h)=\iota (h)$ for every 
$h \in \cO_C^{\times}$.

We can write
$\phi(t^p)=t^{\phi_1(p)} \phi_2(p)$, where 
$\phi_1 \colon S_e \rightarrow S_e$
and $\phi_2 \colon S_e \rightarrow \cO_C^{\times}$
are morphisms of monoids.
Writing that $\phi$ is an involution, we obtain 
\[ t^p= \phi^2(t^p)= \phi(t^{\phi_1(p)}\phi_2(p))= \phi(t^{\phi_1(p)})\phi(\phi_2(p))  \]
Since $\phi_2(p) \in \mathcal{O}_C^{\times}$, we have $\phi(\phi_2(p))= \iota(\phi_2(p))$.
On the other hand, we have \[\phi(t^{\phi_1(p)})=t^{\phi_1^2(p)}\phi_2(p)\] by the definition of $\phi$. Therefore, we have
\[ p= t^{\phi_1^2(p)} \phi_2(p)\iota(\phi_2(p))\,,  \]
where $\phi_1^2(p) \in S_e$ and $\phi_2(p)\iota(\phi_2(p)) \in \mathcal{O}_C^{\times}$. Since each component of $C$ remains invariant  under the standard involution $\iota$ on $C$, we get $\phi_1=\mathrm{Id}$. Hence, we obtain $\phi(t^p)= \phi_2(p)\cdot t^p$. 
Therefore, we have
\[   \phi(s_{((1,0),0)}) = \phi_2(s_{((1,0),0)})s_{((1,0),0)}    \,.  \]
Thus, 
\begin{equation}
\label{Eq:1}
\beta_{\zeta}(\phi(s_{((1,0),0)})) = \phi_2(s_{((1,0),0)}) \zeta^{-1}z
\end{equation}
where $\beta_{\zeta}$ is the chart defined as in Remark \ref{Rem:nonstandard charts}.
Moreover, since $\phi$ is compatible with $\iota$, we have $\beta_{\zeta}\phi= \iota \beta_{\zeta}$,  Hence, 
\begin{equation}
\label{Eq:2}
\beta_{\zeta}(\phi(s_{((1,0),0)})) = \iota(\zeta^{-1}z) = \overline{\zeta^{-1}}z= \zeta \,.
\end{equation}

From equations \eqref{Eq:1} and \eqref{Eq:2}, it follows that $ \phi_2(s_{((1,0),0)})\zeta^{-1}= \zeta$, hence 
\[ \phi_2(s_{((1,0),0)}) =\zeta^2  \]
By an analogous computation, we obtain 
\[ \phi_2(s_{((0,1),0)}) = \phi_2(s_{((0,0),1)}) = 1\,.    \]
\end{proof}

Now we are ready to prove the main result of this section.

\textbf{Proof of Theorem \ref{Thm: lifting md curves to log maps}:}
The locus of strictness of $\pi_0$ is the union of the $n$-dimensional torus orbits of the toric irreducible components of $X_0$,
\[\mathring{X}_0 \colon = X_0 \setminus (\mathrm{Sing}(X_0) \amalg \partial X_0),  \] 
where by $\partial X_0$ we denote the degeneration of the toric boundary of the general fibers.

By the strictness condition in the statement of Theorem 
\ref{Thm: lifting md curves to log maps}, a log lift 
$\varphi_0$ of $\underline{\varphi}_0$
is completely determined away from the union of the nodes of 
$\underline{C}_0$, of the marked points $x_{0,i}$, and of 
the preimages of $\partial X_0$.

At the preimages of $\partial X_0$, toric transversality guarantees the uniqueness of the log extension, as in the proof of Proposition $7.1$ in \cite{NS}. Similarly, the marked points $x_{0,i}$ map to $\mathring{X}_0$
and so there is a unique log extension at these points.

It thus remains to investigate the nodal points. We will show that at a node $q \in C_0$, corresponding to a bounded edge $E \in \Tilde{\Gamma}^{[1]}$, if we fix local coordinates on the two components of $C_0$ meeting at $q$, there are precisely 
\[  \mu^{\RR} \coloneqq  w^{\RR}(E)  \]
pairwise non-isomorphic extensions of $\underline{\varphi}_0^*\mathcal{M}_{X_0}$ to a real log structure $\mathcal{M}_{C_0}$, log smooth over $O_0$. We denote by $e$ the integral length of $h(E)$. Recall that by definition of a good polyhedral decomposition, the weight $w(E)$ divides $e$.

By toric transversality, the node $q$ maps to the $(n-1)$-dimensional
torus orbit of the $n-1$-dimensional toric variety $X_{h(E)}\subset (X_0)_{\mathrm{sing}}$, and the intersection numbers of the branches of $C$ meeting at $q$ with $X_{h(E)}$ equal $\mu \coloneqq w(E)$. It is shown in \cite[Prop 7.1]{NS} that there exist precisely $\mu$ non-isomorphic extensions of $\underline{\varphi}_0^*\mathcal{M}_{X_0}$ to a log smooth structure in a neighbourhood of $q$, and each of these extensions differ by a well-defined $\mu$'th root of unity $\zeta \in \C^{\times}$, giving rise to $\mu$ non-isomorphic log maps $\varphi_1,\ldots,\varphi_{\mu}$, where for each such map a chart for the domain curve around the node $q$ is given by
chart
\begin{eqnarray}
\beta_{\zeta} \colon S_{e/\mu} & \lra &  \mathcal{O}_{C,q}\\
\nonumber
((a,b),c)& \longmapsto & \left\{
	\begin{array}{ll}
		(\zeta^{-1}z^a)w^b  & \mbox{if } c = 0 \\
		0 & \mbox{if } c \neq 0
	\end{array}
\right.
\end{eqnarray}
and the following relation holds
\begin{equation}
\label{Eq: Relation on ghosts}    
     s_{((1,0),0)} \cdot s_{((0,1),0)} = s_{((0,0),e/\mu)}
\end{equation}
We show that there is a real log structure $(\iota,\iota^{\flat})$ on $(C_0,M_{C_0})$ lifting the standard real structure on $C_0$ if and only if $\zeta$ is real, and that in such case this real log structure is unique. By Lemma \ref{Lem: involution on ghosts}, such real log structure $(\iota,\iota^{\flat})$ is of the form
\[\iota^{\flat}(s_{((1,0),0)})  =  \zeta^2, \,\  \iota^{\flat}(s_{((0,1),0)}) =  1, \,\ \iota^{\flat}(s_{((0,0),1))=1.}
=  1. \] 
Applying $\iota^{\flat}$ to both sides of the Equation \eqref{Eq: Relation on ghosts}, we obtain $\zeta^2=1$, hence $\zeta \in \RR$. Therefore, fixing local coordinates on branches of nodes, we obtain  
\[\prod_{E \in  \tilde{\Gamma}^{[1]}\setminus \tilde{\Gamma}^{[1]}_\infty}
w^{\R}(E)\]
many real log lifts 
$\varphi_0 \colon C_0 \rightarrow X_0$
of the real stable map $\underline{\varphi}_0 \colon 
\underline{C}_0 \rightarrow X_0$. 
It remains to show that, allowing action of reparametrizations of local coordinates, these real log lifts define 
\[\prod_{E \in \Gamma^{[1]}\setminus \Gamma^{[1]}_\infty}
w^{\R}(E) \prod_{j=1}^\ell w(E_j)\]
isomorphism classes. As in the proof of \cite[Proposition 7.1]{NS},
it is enough to consider the action of the group $\mathrm{Aut}^{\RR, \underline{\varphi}_0}(\underline{C}_0, \mathbf{x}_0)$ of real automorphisms of $\underline{C}_0$ fixing the marked points
and commuting with the map $\underline{\varphi}_0$.
As $\underline{\varphi}_0$ is an immersion, such automorphism preserves the components and fixes the nodes of $\underline{C}_0$.
As the components of $\underline{C}_0$ are projective lines, 
the action of $\mathrm{Aut}^{\RR, \underline{\varphi}_0}(\underline{C}_0, \mathbf{x}_0)$
 is non-trivial only on the components of 
 $\underline{C}_0$ containing at most two special points, that 
 is corresponding to unmarked divalent vertices of 
 $\tilde{\Gamma}$. Let $V$ be an unmarked divalent vertex of $\tilde{\Gamma}$ and let $w_V$ be the common weight of the two edges attached to $V$. In restriction to the component of
 $\underline{C}_0$ corresponding to $V$, 
 $\underline{\varphi}_0$ is a cover of 
 $\PP^1$ by $\PP^1$ of degree $w_V$ and fully ramified at two points.
 Thus, denoting $w^{\RR}_V$ the common real weight of the two edges attached to $V$, the unmarked bivalent vertex $V$ contributes a factor 
 $\ZZ/_{w^{\RR}_V \ZZ}$ to $\mathrm{Aut}^{\RR, \underline{\varphi}_0}(\underline{C}_0, \mathbf{x}_0)$.
 
Let us study the action of $\mathrm{Aut}^{\RR, \underline{\varphi}_0}(\underline{C}_0, \mathbf{x}_0)$ on the 
 $\prod_{E \in  \tilde{\Gamma}^{[1]}\setminus \tilde{\Gamma}^{[1]}_\infty}
w^{\R}(E)$ real log structures constructed above.
Let $V_1,\ldots,V_n$ be a maximal chain of unmarked bivalent vertices 
of $\tilde{\Gamma}$ and let $e_1, \ldots, e_{n-1}$ be the bounded edges of $\tilde{\Gamma}$ between $V_1$ and $V_2$, $\ldots$, $V_{n-1}$ and 
$V_n$ respectively. Let $V_0$ be a vertex of $\tilde{\Gamma}$
which is not unmarked bivalent and which is attached to the chain.
Necessarily, $V_0$ is either connected to $V_1$ or $V_n$, and up to relabelling, we assume that $V_0$ is connected to $V_1$ by an edge $e_0$.
For every $j=1,\ldots,n$, the $w^{\RR}(e_{j-1})=w^{\RR}_{V_j}$ log structures attached to the edge $E_{j-1}$ form a torsor under the factor $\ZZ/_{w^{\RR}_V \ZZ}$ of $\mathrm{Aut}^{\RR, \underline{\varphi}_0}(\underline{C}_0, \mathbf{x}_0)$ attached to the vertex $V_j$. In particular, the 
$\prod_{j=0}^{n-1} w(e_j)$ log structures attached to the edges 
$e_0, \dots, e_{n-1}$ all become isomorphic under the action of $\mathrm{Aut}^{\RR, \underline{\varphi}_0}(\underline{C}_0, \mathbf{x}_0)$.

Therefore, the number of isomorphism classes of real log structures after action of $\mathrm{Aut}^{\RR, \underline{\varphi}_0}(\underline{C}_0, \mathbf{x}_0)$ on the 
\[ \prod_{E \in  \tilde{\Gamma}^{[1]}\setminus \tilde{\Gamma}^{[1]}_\infty}
w^{\R}(E) \] 
real log structures previously constructed is obtained by replacing all the factors
$\prod_{j=0}^{n-1} w(e_j)$ by $1$.
As going from $\tilde{\Gamma}$ to $\Gamma$
subdivided by the marked points consists precisely in erasing all the chains $V_1,\ldots,V_n$ of unmarked bivalent vertices and so all
the corresponding edges $e_0,\ldots, e_{n-1}$. Hence, it follows that the number of real log maps with underlying map given by $\underline\varphi_0$ is indeed 
\[\prod_{E \in \Gamma^{[1]}\setminus \Gamma^{[1]}_\infty}
w^{\R}(E) \prod_{j=1}^\ell w(E_j)\,.\]

\section{Deformation theory for real log curves}
\label{sec:defo}
In this section we show that given a maximally degenerate real log map in the central fiber of a toric degeneration there exists a unique real stable log map obtained by a deformation of it. 
Throughout this section, for every log morphism $Y \rightarrow Z$, between two log schemes $Y$ and $Z$ we denote by $\Theta_{Y/Z}$
the associated relative log tangent sheaf. Also, for every real scheme $Y$, we denote by $Y^{\RR}$ the corresponding scheme over 
$\Spec \R$, so that $Y = Y^{\RR} \times_{\Spec \RR} \Spec \C$. Similarly,
for every real sheaf $E$ on a real scheme $Y$, we denote by $E^{\RR}$ the corresponding 
sheaf on $Y^{\RR}$, so that $E=E^{\RR} \times_{\Spec \RR} \Spec \C$.
Note that $H^0(Y^{\RR}, E^{\RR})
=H^0(Y,E)^{\RR}$, but $H^i(Y^{\RR}, E^{\RR})
\neq H^i(Y,E)^{\RR}$ in general for $i>0$.

Let $g \in \NN$, 
$\Delta \colon M \setminus \{0\} \rightarrow \NN$ with finite support, and
$\mathbf{A} = (A_1,\ldots,A_\ell)$ an affine constraint that is general for $(g,\Delta)$ in the sense of Definition \ref{Def: general affine}. 
Let $\mathscr{P}$ be a polyhedral decomposition of $M_{\Q}$
which is good for $(g,\Delta,\mathbf{A})$ in the sense of 
Definition \ref{Def: good decomposition}.
Let $\pi \colon \mathcal{X} \rightarrow \A^1$ be the toric degeneration 
defined by $\mathscr{P}$ (see \S \ref{Sec: toric degenerations}).
Let $X_t \coloneqq \pi^{-1}(t)$ be the fiber of $\pi$ over 
$t \in \A^1$. In particular, we denote by $X_0$ the central fiber of 
$\pi$. We fix $\mathbf{P}=(P_1,\ldots,P_\ell)$ a $\ell$-tuple of real points in the $n$-dimensional torus orbit of $X=X_1$. 
By \eqref{Eq:tilde_Z}-\eqref{Eq:tilde_Z_0}, the affine subspaces 
$A_j$ and the points $P_j$
define incidence conditions 
$\mathcal{Z}_{A_j,P_j} \subset \mathcal{X}$
and 
$Z_{A_j,P_j,t} = \mathcal{Z}_{A_j,P_j} \cap X_t \subset X_t\,.$ 
We denote 
$\mathbf{\mathcal{Z}}_{\mathbf{A},\mathbf{P}}
\coloneqq (\mathcal{Z}_{A_1,P_1}, \ldots, \mathcal{Z}_{A_\ell,P_\ell})$
and 
$\mathbf{Z}_{\mathbf{A},\mathbf{P},t} 
\coloneqq (Z_{A_1,P_1,t},\ldots,Z_{A_{\ell},P_{\ell},t})$. In particular, 
$\mathbf{Z}_{\mathbf{A},\mathbf{P},0}$ is a tuple of incidence conditions in the 
central fiber $X_0$.

Let $(\varphi_0 \colon C_0/O_0 \rightarrow X_0, \mathbf{x}_0)$
be a maximally degenerate $\ell$-marked real stable log map over the 
standard log point $O_0$ as in 
Theorem \ref{Thm: lifting md curves to log maps}
and matching incidence conditions 
$\mathbf{Z}_{\mathbf{A},\mathbf{P},0}$,
with associated tropical curve 
$(\Gamma,\mathbf{E},h) \in \T_{g,\ell,\Delta}(\mathbf{A})$.
Recall that $\varphi_0 \colon C_0 \rightarrow X_0$ is strict wherever $X_0 \rightarrow O_0$
is strict and that we are not at a marked point.
For every nonnegative integer $k$, we consider the thickening 
\[O_k \coloneqq \Spec \C[t]/(t^{k+1}),\] endowed with the real log structure obtained by asking that the 
inclusion $O_k \hookrightarrow \A^1$ is strict, that is, the log structure on $O_k$ is defined by the chart 
$\NN \rightarrow \cO_{O_k}$, $1 \mapsto t$. We will to study the lifts order by order in $k$ of $(\varphi_0 \colon C_0/O_0 \rightarrow X_0,\mathbf{x})$ into a $\ell$-marked real stable map 
$(\varphi_k \colon C_k/O_k \rightarrow \mathcal{X})$ over $O_k$ matching the constraints 
$\tilde{\mathbf{Z}}_{\mathbf{A},\mathbf{P}}$. Over $\C$ without taking into account real structures, an analogous study has been done in \cite{NS, nishinou2009correspondence, CFPU} using log smooth deformation theory
\cite{K,KF,Katof}.
We explain below how to incorporate the real structures in this deformation argument.

The deformation theory of $(\varphi_0 \colon C_0/O_0 \rightarrow X_0)$
is controlled by the logarithmic normal sheaf $\mathcal{N}_{\varphi_0}
\coloneq \varphi_0^{*} \Theta_{X_0/O_0} /\Theta_{C_0/O_0}$.
Real log structures on $X_0/O_0$ and 
$C_0/O_0$ induce real structures on the logarithmic tangent sheaves
$\Theta_{X_0/O_0}$ and $\Theta_{C_0/O_0}$ respectively. 
As $\varphi_0$ is a real log map, the real structure on 
$\Theta_{C_0/O_0}$ is induced by the real structure on 
$\varphi_0^{*} \Theta_{X_0/O_0}$, and so the quotient 
$\mathcal{N}_{\varphi_0}=\varphi_0^{*} \Theta_{X_0/O_0} /\Theta_{C_0/O_0}$
admits a natural real structure.

\begin{lemma} \label{lem_defo_torsor}
Let $(\varphi_{k-1} \colon C_{k-1}/O_{k-1} \rightarrow \mathcal{\cX})$ be a real stable log map over 
$O_{k-1}$ lifting
$(\varphi_0 \colon C_0/O_0 \rightarrow X_0)$. Then the set of isomorphism classes of 
real stable log maps $(\varphi_k \colon C_k/O_k \rightarrow \mathcal{X})$ 
over $O_k$ lifting $\varphi_{k-1}$ is a torsor under 
$H^0(C_0^{\RR},\mathcal{N}_{\varphi_0}^{\RR})=H^0(C_0, \mathcal{N}_{\varphi_0})^{\RR}$.
\end{lemma}

\begin{proof}
Since the log smooth deformation theory of \cite{K,KF,Katof} is valid over any base field, it suffices to essentially follow the arguments of \cite{NS, nishinou2009correspondence, CFPU}. We first prove that the set of real lifts $(\varphi_k \colon C_k/O_k \rightarrow \mathcal{X})$ of $\varphi_{k-1}$ is not empty. As $C_0^{\R}$ is of dimension one, we have 
$H^2(C_0^{\RR}, \Theta_{C_0/O_0}^{\RR})=0$, so deformations of a real log smooth curve are unobstructed and so there exists a real log smooth lift $C_k/O_k$ of $C_{k-1}/O_{k-1}$.

By \cite[Proposition 3.3]{CFPU}, whose proof uses general log smooth deformation theory and so is valid over any base field, to show the existence of a real lift 
$(\varphi_k \colon C_k/O_k \rightarrow \mathcal{X})$ of $\varphi_{k-1}$, it is enough to show that the natural map
\begin{equation} \label{eq:coh_map}
H^1(C_0^{\RR}, \Theta_{C_0/O_0}^{\RR}) \rightarrow H^1(C_0^{\RR}, \varphi_0^{*} \Theta_{X_0/O_0}^{\RR} )
\end{equation}
is surjective. As $\Gamma$ is trivalent, the map
\eqref{eq:coh_map} can be identified
by the proof of \cite[Proposition 4.2]{CFPU} with the abundancy map
\begin{eqnarray}
\label{eq:abundancy_map}
\RR^{ \sharp E(\Gamma)} & \longrightarrow & \Hom(H_1(\Gamma),M_{\RR})\\
\nonumber
 (l_E) & \longmapsto & \left(\sum_E a_E [E] \mapsto \sum_{E} a_E l_E (h(\partial^+ E)-h(\partial^-E))\right)
\end{eqnarray}
which is surjective as $(\Gamma, h)$ is non-superabundant.
This concludes the proof of existence of a lift $\varphi_k$ of 
$\varphi_{k-1}$.

The fact that the set of real lifts of $\varphi_{k-1}$ is a pseudo torsor under 
$H^0(C_0, \mathcal{N}_{\varphi_0})^{\R}$ follows from general deformation theory arguments
which are valid over any base field, see the last part of the proof of \cite[Lemma 7.2]{NS}.
\end{proof}

\begin{lemma} \label{lem_defo_unique_lift}
Let $(\varphi_{k-1} \colon C_{k-1}/O_{k-1} \rightarrow \mathcal{X},\mathbf{x}_{k-1})$
be a real stable log map over
$O_{k-1}$ lifting $(\varphi_0 \colon C_0/O_0 \rightarrow \mathcal{X},\mathbf{x}_0)$
and matching the constraints $\mathbf{\mathcal{Z}}_{\mathbf{A},\mathbf{P}}$. 
Then up to isomorphism
there exists a unique real stable log map
$(\varphi_k \colon C_k/O_k \rightarrow X,\mathbf{x}_k)$ over 
$O_k$ matching  the constraints $\mathbf{\mathcal{Z}}_{\mathbf{A},\mathbf{P}}$
and lifting $\varphi_{k-1}$.
\end{lemma}

\begin{proof}
As in the proof of \cite[Proposition 7.3]{NS}, we denote by $T_y$ the fiber at a closed point $y$ of a log tangent sheaf $\Theta$.
As the real log map $\varphi_0$ is real and the constraints $\mathcal{Z}_{A_i,P_i}$ are real, the transversality map 
\[ H^0(C_0,\cN_{\varphi_0}) \longrightarrow \prod_i 
T_{\mathcal{X}/\A^1, \varphi_0(x_i)}/(T_{\mathcal{Z}_{A_i,P_i}/\A^1,\varphi_0(x_i)} + D\varphi_0
(T_{C_0/O_0,x_i}))\]
is real.

According to \cite[Proposition 7.3]{NS}, the transversality map is an isomorphism.
Therefore, as the transversality map is real, it induces an isomorphism at the 
level of real subspaces:
\[ H^0(C_0,\cN_{\varphi_0})^{\R} \simeq \prod_i 
T^{\R}_{\mathcal{X}/\A^1, \varphi_0(x_i)}/(T^{\R}_{\mathcal{Z}_{A_i,P_i}/\A^1,\varphi_0(x_i)} + D\varphi_0
(T^{\R}_{C_0/O_0,x_i})) \,.\]

As by Lemma \ref{lem_defo_torsor} the set of real lifts $(\varphi_k \colon C_k/O_k \rightarrow \mathcal{X})$ of $\varphi_{k-1}$ is a torsor under $H^0(C_0,\cN_{\varphi_0})^{\R}$, this 
isomorphism implies Lemma \ref{lem_defo_unique_lift}
\end{proof}

\begin{theorem} \label{thm: deformation}
There exists a unique real stable log map 
$(\varphi_\infty \colon C_\infty/O_\infty \rightarrow \mathcal{X})$ over 
$O_\infty \coloneqq \Spec \C \lfor t \rfor$, matching the constraints $\mathbf{\mathcal{Z}}_{\mathbf{A},\mathbf{P}}$ and lifting
$(\varphi_0 \colon C_0/O_0 \rightarrow X_0, \mathbf{x}_0)$.
\end{theorem}

\begin{proof}
We take the limit $k \rightarrow +\infty$ of Lemma \ref{lem_defo_unique_lift} and use the existence of the 
moduli space of stable maps, as in the proof of
\cite[Corollary 7.4]{NS}.
\end{proof}

\section{The tropical correspondence theorem for real log curves}
\label{sec:The Main Theorem}
Let $X$ be a $n$-dimensional proper toric variety, defined by a complete fan $\Sigma$ in $M_{\RR}$. Fix a tuple 
$(g,\Delta,\mathbf{A},\mathbf{P})$, where:

\begin{itemize}
    \item $g$ is a nonnegative integer.
    \item $\Delta$ is a map $\Delta:M\setminus\{0\}\to \NN$ with support contained in the union of rays of $\Sigma$. The choice of $\Delta$ specifies a degree and tangency conditions along the toric divisors for a curve in $X$, and a tropical degree for a tropical curve in $M_{\RR}$.
    We denote by $|\Delta|$ the number of $v \in M\setminus \{0\}$
    with $\Delta(v) \neq 0$.
    \item  $\mathbf{A}=(A_1,\ldots,A_\ell)$ is an affine constraint that is general for $(g,\Delta)$ in the sense of Definition \ref{Def: general affine}. 
    \item  $\mathbf{P}=(P_1,\ldots, P_\ell)$ is a tuple of real points in the $n$-dimensional torus orbit of $X$.
\end{itemize}

Let $\mathscr{P}$ be a polyhedral decomposition of $M_{\Q}$
which is good for $(g,\Delta,\mathbf{A})$ in the sense of 
Definition \ref{Def: good decomposition}.
Let $\pi \colon \mathcal{X} \rightarrow \A^1$ be the toric degeneration 
defined by $\mathscr{P}$ (see \S \ref{Sec: toric degenerations}).
Let $X_t \coloneqq \pi^{-1}(t)$ be the fiber of $\pi$ over 
$t \in \A^1$. In particular, we denote by $X_0$ the central fiber of 
$\pi$. We fix $\mathbf{P}=(P_1,\ldots,P_\ell)$ a $\ell$-tuple of real points in the $n$-dimensional torus orbit of $X=X_1$. 
By \eqref{Eq:tilde_Z}-\eqref{Eq:tilde_Z_0}, the affine subspaces 
$A_j$ and the points $P_j$
define incidence conditions 
$\mathcal{Z}_{A_j,P_j} \subset \mathcal{X}$
and 
$Z_{A_j,P_j,t} = \mathcal{Z}_{A_j,P_j} \cap X_t \subset X_t$. 
We denote 
$\mathbf{\mathcal{Z}}_{\mathbf{A},\mathbf{P}}
\coloneqq (\mathcal{Z}_{A_1,P_1}, \ldots, \mathcal{Z}_{A_\ell,P_\ell})$
and 
$\mathbf{Z}_{\mathbf{A},\mathbf{P},t} 
\coloneqq (Z_{A_1,P_1,t},\ldots,Z_{A_{\ell},P_{\ell},t})$.
Because the points $P_j$ are taken to be real, the subvarieties 
$\mathcal{Z}_{A_j,P_j}$ are real subvarieties of 
$\mathcal{X}$, and for $t \in \A^1(\RR)$, 
$Z_{A_j,P_j,t}$ is a real subvariety of $X_t$.

For every $t \in \A^1(\R) \setminus \{0\} \simeq \RR^{\times}$, the \emph{number of real curves in $X_t$ of degree $\Delta$
and constrained by $\mathbf{Z}_{\mathbf{A},\mathbf{P},t}$}, which we denote by
$N_{(g,\Delta, \mathbf{A},\mathbf{P}),t}^{\RR-log}$,
is defined as follows.
Recall that the toric prime divisors on $X$ are in correspondence with primitive generators of the rays of the toric fan $\Sigma$ for $X$. 
Denote
by $D_v$ the toric divisor corresponding to $v\in M$. For a torically transverse real stable map $\varphi:C\to
X$ (Definition \ref{Def: torically transverse}), and for $\lambda\in \NN$, define $P_{\lambda}$ as the number of points of multiplicity $\lambda$ in $\varphi^* D_v$, and define a map $\Delta(\varphi):M\setminus\{0\}\to \NN$ by setting
\begin{equation}
\label{Eq: Delta-varphi}
    \Delta(\varphi)(\lambda \cdot v) = \begin{cases} 
      0 & \mathrm{if }~ \QQ_{\geq 0} v\not\in \Sigma^{[1]} \\
   P_{\lambda} &  \mathrm{othewise} \,.
   \end{cases}
\end{equation}
 So, $\Delta(\varphi)$ records the data of the number of intersection
points with the toric prime divisors of any given multiplicity, and clearly has finite support. 
Let $M^{\RR-log}_{(g,\Delta,\mathbf{A},\mathbf{P}),t}$ 
be the set of isomorphism classes of real torically transverse $\ell$-marked
stable maps $(\varphi:C\to X_t,\mathbf x)$ of genus $g$ and
matching the constraints $\mathbf{Z}_{\mathbf{A},\mathbf{P},t}$. Define
\begin{equation}
\label{Eq: real curves}    
N_{(g,\Delta, \mathbf{A},\mathbf{P}),t}^\text{$\RR$-log}
\coloneqq
\sharp\, M_{(g,\Delta,\mathbf{A},\mathbf{P}),t}^{\R-log}\,.
\end{equation}
The count of real curves  
$N_{(g,\Delta, \mathbf{A},\mathbf{P}),t}^\text{$\RR$-log}$ is a piecewise constant function of 
$t \in \A^1(\R) \setminus \{0\} \simeq \mathbb{R}^{\times}$, possibly jumping at the finitely many values of $t$ for which the 
constraints $\mathbf{Z}_{\mathbf{A},\mathbf{P},t}$ become non-generic. We denote by $t_+$ (resp.\ $t_-$) the smallest positive (resp.\ biggest negative) value of $t$ at which $N_{(g,\Delta, \mathbf{A},\mathbf{P}),t}^\text{$\RR$-log}$ jumps. In particular, $N_{(g,\Delta, \mathbf{A},\mathbf{P}),t}^\text{$\RR$-log}$ is a locally constant function of $t$ for 
 $t \in (t_-,t_+) \setminus \{0\}$.

We now define a tropical count. 
Let
$\T_{g,\ell,\Delta}(\mathbf A)$
be the set of genus $g$ $\ell$-marked tropical curves of degree $\Delta$ and  matching $\mathbf{A}$. This set is finite by Proposition 
\ref{prop_finite_tropical}.
Define
\begin{equation}
    \label{Eq: the tropical count}
     N_{(g,\Delta,\mathbf{A},\mathbf{P})}^\text{$\RR$-trop}
     \coloneqq 
\sum_{(\Gamma,\mathbf E,h)\in \T_{g,\ell,\Delta}(\mathbf A)}
w^{\RR}(\Gamma,\mathbf{E})\cdot\mathcal{D}^{\RR}_{\mathcal{T}_h,\sigma} \cdot \prod_{i=1}^\ell \mathcal{D}^{\RR}_{\mathcal{A}_i} \,,
\end{equation}
where $w^{\RR}(\Gamma,\mathbf{E})$ is the total real weight defined as in \eqref{Eq: Total real weight},  $\mathcal{D}^{\RR}_{\mathcal{T},\sigma}$ is the twisted real lattice index of the map $\mathcal{T}_h$ (see \eqref{eq:NS map}), and $\mathcal{D}^{\RR}_{\mathcal{A}_i}$ is the real lattice index of the inclusion of lattices $\mathcal{A}_i$ defined in 
\eqref{Eq: lattice marked points} -- recall that the real lattice index is defined in Definition \ref{Def: real index} and the twisted real lattice index is defined in 
Definition \ref{Def:twisted real index}. Remark that, unlike its complex analogue, the count of tropical curves 
$ N_{(g,\Delta,\mathbf{A},\mathbf{P})}^\text{$\RR$-trop}$ depends on $\mathbf{P}$ through the twist by $\sigma$ in the definition of $\mathcal{D}^{\RR}_{\mathcal{T}_h,\sigma}$.

\begin{theorem}\label{main theorem}
For every tuple $(g,\Delta, \mathbf{A}, \mathbf{P})$ as above, we have 
\begin{equation}
\label{Eq: tropical counts equalts log count}    
N_{(g,\Delta,\mathbf{A},\mathbf{P}),t}^\text{$\RR$-log}
=N_{(g,\Delta,\mathbf{A},\mathbf{P})}^\text{$\RR$-trop}
\end{equation}
for every $t \in (t_-,t_+) \setminus \{0\}$.
\end{theorem}

\begin{proof}
By Theorem \ref{Thm: counts of prelog curves}, $\mathcal{D}^{\RR}_{\mathcal{T}_h, \sigma}$ is the number of 
real maximally degenerate curves in $X_0$ intersecting 
$\mathbf{Z}_{\mathbf{A}, \mathbf{P},0}$ and with associated tropical curve 
$h \colon \Gamma \rightarrow M_\RR$.
For every $1 \leq j \leq \ell$,
the image of such real maximally degenerate curve 
$\underline{\varphi}_0 \colon \underline{C}_0 \rightarrow X_0$
intersects $\mathbf{Z}_{\mathbf{A}, \mathbf{P},0}$ in $\mathcal{D}^{\RR}_{\mathcal{A}_j}$ real points, where $\mathcal{A}_j$ is the index of the inclusion of lattices defined in Proposition \ref{Prop: number of intersection points}. Thus, 
there are $\mathcal{D}^{\RR}_{\mathcal{A}_j}$ choices of real marked points mapping to 
$\mathbf{Z}_{\mathbf{A}, \mathbf{P},0}$. Moreover, by Theorem \ref{Thm: lifting md curves to log maps}, there are $w^{\RR}(\Gamma, \mathbf{E})$ pairwise nonisomorphic ways to
lift the real stable map $\underline{\varphi}_0$ to a real stable log map $\varphi_0 \colon C_0 \rightarrow X_0$.
Therefore, 
$N_{(g,\Delta,\mathbf{A},\mathbf{P})}^\text{$\RR$-trop}$ is the number of 
real stable log maps to $X_0$ lifting maximally degenerate stable maps 
$\underline{\varphi}_0 \colon \underline{C}_0
\rightarrow X_0$
passing through the constraints 
$\mathbf{Z}_{\mathbf{A}, \mathbf{P},0}$.

By Theorem \ref{thm: deformation}, each real stable log map 
$\varphi_0 \colon C_0 \rightarrow X_0$ obtained as real log lift of a 
maximally degenerate stable map
$\underline{\varphi}_0 \colon \underline{C}_0
\rightarrow X_0$ passing through the constraints 
$\mathbf{Z}_{\mathbf{A}, \mathbf{P},0}$
deforms in a unique way in family of real stable log maps
$\varphi_{\infty} \colon C_{\infty} \rightarrow \mathcal{X}$
over $\Spec \C\lfor t \rfor$
passing through the constraints 
$\mathbf{\mathcal{Z}}_{\mathbf{A},\mathbf{P}}$.

Conversely, we have to show that every real stable log map 
$\varphi \colon C \rightarrow \mathcal{X} \times_{\C \lfor t \rfor } 
\C(\!(t)\!)$ 
defined over the point
$\Spec \C (\!(t)\!)$ (with trivial log structure)
and passing through the constraints 
$\mathbf{\mathcal{Z}}_{\mathbf{A},\mathbf{P}}\times_{\C \lfor t \rfor } 
\C(\!(t)\!)$  is obtained by deformation of a real maximally degenerate stable stable log map 
$\varphi_0 \colon C_0 \rightarrow X_0$ defined over the standard log point.
Without the reality condition, it is proved in the proof of 
\cite[Theorem 8.3]{NS}. The real case follows immediately:
a limit of real stable log maps is  real and so if 
$\varphi \colon C \rightarrow \mathcal{X} \times_{\C \lfor t \rfor } 
\C(\!(t)\!)$ is a deformation of $\varphi_0 \colon C_0 \rightarrow X_0$, 
then $\varphi_0 \colon C_0 \rightarrow X_0$ is automatically a real stable log map.

Using the general theory of basic stable log maps developed in
\cite{GSlogGW} (see Remark \ref{rem_basic}), we give a proof that $\varphi$ is a deformation of some $\varphi_0$ 
which is distinct and more conceptual that the one given in the proof of
\cite[Theorem 8.3]{NS} (which relies on the full \cite[Section 6]{NS}).
By the stable reduction theorem for basic stable log maps 
(\cite[Theorem 4.1]{GSlogGW}), after a finite base change, there exists an extension of 
$\varphi \colon C \rightarrow \mathcal{X} \times_{\C \lfor t \rfor } 
\C(\!(t)\!)$ into a basic stable log map 
$\overline{\varphi} \colon \overline{C} \rightarrow \mathcal{X}$
over $\Spec \C \lfor t \rfor$
endowed with a possibly non-trivial log structure, and passing through the constraints 
$\mathbf{\mathcal{Z}}_{\mathbf{A},\mathbf{P}}$.
By restriction to the special fiber, we get a basic stable log map 
$\varphi_0 \colon C_0 \rightarrow X_0$ over $\Spec \C$ endowed with some log structure.
Then, by the construction of a tropicalization of stable log maps 
(see \cite[Appendix B]{GSlogGW}, 
\cite[Construction 8.3]{A_corals_2}
) the tropicalization of 
$\varphi_0 \colon C_0 \rightarrow X_0$ is a $\ell$-marked tropical curve of degree 
$\Delta$ matching $\mathbf{A}$, that is an element of 
$\T_{g,\ell,\Delta}(\mathbf A)$. As by construction, the polyhedral decomposition defining our 
toric degeneration contains all the tropical curves in $\T_{g,\ell,\Delta}(\mathbf A)$, 
it follows that the stable map underlying
$\varphi_0 \colon C_0 \rightarrow X_0$
is maximally degenerate, that the basic monoid is equal to 
$\NN$ and so that $\varphi_0$ is defined over the standard log point $O_0$.
\end{proof}

\section{Tropical Welschinger signs in dimension two}
\label{Sec: tropical W signs}
In this section we set $n=2$ and we define tropical Welschinger 
multiplicities for trivalent tropical curves in $M_{\R}=\R^2$, following 
\cite{Sh0},\cite{Mi},\cite[\S2.5.1]{IMS}.

\subsection{Tropical Welschinger signs and the dual subdivision}
\label{Generalised tropical Welschinger signs and the dual subdivision}

\begin{definition}
\label{Def: dual triangle}
Let $h \colon \Gamma \to \RR^2$ be a trivalent tropical curve
and let $V$ be a vertex of $\Gamma$.
The \emph{dual triangle} $\Delta_V$ is the cell associated to $V$
in the dual subdivision of $h \colon \Gamma \to \RR^2$
(see \cite[Section 3.2]{IMS}).
Explicitly, $\Delta_V$ is the integral triangle with inner normals given by the 
direction vectors of the edges adjacent to $V$, and 
with sides of integral length equal to the weight of the corresponding edge.
\end{definition}

\begin{definition}
\label{Def:tropical W sign}
Let $h \colon \Gamma \to \RR^2$ be a trivalent tropical curve and let $V$
be a vertex of $\Gamma$. Let $\Delta_V$ be the triangle dual to $V$.
We define
\begin{equation}
    \mathrm{Mult}_\RR(V) \coloneqq (-1)^{I_{\Delta_V}} \,,
\end{equation}
where $I_{\Delta_V}$ is the number of integral points in the interior of 
$\Delta_V$.
\end{definition}

\begin{definition}
\label{Def:tropical W mult}
Let $h \colon \Gamma \to \RR^2$ be a trivalent tropical curve. We define
\begin{equation} \label{eq_mult_R_h}
\mathrm{Mult}_{\RR}(h) \coloneqq \begin{cases} 
      0 & \mathrm{if}~\Gamma~\mathrm{contains~a~bounded~edge~of~even~weight} \\
      \prod_{V \in \Gamma^{[0]}} \mathrm{Mult}_{\R}(V)  & \mathrm{else}.
   \end{cases}
\end{equation}
\end{definition}

\subsection{Tropical Welschinger signs: reformulation}

First recall that given a trivalent tropical curve $h \colon \Gamma \to \RR^2$, 
the \emph{complex vertex multiplicity} is defined as follows \cite{Mi}. For each vertex $V\in \Gamma^{[0]}$, 
consider two different edges $E_1,E_2$ emanating from $V$ and set
\begin{equation}
\label{Eq: vertex multiplicity}    
    \mathrm{Mult}(V) \coloneqq   w(E_1) \cdot w(E_2) 
\cdot |\mathrm{det}(u_{(V,E_1)}, u_{(V,E_2))})|
\end{equation}
where $w(E_i)$ is the weight on the edge $E_i$, and $u_{(V,E_i)}$ 
are primitive integral vectors emanating from $h(V)$ in the direction of $h(E_i)$. 
By the balancing condition \eqref{Eq:balancing_condition} 
this number does not depend on the choices of $E_1$ and $E_2$.

\begin{lemma} 
Let $h\colon \Gamma \to \RR^2$ be a trivalent tropical curve
and let $V$ be a vertex of $\Gamma$.
Let $E_1$, $E_2$, $E_3$ be the edges of $\Gamma$ emanating from $V$ with weights
$w(E_1)$, $w(E_2)$, $w(E_3)$. 
Then
\begin{equation}\label{eq_real_mult_reformulation}
\mathrm{Mult}_{\RR}(V)= 
      (-1)^{\frac{1}{2}(\mathrm{Mult}(V)-w(E_1)-w(E_2)-w(E_3))+1} \,.\end{equation}
\end{lemma}

\begin{proof}
As the complex vertex multiplicity $\mathrm{Mult}(V)$ defined as in \eqref{Eq: vertex multiplicity} 
is twice the area of $\Delta_V$ and $w(E_1)+w(E_2)+w(E_3)$ is the number of integral points on the boundary of $\Delta_V$, we have 
\[ I_{\Delta_V}=\frac{1}{2}(\mathrm{Mult}(V)-w(E_1)-w(E_2)-w(E_3))+1 \]
by Pick's formula. Hence, the result follows from the definition of $\mathrm{Mult}_\RR(V)$.
\end{proof}

\begin{lemma} \label{lem_odd_multiplicity}
Let $h\colon \Gamma \to \RR^2$ be a trivalent tropical curve
and let $V$ be a vertex of $\Gamma$. 
Let $E_1$, $E_2$, $E_3$ be the edges of $\Gamma$
emanating from $V$ with weights $w(E_1)$, $w(E_3)$, $w(E_3)$.
If the weights $w(E_1)$, $w(E_2)$, $w(E_3)$
are odd, then the complex multiplicity $\mathrm{Mult}(V)$ defined as in \eqref{Eq: vertex multiplicity} is also odd.
\end{lemma}
\begin{proof}
For $i\in \{1,2,3\}$, let $u_{(V,E_i)}$ be the primitive integral vector emanating from $h(V)$ in the direction $h(E_i)$.
Set
\[m \coloneq \det(u_{(V,E_1)}, u_{(V,E_2)}) \,.\]
Using the action of $SL(2,\Z)$ on $\Z^2$, we can assume that 
$u_{(V,E_1)} = (1, 0)$ and $u_{(V,E_2)} = (a,m)$ with
$a$ coprime to $m$. 

By \eqref{Eq: vertex multiplicity}, we have 
$\mathrm{Mult}(V)=w(E_1)w(E_2)|m|$. As we are assuming that $w(E_1)$
and $w(E_2)$ are odd, it is enough to show that $m$ is odd to prove that $\mathrm{Mult}(V)$ is odd. 
We assume by contradiction that $m$ is even. Then, $a$ is odd since $u_{(V,E_2)}$ is primitive. By the tropical balancing condition \eqref{Eq:balancing_condition}, we have
\[ w(E_3)u_{(V,E_3)} = -w(E_1)u_{(V,E_1)}-w(E_2)u_{(V,E_2)} = -(w(E_1) + aw(E_2),m) \,.\]
As $w(E_1)$, $w(E_2)$, $a$ are odd, we deduce that $w(E_1)+aw(E_2)$ is even, and as we are assuming that $m$ is also even, that $w(E_3)u_{(V,E_3)}$ is divisible by $2$ in $\Z^2$. But as $u_{(V,E_3)}$ is primitive in $\Z^2$, this is only possible if $w(E_3)$ is even, contradiction.
\end{proof}

\subsection{Tropical Welschinger signs \`a la Mikhalkin}
In \cite[Defn. 7.19]{Mi}, Mikhalkin only considers tropical curves with 
unbounded edges of weight $1$ and uses a slightly different 
looking version of the tropical Welschinger sign.

\begin{definition}
Let $V$ be a trivalent vertex with complex multiplicity 
$\mathrm{Mult}(V)$. Then we define
\begin{equation} \label{eq_mult_R_M}
\mathrm{Mult}_{\RR}^M(V) \coloneqq \begin{cases} 
      (-1)^{\frac{\mathrm{Mult}(V)-1}{2}} & \mathrm{if}~\mathrm{Mult}(V)~\mathrm{is~ odd} \\
      0 & \mathrm{otherwise} 
   \end{cases}
\end{equation}
\end{definition}

\begin{definition}
\label{Def:mult}
Let $h \colon \Gamma \rightarrow \R^2$ be a trivalent tropical curve. We define
\begin{equation}
\label{Eq: tropical Mikhalkin}
\mathrm{Mult}^M_{\RR}(h)=  \prod_V \mathrm{Mult}^M_{\RR}(V)   
\end{equation}
\end{definition}

\begin{lemma}
Let $h \colon \Gamma \rightarrow \R^2$ be a trivalent tropical curve.
Assume that all the weights of unbounded edges of $\Gamma$ are odd.
If the number of unbounded edges of $h$ with weights congruent to $3 \mod 4$ is even, 
then \[ \mathrm{Mult}_\RR(h) = \mathrm{Mult}^M_{\RR}(h) \,.\]
If the number of unbounded edges of $h$ with weights congruent to $3 \mod 4$ is odd, then 
\[ \mathrm{Mult}_\RR(h) =- \mathrm{Mult}^M_{\RR}(h) \,.\]
\end{lemma}

\begin{proof}
If $\Gamma$ contains a bounded edge of even weight, then 
$\mathrm{Mult}_\RR(h)=0$ by \eqref{eq_mult_R_h}, 
and $\Gamma$ contains a vertex of even complex multiplicity by 
\eqref{Eq: vertex multiplicity}, and so also $\mathrm{Mult}^M_{\RR}(h)=0$ by
\eqref{eq_mult_R_M}-\eqref{Eq: tropical Mikhalkin}.

Hence, we can assume that all
bounded edges of 
$\Gamma$ have odd weight for the remainder of the proof.
By assumption, all unbounded weights of $\Gamma$ are odd, so we are now assuming that 
all edges of $\Gamma$ have odd weights.
As every bounded edge is adjacent to two vertices whereas an unbounded is adjacent to only one vertex, we deduce from 
\eqref{eq_real_mult_reformulation} that
\[ \mathrm{Mult}_\RR(h)
=\prod_{V \in \Gamma^{[0]}} (-1)^{\frac{\mathrm{Mult}(V)}{2}+1}
\prod_{ E \in \Gamma^{[1]} \setminus \Gamma^{[1]}_\infty } (-1)^{-w(E)}
\prod_{E \in \Gamma^{[1]}_\infty} (-1)^{-\frac{w(E)}{2}} \,.\]
On the other hand, as we are assuming that all edges of $\Gamma$ have odd weight, we deduce from Lemma \ref{lem_odd_multiplicity} that all vertices of $\Gamma$ of odd complex multiplicity and so
\[ \mathrm{Mult}_\RR^M(h)=\prod_{V \in \Gamma^{[0]}} (-1)^{\frac{\mathrm{Mult}(V)-1}{2}} \]
by \eqref{eq_mult_R_M}-\eqref{Eq: tropical Mikhalkin}.
Therefore, $\mathrm{Mult}_\RR(h)$ and $\prod_V \mathrm{Mult}_\RR^M(V)$ differ by the sign given by $-1$ to the power 
\[ \frac{3}{2}\sharp \Gamma^{[0]}-\sharp (\Gamma^{[1]} \setminus 
\Gamma^{[1]}_\infty)-\frac{1}{2}\sharp \Gamma^{[1]}_\infty+\sharp \{ E \in 
\Gamma^{[1]}_\infty \,|\, w(E)=3 \mod 4\}\]
\[ = \frac{1}{2} (3\sharp \Gamma^{[0]}-2\sharp (\Gamma^{[1]} \setminus 
\Gamma^{[1]}_\infty)-\sharp \Gamma^{[1]}_\infty)
+\sharp \{ E \in 
\Gamma^{[1]}_\infty \,|\, w(E)=3 \mod 4\} \,.\]
As $\Gamma$ is a trivalent graph, we have 
\[3 \sharp \Gamma^{[0]}=2 \sharp (\Gamma^{[1]} \setminus 
\Gamma^{[1]}_\infty) + \sharp \Gamma^{[1]}_\infty \,,\]
and so $\mathrm{Mult}_\RR(h)$ and $\prod_V \mathrm{Mult}_\RR^M(V)$ differ by the sign given by $-1$ to the power 
\[ \sharp \{ E \in 
\Gamma^{[1]}_\infty \,|\, w(E)=3 \mod 4\} \,.\]
\end{proof}

\section{Log Welschinger signs and invariance in dimension two}
\label{Sec: Log W signs}

In this section, we specialize the setup of 
\S \ref{sec:The Main Theorem} to $n=2$
and zero-dimensional constraints 
$\mathbf{A}$. In 
\S \ref{Sec:nodes}, we show that for $t \in \R^{\times}$
sufficiently close to $0$ and for every
$(\varphi_t \colon C_t \rightarrow X_t)$
in $M_{(g,\Delta, \mathbf{A},\mathbf{P}),t}^{\R-log}$, all the singularities of the curve 
$\varphi_t(C_t)$ are nodes. 
In \S \ref{Sec:log W signs}, we introduce counts of $\varphi_t$
with Welschinger signs, defined by the real nature of the nodes of $\varphi_t(C_t)$.
In 
\S \ref{Sec: new nodes}-\ref{Sec: nodes preserved}, we study explicitly the real nature of the nodes of $\varphi_t(C_t)$. 
We deduce from this study 
Theorem 
\ref{Thm: log Welschinger equals tropical}
(=Theorem 
\ref{Thm: log Welschinger equals tropical intro}) 
in 
\S\ref{Sec: corr thm W signs}. 
In \S\ref{Sec:W invariants}, we discuss the relation with the symplectically defined Welschinger invariants.

\subsection{Nodes as only singularities}
\label{Sec:nodes}

We show in Proposition \ref{prop_only_nodes} below that
for $t$
sufficiently close to $0$, all the curves $\varphi_t(C_t)$ for
$(\varphi_t \colon C_t \rightarrow X_t)$
in $M_{(g,\Delta, \mathbf{A},\mathbf{P}),t}^{\R-log}$
are nodal.
We start by showing that the image curve of a line in a toric surface is nodal.

\begin{proposition} \label{prop_nodes}
Let $X$ be a complete toric surface and let 
\[ (\mathbf{u}, \mathbf{w})=((u_i)_{1\leq i \leq 3}, (w_i)_{1\leq i \leq 3}) \in M^3 \times (\NN \setminus \{0\})^3\]
with $u_i$ primitive and $\sum_{i=1}^3 w_i u_i=0$. 
Let $\varphi \colon \PP^1 \rightarrow X$ be a line in $X$
with type $(\mathbf{u}, \mathbf{w})$
(see Definition \ref{def lines}). 
Then, all the singularities of the image curve $\varphi(\PP^1)$
are nodes, that is ordinary double points, and are all contained in the $2$-dimensional torus orbit of $X$.
\end{proposition}

\begin{proof}
Without loss of generality, we can assume that $X$ is the complete
toric surface whose toric fan is given by the rays $\Q u_i$.
Let $D_1, D_2, D_3$ be the toric divisors of $X$ corresponding to 
the rays $\Q u_1, \Q u_2, \Q u_3$.

We will first prove that $\varphi(\PP^1)$ does not have unibranch singularities.
As in the proof of Lemma \ref{Lem: torsor}, one can describe a real
line $\varphi \colon \PP^1  \rightarrow X$ in $X$ of type $(\mathbf{u},
\mathbf{w})$ by 
\begin{equation} \label{eq: varphi_nodes}
\varphi^{*}(z^n)=\chi_{\varphi}(n) \prod_{i=1}^3 (y-y(q_i))^{(w_i u_i, n)}
\end{equation}
for every $n \in N$, where $\chi_{\varphi}$ is a real character of $N$, 
$y$ the unique real coordinate on $\PP^1$ such that at the points $q_1, q_2, q_3$ of $\PP^1$ 
mapped by $\varphi$ on $D_1$, $D_2$, $D_3$, we have 
$y(q_1)=-1, y(q_2)=0, y(q_3)=1$.

Applying \eqref{eq: varphi_nodes} to 
$n=\det(-,w_1 u_1)$ and using $\sum_{i=1}^3 w_i u_i=0$,
we obtain 
\[ \varphi^{*}(z^{\det(-,w_1 u_1)})
= \chi_{\varphi}(n) \left(\frac{y-1}{y} \right)^{\mu}\,,\]
where $\mu \coloneqq w_1 w_2 \det(u_1, u_2)=-w_3w_2 \det(u_3,u_2)$.
Therefore, if $\mu >0$, $\frac{d}{dy}\varphi^{*}(z^{\det(-,w_1 u_1)}) \neq 0$ for $y \neq 1$ 
and so $\varphi(\PP^1)$ does not have a unibranch singularity for 
$y \neq 1$. If $\mu <0$,  $\frac{d}{dy}\varphi^{*}(z^{\det(-,w_1 u_1)}) \neq 0$ for $y \neq 0$ 
and so $\varphi(\PP^1)$ does not have a unibranch singularity for 
$y \neq 0$.
Similarly applying 
\eqref{eq: varphi_nodes} to 
$n= \det(-,w_2 u_2)$ and $n=\det(-,w_3 u_3)$, we obtain that 
$\varphi(\PP^1)$ does not have unibranch singularities at all.

Therefore, it remains to study multibranch singularities of 
$\varphi(C)$, created by distinct points of $C$ mapped by 
$\varphi$ to the same point. By the Definition \ref{def lines} of a line, 
there is a unique point of $\PP^1$ mapped by $\varphi$ 
to each intersection point of $\varphi(\PP^1)$ with the toric divisors of $X$, and so multibranch singularities of $\varphi(\PP^1)$
are necessarily contained in the $2$-dimensional torus orbit of $X$.
Our goal is to show that all multibranch singularities of 
$\varphi(\PP^1)$ are ordinary double points, that is double points with
distinct tangent lines.

By \cite[Proposition 5.5]{NS} (or see the proof of Lemma 
\ref{Lem: torsor} for a different argument), the natural action of $M \otimes_{\Z} \C^{\times}$ 
on the space of lines in $X$ of type $(\mathbf{u}, \mathbf{w})$ is transitive. 
Therefore, it is enough to prove the result for one specific line in $X$ of type $(\mathbf{u}, \mathbf{w})$.
Let $\psi \colon \Z^2 \rightarrow \Z^2$ be the map of lattices such that 
$(1,0) \mapsto w_1 u_1$, $(0,1) \mapsto w_2 u_2$, $(-1,-1) \mapsto w_3 u_3$, 
and let $\psi_{\C} \colon (\C^{\times})^2 \mapsto (\C^{\times})^2$
be the corresponding map of complex tori, which can be naturally 
compactified in a map $\overline{\psi}_{\C} \colon \PP^2 \rightarrow X$ of toric varieties.
Let $\iota$ be the linear embedding 
$\PP^1 \rightarrow \PP^1$ of image the closure of
$\{ (x,y) \in (\C^{\times})^2\,|\,1+x+y=0\}$.
We will study the line $\varphi \colon \PP^1 \rightarrow X$
of type $(\mathbf{u},\mathbf{w})$ defined by the composition 
$\varphi \coloneqq \overline{\psi}_{\C} \circ \iota$.

We prove that multibranch singularities of $\varphi(\PP^1)$ are double points. 
We have to prove that $3$ distinct points of $\PP^1$ cannot be mapped by $\varphi$ on the same point of $X$. 
It is enough to show that there are at most two points $(x_1,y_1), (x_2,y_2) \in (\C^{\times})^2$ 
on the line of equation $1+x+y=0$ and differing multiplicatively by an element of the kernel of $\Psi_{\C}$. 
As $\Psi_{\C}$ is a monomial map, 
coordinates of elements in the kernel of $\Psi_{\C}$ are roots of unity, 
and so it is enough to show that for given $(x_1,y_1) \in (\C^{\times})^2$
on the line of equation $1+x+y=0$, there is at most another point 
$(x_2,y_2) \in (\C^{\times})^2$ on the line with 
$|x_2|=|x_1|$, $|y_2|=|y_1|$. Writing $x_2=\zeta x_1 $ and 
$y_2=\xi y_1$ with $|\zeta|=|\xi|=1$, $\zeta \neq 1$, $\xi \neq 1$, we are looking for $\zeta$ and $\xi$
such that $1+\zeta x_1 + \xi y_1=0$. We can set 
$\xi=-\frac{1}{y_1}(1+\zeta x_1)$ only if 
$|\frac{1}{y_1}(1+\zeta x_1)|^2=1$, that is 
$\frac{1}{|y_1|^2}(1+\zeta  x_1)(1+\zeta^{-1} \overline{x_1})=1$,
which is a non-trivial degree two equation in $\zeta$, so with at most two roots.
As $1+x_1+y_1=0$, $\zeta=1$ is a root and so there is at most one root $\zeta$ with $\zeta \neq 1$.

It remains to show that double points of $\varphi(\PP^1)$
have distinct tangent lines. The map 
$\Psi_{\C} \colon (\C^{\times})^2 \rightarrow (\C^{\times})^2$
is of the form $(x,y) \mapsto (x'=x^a y^b, y'=x^c y^d)$
for some $a,b,c,d \in \Z$ with $ad-bc \neq 0$.
We compute the differential of $\Psi_\C$ restricted to the line $1+x+y=0$. We have 
\[ dx'= ax^{a-1}y^b dx+bx^a y^{b-1}dy
=\frac{a}{x} x' dx+\frac{b}{y} x' dy \]
\[ dy'= cx^{c-1}y^d dx+dx^c y^{d-1}dy
=\frac{c}{x} y' dx+\frac{d}{y} y' dy \,.\]
As $1+x+y=0$, we have $dx=-dy$ so 
\[ dx'
=\left( \frac{a}{x} -\frac{b}{y}\right) x' dx \]
\[ dy'
=\left( \frac{c}{x} -\frac{d}{y}\right) y' dx \,.\]
Therefore, the equation of the tangent at the point $(x',y')=\Psi_{\C}(x,y)$
is 
\[ \left( \frac{c}{x} -\frac{d}{y}\right) y' dx' + \left( \frac{a}{x} -\frac{b}{y}\right)x' dy'=0\,.\]
Let us assume that we have a double point formed by two distinct points 
$(x_1,y_1)$ and $(x_2,y_2)$ on the line $1+x+y=0$
mapped by $\Psi_{\C}$ to the same point 
$(x_1',y_1')=(x_2',y_2')$. If the tangent lines to 
$\varphi(\PP^1)$ at the nodes along the two branches coincide, then 
\[ \frac{ay_1-bx_1}{cy_1-dx_1}=\frac{ay_2-bx_2}{cy_2-dx_2}\,.\]
It follows that $(ad-bc)(x_1 y_2 -x_2 y_1)=0$, so 
 $x_1 y_2 =x_2 y_1$. As $y_2=-1-x_2$ and 
 $y_1=-1-x_1$, we obtain $x_1(1+x_2)=x_2(1+x_1)$ and so 
 $x_1=x_2$, and $y_1=y_2$, in contradiction with the assumption that 
 $(x_1,y_1)$ and $(x_2,y_2)$ are distinct.
\end{proof}

\begin{proposition}
\label{prop_only_nodes}
For every $t \in \A^1 \setminus \{0\}$
sufficiently close to $0$
and for every
$(\varphi_t \colon C_t \rightarrow X_t)$
in $M_{(g,\Delta, \mathbf{A},\mathbf{P}),t}^{\R-log}$,
the only singularities of the curves 
$\varphi_t(C_t)$ are nodes.
\end{proposition}

\begin{proof}
The map $(\varphi_t \colon C_t \rightarrow X_t)$ is 
a deformation of $(\varphi_0 \colon C_0 
\rightarrow X_0)$, so $\varphi_t(C_t)$ is a 
deformation of $\varphi_0(C_0)$.
As $\varphi_0$ is maximally degenerate, it follows from
Proposition \ref{prop_nodes}
that $\varphi_0(C_0)$ is a nodal curve. Therefore, 
$\varphi_t(C_t)$ is nodal for $t$ sufficiently close to $0$. 
As $M_{(g,\Delta, \mathbf{A},\mathbf{P}),t}^{\R-log}$ is finite, this holds for all $\varphi_t$ for $t$
sufficiently close to $0$.
\end{proof}

\subsection{Log Welschinger signs}
\label{Sec:log W signs}

We first recall the classification of possible types of nodes of a real nodal curve. These nodes can be of one of the following possible types.
\begin{itemize}
    \item[(i)] \emph{Elliptic (isolated) nodes:} These are real nodes $x^2 + y^2 = 0$ for a choice of local real coordinates $(x, y)$.
    \item[(ii)]  \emph{Hyperbolic (non-isolated) nodes:} These are real nodes, with local equation $x^2 - y^2 = 0$ for a choice of local real coordinates $(x, y)$.
    \item[(iii)]  \emph{Imaginary nodes:} These are nodes of $C$ that are at non-real points. Such nodes come in complex conjugate pairs.
\end{itemize}

\begin{definition} 
\label{Def: log Welschinger sign}
Let $\varphi \colon C \to X$ be a real stable log map in a toric surface $X$ with $\varphi(C)$ nodal.
The \emph{log Welschinger sign} of $\varphi$ is
\begin{equation}
    \label{Eq: mass log}
    \mathcal{W}^{\mathrm{log}}(\varphi) \coloneqq (-1)^{m(\varphi)}
\end{equation}
where $m(\varphi)$ is the total number of elliptic (isolated) nodes of $\varphi(C)$.
\end{definition}

Let $\varphi_t \colon C_t \rightarrow X_t$ be an element of $M_{(g,\Delta,\mathbf{A},\mathbf{P}),t}^{\R-log}$. 
By Proposition \ref{prop_only_nodes}, for 
$t \in \A^1(\R) \setminus \{0\} \simeq \R^{\times}$
sufficiently close to $0$, $\varphi_t(C_t)$ is a real nodal curve. 
We will study in the next sections the real types of these nodes. 
As $\varphi_t \colon C_t \rightarrow X_t$
is a deformation of $\varphi_0 \colon C_0 
\rightarrow X_0$, the nodes of $\varphi_t(C_t)$
 fall into one of the two following cases, see Figure 
 \ref{Fig: Nodes}.
\begin{itemize}
    \item[(i)] Nodes obtained by local isomorphic deformations of nodes of $\varphi_0(C_0)$ contained in the smooth locus of $X_0$.
    \item[(ii)] Nodes obtained by the non-trivial deformation of a node of $\varphi_0(C_0)$ contained in the double locus of $X_0$.
\end{itemize}

We study nodes of type (ii) in 
\S \ref{Sec: new nodes} and nodes of type (i) in 
\S \ref{Sec: nodes preserved}.

\subsection{New nodes generated during the deformation}
\label{Sec: new nodes}
Let $\varphi_0 \colon C_0 \to X_0$ be a maximally degenerate real stable log map, with associated tropical curve $(h \colon \Gamma 
\rightarrow M_{\R} \simeq \R^2) \in \T_{g,\ell,\Delta}(\mathbf A)$.
Let $(\varphi_t \colon C_t \rightarrow X_t) 
\in M^{\R-log}_{(g,\Delta,\mathbf{A},\mathbf{P}),t}$
be the deformation of $\varphi_0$ for 
$t \in \A^1(\R)\setminus \{0\} \simeq \R^{\times}$.
We focus on a node $P$ of $C_0$, corresponding to an edge 
$E$ of $\Gamma$ of weight 
$\mu \coloneqq w(E)$, and on the node of $\varphi_t(C_t)$
which are contracted to $\varphi_0(P)$ at $t=0$.
There are $\mu-1$ such nodes, corresponding to the 
$\mu -1$ interior integral points on the edge dual to $E$
in the subdivision dual to $\Gamma$. 
In this section, we analyze the real type of these nodes. 
The following Lemma will allow us in the proof of Theorem 
\ref{thm_signs_edges} to reduce the general case to a specific situation which can be handled explicitly.

\begin{lemma} \label{lem:blow_up}
Let $\varphi_0 \colon C_0 \to X_0$ be a real maximally degenerate stable log map with associated tropical curve 
$(h\colon \Gamma \to M_{\RR}\cong \RR^2) \in \T_{g,\ell,\Delta}(\mathbf A)$, and let $P\in C_0$ be a nodal point corresponding to an edge $E \in \Gamma^{[1]}$, 
of weight $\mu \coloneqq w(E)$. Denote the vertices adjacent to $E$ by $\partial E=\{ V_1,V_2 \}$, 
and let $\{E,E_i,E_i'\}$ be the set of edges adjacent to $V_i$. 
Then, by a suitable base change, a  blow-up
and a further degeneration of $X_0$, we can assume that
\[ u_{V_1,E_1}=(-1,0) \,, u_{V_1,E_1'}=(1,\mu) \,, 
u_{V_2,E_2}=(-1,0)\,, u_{V_2,E_2'}=(1,-\mu)\,,\]
as illustrated in Figure \ref{Fig: tropical node}.
\end{lemma}

\begin{figure} 
\center{\input{AfterShift.pspdftex}}
\caption{The tropical image around a nodal point after a suitable blow-up and base-change.}
\label{Fig: tropical node}
\end{figure}

\begin{proof}
Since $\varphi_0$ is torically transverse (Definition \ref{Def: torically transverse}), 
the image $h(E)$ of the edge $E$ is a segment of the polyhedral decomposition $\mathscr{P}$. 
Since the question is local near $\varphi_0(P)$ in $X_0$ and near 
$h(E)$ in $M_{\R}$, we can assume by a local rescaling of $M$ that the affine integral length of $h(E)$ equals $\mu$. 
Let $X_1$ and $X_2$ be the irreducible components of the central fiber of the toric degeneration corresponding to the vertices $h(V_1)$ and $h(V_2)$ of the polyhedral 
decomposition. The surfaces $X_1$ and $X_2$ intersect along a divisor $X_E$ corresponding to the 
edge $h(E)$ of the polyhedral decomposition. Let $C_1$ and $C_2$ be the two components of $C$ mapped by $\varphi$ to $X_1$ and $X_2$ respectively. 
Let $p \in C$ be node of $C$ given by the intersection point of $C_1$ and $C_2$. We have $\varphi(p) \in X_E$.

\'Etale locally near $\varphi(p)$, the total space $\mathcal{X}$ of the toric degeneration can be described as 
$\Spec \C[x,y,\gamma,t]/(xy-t^\mu)$, 
where $\varphi(p)$ is the point 
$x=y=t=\gamma=0$, $\gamma$ is a coordinate along $X_E$, $X_1$ is locally defined by $x=0$ and $X_2$ is locally defined by $y=0$, 
such that $h(C_1)$ is locally defined by an equation of the form $x=\gamma^\mu$ and $h(C_2)$ is locally defined by an equation of the form $y=\gamma^\mu$. 
Note that because $\gamma$ is centered around the point $\varphi(p)$, which is in the $1$-dimensional torus orbit of $X_E$,
$\gamma$ is not induced by a toric monomial.

We perform the base change $t=s^{2}$. The resulting total space can be locally described as 
$\Spec \C[x,y,\gamma,s]/(xy-s^{2 \mu})$.
We then blow-up the ideal 
\[ (x,y, s^{\mu}, s^{\mu-1} \gamma, \ldots, s \gamma^{\mu-1}, \gamma^\mu)\] of $\Spec \C[x,y,\gamma,s]/(xy-s^{2 \mu})$, 
that is a non-reduced version of the point $\varphi(p)$ defined by 
$x=y=\gamma=s=0$.
We denote by $\tilde{\mathcal{X}}$ the resulting total space, and $\pi \colon 
\tilde{\mathcal{X}} \rightarrow \mathcal{X}$ the blow-up morphism composed with the base change morphism. Note that because 
$\varphi(p)$ is not a torus fixed point of $\mathcal{X}$, $\tilde{\mathcal{X}}$ is not toric in general.

We show that the exceptional divisor $S$ is isomorphic to the toric surface whose 
fan is dual to the triangle with vertices $(0,1)$, $(0,-1)$, $(0,\mu)$. As the ideal  
$(x,y, s^{\mu}, s^{\mu-1} \gamma, \ldots, s \gamma^{\mu-1}, \gamma^\mu)$ of $\Spec \C[x,y,\gamma,s]/(xy-s^{2 \mu})$ is monomial, the blow-up
can be studied using the tools of toric geometry. The moment cone of 
$\Spec \C[x,y,\gamma,s]/(xy-s^{2 \mu})$ is isomorphic to the cone $\mathcal{C}$ in $\Z^3$
generated by $(\mu,1,0)$, $(\mu,-1,0)$, $(0,0,1)$, with $(\mu,1,0)$
corresponding to $x$, $(\mu,-1,0)$ corresponding to $y$, 
$(0,0,1)$ corresponding to $\gamma$ and 
$(1,0,0)$ corresponding to $s$. The linear relation 
\[ (\mu,1,0) + (\mu,-1,0)=2\mu(1,0,0)\]
corresponds to the relation $xy=s^{2 \mu}$.
The monomials $x,y, s^{\mu}, s^{\mu-1} \gamma, \ldots, s \gamma^{\mu-1}, \gamma^\mu$ exactly correspond to the integral points 
of the intersection of $\mathcal{cC}$ with the integral plane $P$
in $\Z^3$
passing through the three points $(\mu,1,0)$, $(\mu,-1,0)$, $(0,0,1)$.
Therefore, by standard toric geometry, blowing-up the ideal 
$(x,y, s^{\mu}, s^{\mu-1} \gamma, \ldots, s \gamma^{\mu-1}, \gamma^\mu)$ corresponds to chopping off $\mathcal{C}$
along $\mathcal{C} \cap P$, and the cone of the exceptional divisor is given by the triangle $\mathcal{C} \cap P$ inside $P$, see Figure \ref{figure_blow_up}. 
The map $(\mu,1,0) \mapsto (0,1)$, 
$(\mu,-1,0) \mapsto (0,-1)$, $(0,0,\mu) \mapsto (\mu,0)$
induces an affine integral isomorphism between 
$\mathcal{C} \cap P$ in $P$ and the triangle with vertices 
$(0,1), (0,-1), (0,\mu)$ in $\Z^2$.

\begin{figure} \label{figure_blow_up}
\center{\input{J.pspdftex}}
\caption{Toric polytope description of the blow-up}
\label{Fig: Fig1}
\end{figure}

It follows that the fan of $S$ consists of the three rays generated by $(-1,0)$, $(1,\mu)$, and $(1,-\mu)$. 
Let $D_1$ be the toric divisor of $S$ dual to the ray generated by $(1,\mu)$ and 
let $D_2$ be the toric divisor of $S$ dual to the ray generated by $(1,-\mu)$.
If we continue to denote by $X_1$ and $X_2$ the strict transforms of $X_1$
and $X_2$ in $\tilde{\mathcal{X}}$, then $X_1 \cap S=D_1$ and $X_2 \cap S=D_2$.

As $\pi$ is not a toric blowup, $S$ does not come naturally with a toric structure: 
there are natural divisors $D_1$ and $D_2$ but no canonical choice of divisor $D_3$
dual to the ray generated by $(-1,0)$. 
The blown-up geometry 
$\tilde{\mathcal{X}}$ comes with a natural log smooth log structure over $\A^1$.
The corresponding central fiber $\tilde{X}_0$ is log smooth over the standard log 
point. In restriction to $S$, the morphism to the standard log point is strict away from 
$D_1 \cup D_2$.

The stable log map $\varphi \colon C \rightarrow X_0$ lifts to a stable log map 
$\tilde{\varphi} \colon \tilde{C} \rightarrow \tilde{X}_0$, 
where $\tilde{C}$ is obtained from $C$ by replacing the node $p$ by a new $\PP^1$-component $C_S$, meeting the rest of $\tilde{C}$ in two nodes $p_1$, $p_2$, and 
$\tilde{\varphi}(C_S) \subset S$,
$\tilde{\varphi}(p_1) \in D_1$, 
$\tilde{\varphi}(p_2) \in D_2$. 
By construction of the blow-up,
the curve $\tilde{\varphi}(C_S)$ has contact orders $1$ with $D_1$ and $D_2$
at the points $\tilde{\varphi}(p_1)$ and $\tilde{\varphi}(p_2)$,
and so the smoothing of the two nodes $p_1$ and 
$p_2$ of $\tilde{C}$ does not create new nodes in the image curve.
Therefore, the nodes of the image curve coming from the smoothing of the node $p$
of the curve $C$ in $X_0$ can be identified with the nodes of the image curve 
$\tilde{\varphi}(C_S)$ in $S \subset \tilde{X}_0$.

It follows that the study of the real nature of the nodes of the image curve 
coming from the smoothing of the node $p$ of $C$ is reduced to the study of the real nature of the nodes of the image of $\tilde{\varphi}(C_S)$.
In order to do that, we 
complete $(S,D_1 \cup D_2)$ into a general toric structure 
$(S,D_1 \cup D_2 \cup D_3)$ and we pick a toric degeneration of $S$
defined by the polyhedral decomposition given by a tropical curve $h_S \colon
\Gamma_S \rightarrow M_\RR$
with two vertices $V_1$, $V_2$, one bounded edge of weight $\mu$ connecting $V_1$, 
$V_2$, two unbounded edges of directions $(-1,0)$ and $(1,\mu)$ with weight $1$
attached to $V_1$, and two unbounded edges of directions 
$(-1,0)$ and $(1,-\mu)$ with weight $1$ attached to $V_2$. In such degeneration,
the curve $\tilde{\varphi}|_{C_S} \colon C_S \rightarrow S$ degenerates into 
a curve with tropicalization $h_S \colon
\Gamma_S \rightarrow M_\RR$. Therefore, we have reduced the study of the nodes 
of the image curve produced by the smoothing of a node corresponding to an edge of weight $\mu$ in a 
general tropical curve to the study of the nodes of the image curve produced by the smoothing 
of a node corresponding to the edge of weight $\mu$ in the tropical curve 
$h_S \colon
\Gamma_S \rightarrow M_\RR$, and this ends the proof of 
Lemma \ref{lem:blow_up}.
\end{proof}

We remark that the proof of Lemma \ref{lem:blow_up}
is a version of the shift operation 
introduced by Shustin in \cite[section 3.5]{Sh0}.
The reformulation of the shift operation in terms of
blow-up is also discussed by Shustin and Tyomkin in
\cite{ShT1, ShT2}.

\begin{theorem} \label{thm_signs_edges}
Let $\varphi_0 \colon C_0 \to X_0$ be a real maximally degenerate stable log map with associated tropical curve 
$(h\colon \Gamma \to M_{\RR}\cong \RR^2) \in \T_{g,\ell,\Delta}(\mathbf A)$.
Let $P \in C_0$ be a node corresponding to an edge $E \in \Gamma^{[1]}$
of weight $\mu \coloneqq w(E)$, and let $e$ be the integral affine length of $h(E)$.
Let $\zeta$ be the real root of unity specifying the real log structure of $C_0$ at $P$.
Then, the real types of the $\mu-1$ nodes in the image 
$\varphi(C)$ of the deformation $\varphi:C\to X$ of $\varphi_0$
over $t \in \A^1(\R) \setminus \{0\} \simeq \R^{\times}$ and which have limit $\varphi_0(P)$ at $t=0$, are as follows:
\begin{itemize}
    \item If $\mu$ is odd (and so necessarily $\zeta=1$), then all the $\mu-1$ nodes of $\varphi (C)$ are elliptic.
    \item If $\mu$ is even and $\zeta$ and $t^{\frac{e}{\mu}}$ have the same sign, then all the $\mu-1$ created nodes of 
    $\varphi (C)$ are elliptic.
    \item If $\mu$ is even and $\zeta$ and $t^{\frac{e}{\mu}}$ have opposite signs, 
   then one of the nodes of 
   $\varphi (C)$ is  hyperbolic and the $\mu-2$ others form $\frac{\mu-2}{2}$ pairs of complex conjugated nodes.
\end{itemize}
\end{theorem}

\begin{proof}
By Lemma \ref{lem:blow_up}, we can assume that the tropical curve $h \colon \Gamma
\rightarrow \R^2$ is given locally near the edge $E$ by Figure \ref{Fig: tropical node}.
Denote by $V_1,V_2\in \Gamma^{[0]}$ the vertices adjacent to $E$, and let 
$C_1$ and $C_2$ be the corresponding components of $C_0$, meeting at the node $P$. Let $X_1,X_2\subset X_0$ be the irreducible components of $X_0$ corresponding to 
$h(V_1)$ and $h(V_2)$. 
Denote by $X_E$ the double locus in $X_0$, where $X_1$ and $X_2$ intersect.

A chart for the log structure $\mathcal{M}_{X_0}$, 
on an open neighbourhood $U$ around the point $\varphi_0(P)$ is given by
\begin{eqnarray}
\label{Eq: chart around varphi(P)}
 \overline{\M}_{X_0}|_U & \lra &  \M_{X_0}|_U\\
\nonumber
\overline{x} & \longmapsto & s_x\\
\nonumber
\overline{y} & \longmapsto & s_y\\
\nonumber
\overline{\gamma} & \longmapsto & s_{\gamma}\\
\nonumber
\overline{t} & \longmapsto & s_t\\
\nonumber
\end{eqnarray}
where $\overline{x} \cdot \overline{y} = \overline{t}^e$. 
The coordinates $x,y,\gamma$ on $U$ are such that 
$\gamma$ is a toric coordinate along $X_E$ (unlike the 
$\gamma$ coordinate used in the proof of Lemma \ref{lem:blow_up}), $X_1$ is locally defined by $x=0$ and $X_2$ is locally defined by $y=0$, 

Now, before analysing the log structure on the domain curve around $C$, recall by Definition \ref{Def: prelog}, 
the intersection index of either $C_1$ or $C_2$ with $X_E$ is equal to $\mu$. 
From the proof of Theorem \ref{Thm: lifting md curves to log maps} it follows that there exists a chart for the log structure $\mathcal{M}_C$ on a neighbourhood $V$ around $P$
given by
\begin{eqnarray}
\label{Eq: chart around P}
 \overline{\M}_C|_V & \lra &  \M_C|_V\\
\nonumber
\overline{z} & \longmapsto & \zeta^{-1}s_z\\
\nonumber
\overline{w} & \longmapsto & s_w\\
\nonumber
\overline{t} & \longmapsto & s_t^{e/\mu}\\
\nonumber
\end{eqnarray}
where $\overline{z} \cdot \overline{w} = \overline{t}^{e/\mu}$. 
Moreover, $\zeta=1$ if $\mu$ is odd, and $\zeta \in \{ 1,-1 \}$ if $\mu$ is even. 
The coordinates on $V$ are given by $\{ z, w, t | z \cdot w = t^{e/ \mu} \}$. 
Now, in the remaining part of this section we assume that the charts for the log structures around $P$ 
and $\varphi_0(P)$ are given as in \eqref{Eq: chart around P} and \eqref{Eq: chart around varphi(P)} respectively.

Label the points $\{ 0_{C_i},1_{C_i},\infty_{C_i} \}$ on each irreducible component $C_i$ mapping to the toric boundary of $X_i$, 
and assume that $P= 0_{C_1}=0_{C_2}$. Let
\begin{eqnarray}
\nonumber
 \varphi_0(1_{C_1}) & = & p_1 \coloneqq (x=c,y=0,\gamma=0) \in X_1 \\
\nonumber
 \varphi_0(1_{C_2}) & = & p_2 \coloneqq (x=0,y=c',\gamma=0) \in X_2 \\
\varphi_0(P) & = & p_0 \coloneqq (x=0,y=0,\gamma=d) \in X_E \\
\nonumber
\end{eqnarray}
for some $c,c',d \in \RR^{\times}$. 
The data of the tropical curve $h \colon \Gamma \to \RR^2$ given by 
Figure \ref{Fig: tropical node} 
defines the morphisms $\overline{\M}_{X_0,\varphi_0(x)} \to \overline{\M}_{C_0,x}$ on ghost sheaf level for any closed point $x\in C_0$ \cite[\S8]{A_corals_2}.
Moreover, by the definition of the log structure on $X_0$, the sections of $\M_{X_0,\varphi_0(x)}$ are obtained by restrictions of regular functions in \[\M_{\mathcal{X},\varphi_0(P)} \subset \mathcal{O}_{\mathcal{X},\varphi_0(P)},\]
where we recall that $\mathcal{X}$
denotes the total space of degeneration. 
Thus, determining lifts to log morphisms amounts to specifying the values $c,c',d \in \RR^{\times}$, which define the morphism
\begin{eqnarray}
 \label{Eq: the log morphism}
 \mathcal{M}_{\mathcal{X},\varphi_0(P)} & \lra & \mathcal{M}_{C_0,P} \\
\nonumber
x & \longmapsto & c z^{\mu}  \\
\nonumber
y & \longmapsto & c' w^{\mu}  \\
\nonumber
\gamma & \longmapsto & d (z + w-1)  \\
\nonumber
t & \longmapsto &  t  \\
\nonumber
\end{eqnarray}
Since on a neighbourhood of $P$, the coordinates satisfy the relation
\[ \{ z,w,t~|~z\cdot w=t^{e/\mu}  \}  \]
and on a neighbourhood of $\varphi_0(P)$ on the total space we have the relation
\[ \{ x,y,\gamma,t ~|~x\cdot y=t^e \}  \]
we obtain $c'=c^{-1}$. Then, we indeed get a well-defined map in \eqref{Eq: the log morphism} because
\begin{equation}
    \label{Eq: relations}
    xy = cz^\mu c^{-1}w^\mu
=(zw)^\mu=(t^{e/\mu})^\mu=t^e.
\end{equation}

For fixed $t\neq 0$, from \eqref{Eq: relations} we get
\[ y=\frac{t^e}{x} \,\ \,\ \mathrm{and} \,\ \,\  w=\zeta \frac{t^{e/\mu}}{z}, \]
and obtain a parametrization of the deformation of $\varphi_0(C_0)$ in a neighbourhood of $P$, given by
\begin{eqnarray}
 \nonumber
\C^{\times} & \longrightarrow & (\C^{\times})^2 \\
\nonumber
z & \longmapsto & (x(z),\gamma(z))
\end{eqnarray}
with
\begin{eqnarray}
 \label{Eq: parametrization}
x(z) & = & cz^{\mu} \\
\nonumber
\gamma(z) & = & d \left(z+\zeta \frac{t^{e/\mu}}{z}-1\right) \,.
\end{eqnarray}
The nodal points in the image of the deformation are generated as the images of points $z_1,z_2 \in \C^{\times}$, with $z_1 \neq z_2$, and
\[(x (z_1),\gamma (z_1))
=(x (z_2),\gamma (z_2))\,.\] 
From the parametrization \eqref{Eq: parametrization}, these points are obtained as solutions of the following equations:
\begin{eqnarray}
 \label{Eq:first eqn}
 z_1^{\mu} & = & z_2^{\mu} \\
 \label{Eq: second eqn}
 z_1+\zeta \frac{t^{e/\mu}}{z_1}
& = & z_2+\zeta \frac{t^{e/\mu}}{z_2} 
\end{eqnarray}
From \eqref{Eq:first eqn}, since $z_1 \neq z_2$, we obtain $z_2=\xi z_1$ for some 
$\xi \in \C$ such that 
$\xi^{\mu}=1$, $\xi \neq 1$.
Then \eqref{Eq: second eqn} gives
\begin{eqnarray}
 \nonumber
 z_1+\zeta \frac{t^{e/\mu}}{z_1} & = & \xi z_1+\zeta \frac{t^{e/\mu}}{\xi z_1} \\
 \nonumber
 z_1^2+\zeta t^{e/\mu} & = & \xi z_1^2+\frac{\zeta}{\xi} t^{e/\mu} \\
 \nonumber
 z_1^2(1-\xi) & = & \zeta t^{e/\mu}(\frac{1}{\xi}-1) \\
 \nonumber
 z_1^2 & = & \zeta \frac{t^{e/\mu}}{\xi}
\end{eqnarray}
It follows that $z_2^2=(\xi z_1)^2=\xi \zeta t^{e/\mu}$.

We have $(\mu-1)$ choices for $\xi$ and two choices of square root of $\zeta \frac{t^{e/\mu}}{\xi}$.
Because of the symmetry between 
$z_1$ and $z_2$, we finally get 
$\mu-1$ nodes in the image of the curve.

We now study the real nature of these nodes. We will use the following facts.
Assume that $z_1$ and $z_2$ are such that $(x (z_1),\gamma (z_1))
=(x (z_2),\gamma (z_2))$ and so define a node in the image curve.
 If $z_1,z_2 \in \R$, then we get an hyperbolic node.
If $z_1 \notin \R$, 
$z_2 \notin \R$, and $(x(z_1),\gamma(z_1))\in \R^2$, then we get an elliptic node.
If $z_1 \notin \R$, 
$z_2 \notin \R$, and $(x(z_1),\gamma(z_1))\notin \R^2$, then we get an imaginary node.

As $\zeta$ and $t^{\frac{e}{\mu}}$ only enter the above equation through the combination $\zeta t^{\frac{e}{\mu}}$, 
flipping the sign of $\zeta$ is equivalent to flipping the sign of $t^{\frac{e}{\mu}}$, and so we can assume without loss of 
generality that $t^{\frac{e}{\mu}} \in \R_{>0}$.

Assume $\zeta=1$. We then have $z_1^2=\frac{t^{\frac{e}{\mu}}}{\xi}$
and so
$|z_1|^2=t^{\frac{e}{\mu}}$. It follows that 
\[ \frac{t^{\frac{e}{\mu}}}{z_1}=\frac{|z_1|^2}{z_1}= \overline{z}_1\]
and so 
\[\gamma(z_1)=d(z_1+\overline{z}_1
-1)=d(2 \Rea (z_1) -1) \in \R\,.\]
On the other hand, we have 
$z_1^{2\mu}=t^e$, so $z_1^{\mu}=\pm t^{\frac{e}{2}}$
and $\alpha(z_1)=\pm c t^{\frac{e}{2}} \in \R$.
The upshot is that the node is elliptic.

Assume $\zeta=-1$.
We then have $z_1^2=-\frac{t^{\frac{e}{\mu}}}{\xi}$ and so $|z_1|^2=t^{\frac{e}{\mu}}$. It follows that 
\[ \frac{t^{\frac{e}{\mu}}}{z_1}=\frac{|z_1|^2}{z_1}=\overline{z}_1 \]
and so 
\[ \gamma (z_1)=d(z_1-\overline{z}_1-1) =d(2i \Ima z_1 -1)\,.\]
If $\xi \notin \R$, which is always the case if $\mu$ is odd, then 
$z_1 \notin  \R$, so $\gamma(z_1)
\notin \R$ and therefore the node if 
imaginary.
If $\xi \in \R$, which is only the case if $\mu$ is even and then 
$\xi=-1$, we have $z_1^2=t^{\frac{e}{\mu}}$
and $z_2^2=t^{\frac{e}{\mu}}$, so $z_1,z_2 \in \R$, and 
therefore the node is hyperbolic.

We can now summarize the results. If $\mu$ is odd, then $\zeta=1$ and all the 
$(\mu -1)$ nodes are elliptic. If $\mu$ is even and $\zeta=1$, then all the 
$(\mu-1)$ nodes are elliptic. If $\mu$ is even and $\zeta=-1$, then the node corresponding to $\xi=-1$ is hyperbolic, and the $\mu-2$ other nodes 
corresponding to $\xi \neq -1$ are imaginary nodes.
\end{proof}

\subsection{Nodes which are preserved under the deformation.} 
\label{Sec: nodes preserved}
In this section we analyse the types of nodes that appear in the  image of a maximally degenerate real stable log map 
$\varphi_0 \colon C_0\to X_0$ and which are contained in the smooth locus of $X_0$. 
They will deform into locally isomorphic nodes in the deformation $\varphi_t \colon C_t \rightarrow X_t$.

Recall that nodal points on the image of $\varphi_0$ are generated, 
if two points of $C_0$, say $p_1,p_2$ map to the point under the map $\varphi_0: C_0\to X_0$. There are two possible cases:
\begin{itemize}
\label{items: possible nodes}
    \item[(i)] The points $p_1,p_2$ are points  on different irreducible components of $C_0$.
    \item[(ii)] The points $p_1,p_2$ belong to the same irreducible component of $C_0$.
\end{itemize}
We focus on nodes of $\varphi_0(C_0)$ contained in the smooth locus of $X_0$, and we will analyse the types of the nodes in each case separately. 
The following lemma shows that in the first case we only obtain hyperbolic nodes.

\begin{lemma} \label{lem:4-valent}
Let $\varphi_0 \colon C_0 \to X_0$ be a real maximally degenerate stable log map with associated tropical curve 
$(h\colon \Gamma \to M_{\RR}\cong \RR^2) \in \T_{g,\ell,\Delta}(\mathbf A)$.
Every node in the image of 
$\varphi_0(C_0)$ that is contained in the smooth locus of $X_0$ and that is obtained from 
identification by $\varphi_0$ of two points lying on different irreducible 
components of $C_0$ is hyperbolic.
\end{lemma}

\begin{proof}
Let $p_1,p_2\in C_0$ be two distinct points mapping by $\varphi_0$
to the same point $\varphi_0(p_1)=\varphi_0(p_2)$ contained in the smooth locus of 
$X_0$. 
Assume that $p_1$ and $p_2$ lie on different irreducible components
$C_{0,1}$ and $C_{0,2}$ of $C_0$.
As $\varphi_0$ is a real map, this can only happen if either $p_1$ and $p_2$ are complex conjugated, in which case $p$ is elliptic, or if 
$p_1$ and $p_2$ are real, in which case $p$ is hyperbolic. We will show that the latter case is realized.

We denote the vertices of 
$\Gamma$ corresponding to $C_{0,1}$ and $C_{0,2}$ by $V_1$ and $V_2$ respectively, 
and let $X_p$ be the irreducible component of $X_0$ corresponding to the vertex $h(V_1)=h(V_2)$ in the image of $\varphi_0$. 
Since $h$ is a general tropical curve, two vertices $V_1,V_2\in \Gamma$ can map to the same point transversally only if they are bivalent vertices, 
and create a $4$-valent vertex, so that $X_p \cong \PP^1 \times \PP^1$. In particular, the
homology classes of $\varphi_0(C_{0,1})$ and $\varphi_0(C_{0,2})$
in $X_p$ are distinct. As the real involution of $X_p$
acts trivially on the homology of $X_p$, we deduce that 
the components $C_{0,1}$ and $C_{0,2}$ are not complex-conjugated but are each preserved by the real involution of $C_0$.
Therefore, the points $p_1$ and $p_2$
are not complex-conjugated and are both real. It follows that 
$p$ is an hyperbolic node of $\varphi_0(C_0)$.
\end{proof}

Next we will analyse the second case, assuming two points $p_1,p_2$ belonging to the same component of $C_0$ map to the same point under $\varphi_0 \colon C_0 \to X_0$. 
Let $V=\varphi_0(p_1)=\varphi_0(p_2)$ denote the vertex in the image of the tropicalization $h \colon \Gamma \to \RR^2$, 
and let $X_V$ denote the irreducible component of $X_V$ corresponding to $V$. 
We will denote the restriction of $\varphi_0$ to the $\PP^1$-component of $C_0$ containing $p_1,p_2$ as
\begin{equation}
    \label{Eq: varphi restricted to P1}
   \varphi \colon \PP^1  \longrightarrow X_V \,.
\end{equation}

\begin{proposition}
\label{Prop: all interior points give elliptic nodes}
Let $V$ be a trivalent vertex of $\Gamma$. 
Then all real nodes in the image of the map $\varphi \colon \PP^1  \to X_V$ in \eqref{Eq: varphi restricted to P1} are elliptic.
\end{proposition}

\begin{proof}
Let $(\mathbf{u},
\mathbf{w})=((u_i)_{1\leq i\leq 3},(w_i){1\leq i\leq 3})$ be the tuple of direction vectors and weights associated to the edges $E_i$ adjacent to $V$, 
for $i\in \{1,2,3 \}$. Denote by $D_i$ the toric divisor of $X_V$ corresponding to $E_i$. As in the proof of Lemma \ref{Lem: torsor}, the real
line $\varphi \colon \PP^1  \rightarrow X_V$ in $X_V$ of type $(\mathbf{u},
\mathbf{w})$ is defined by the equation
\begin{equation} \label{eq: varphi}
\varphi^{\times}(z^n)=\chi_{\varphi}(n) \prod_{i=1}^3 (y-y(q_i))^{(w_i u_i, n)}
\end{equation}
for every $n \in N$, where $\chi_{\varphi}$ is a real character of $N$, 
and denotes $y$ the unique real coordinate on $\PP^1$. 
By $y(q_i)$ we denote the coordinate of the point $q_i \in \PP^1$ whose image under $\varphi$ intersects the toric boundary divisor $D_i$, 
and we assume $y(q_1)=-1, y(q_2)=0, y(q_3)=1$.

Now let us assume that a real node of $\varphi(\PP^1)$ is hyperbolic. Then, we can find $y$ and $y'$ real with 
$y \neq y'$ and $\varphi(y)=\varphi(y')$. Using \eqref{eq: varphi}, the condition $\varphi(y)=\varphi(y')$ can be rewritten as 
\begin{equation} \label{eq: node_condition}
\prod_{i=1}^3 (y-y(q_i))^{(w_i u_i, n)}
=\prod_{i=1}^3 (y'-y(q_i))^{(w_i u_i, n)}
\end{equation}
for every $n \in N$. Applying
\eqref{eq: node_condition} to 
$n=\det(-,w_1 u_1)$, we obtain 
\begin{equation} \label{eq: node_condition_1}
\left( \frac{y+1}{y'+1} \right)^{\mu} = \left( \frac{y}{y'} \right)^{\mu} \,,
\end{equation}
where, recalling from Definition \ref{tropcurve}
the balancing condition \eqref{Eq:balancing_condition}
$w_1u_1+w_2u_2+w_3u_3=0$, we have  
\[ \mu \coloneqq w_1 w_2 \det(u_1, u_2)
=w_3 w_1 \det(u_3,u_1)=w_2 w_3 \det(u_2,u_3)\,, \]
which
equals the multiplicity of the vertex $V$
possibly up to a sign. 
As both $\frac{y+1}{y'+1}$ and $\frac{y}{y'}$ are real, it follows from 
\eqref{eq: node_condition_1} that either
\begin{equation} \label{eq: node_condition_2}
\frac{y+1}{y'+1}=\frac{y}{y'}
\end{equation}
or 
\begin{equation} \label{eq: node_condition_3}
\frac{y+1}{y'+1}=-\frac{y}{y'}\,.
\end{equation}
But \eqref{eq: node_condition_2} implies $y=y'$, in contradiction with the assumption that $y \neq y'$. Therefore, 
\eqref{eq: node_condition_3} holds and so we have 
\begin{equation} \label{eq: node_condition_4}
    2yy'=-y-y' \,.
\end{equation}

Applying \eqref{eq: node_condition} to 
$n=\det(-,w_3 u_3)$, we obtain 
\begin{equation} \label{eq: node_condition_5}
    \left(\frac{y-1}{y'-1}\right)^{\mu}=\left( \frac{y}{y'} \right)^{\mu} \,.
\end{equation}
As both 
$\frac{y-1}{y'-1}$ and $\frac{y}{y'}$ are real, it follows from 
\eqref{eq: node_condition_5} that either
\begin{equation} \label{eq: node_condition_6}
\frac{y-1}{y'-1}=\frac{y}{y'}
\end{equation}
or 
\begin{equation} \label{eq: node_condition_7}
\frac{y-1}{y'-1}=-\frac{y}{y'}\,.
\end{equation}
But \eqref{eq: node_condition_6} implies $y=y'$, in contradiction with the assumption that $y \neq y'$. Therefore, 
\eqref{eq: node_condition_7} holds and so we have
\begin{equation} \label{eq: node_condition_8}
    2yy'=y+y' \,.
\end{equation}
Comparing \eqref{eq: node_condition_4} and \eqref{eq: node_condition_8}, we obtain
first $yy'=0$ and then $y=y'=0$, which is a contradiction with the assumption 
$y \neq y'$. 

Therefore, we conclude that $\varphi(\PP^1)$ does not admit hyperbolic real nodes and so all the real nodes of $\varphi(\PP^1)$ are elliptic.
\end{proof}

\subsection{Correspondence theorem with Welschinger signs}

\label{Sec: corr thm W signs}

We specialise the setup of \S\ref{sec:The Main Theorem} with $\mathrm{rank}M=2$ and zero-dimensional constraints $\mathbf{A}=(A_1,\ldots,A_\ell)$. 
For $t \in \A^1(\R) \setminus \{0\}
\simeq \R^{\times}$ general, we have a finite set 
$M_{(g,\Delta,\mathbf{A},\mathbf{P}),t}^{\R-log}$
of real stable log maps to $X_t$ matching the constraints $\mathbf{Z}_{\mathbf{A},\mathbf{P},t}$.
By Proposition \ref{prop_only_nodes}, for every 
$(\varphi \colon C \rightarrow X)$
element of $M_{(g,\Delta,\mathbf{A},\mathbf{P}),t}^{\R-log}$, the real curve $\varphi(C)$ is nodal, and so it makes sense to define 
$\mathcal{W}^{log}(\varphi)$
following Definition 
\ref{Def: log Welschinger sign}.
We define \emph{log Welschinger numbers} 
\[ \mathcal{W}^{\RR-log}_{(g,\Delta,\mathbf{A},\mathbf{P}),t} \coloneqq \sum_{ \varphi \in M_{(g,\Delta, \mathbf{A},\mathbf{P}),t}^{\RR-log}}\mathcal{W}^{log}(\varphi)\,.\]

On the other hand, we have
finite set $\T_{g,\ell,\Delta}(\mathbf{A})$ of tropical curves matching the tropical constraints $\mathbf{A}$.
For every 
$(\Gamma,\mathbf{E},h)$ element of 
$\T_{g,\ell,\Delta}^{\R-log}$, we can define 
$\mathrm{Mult}_{\R}(h)$ following Definition
\ref{Def:tropical W mult}.
We define \emph{tropical Welschinger numbers}
\[\mathcal{W}^{\RR-trop}_{(g,\Delta,\mathbf{A},\mathbf{P})} \coloneqq \sum_{(\Gamma,\mathbf{E},h) \in 
\mathcal{T}_{g,\ell,\Delta}(\mathbf{A})} \mathrm{Mult}_\RR(h)\,. \]

\begin{theorem}
\label{Thm: log Welschinger equals tropical}
For every $t \in \A^1(\R) \setminus \{0\}
\simeq \R^{\times}$ sufficiently close to $0$, we have
\begin{equation}
\mathcal{W}^{\RR-log}_{(g,\Delta,\mathbf{A},\mathbf{P}),t}
=
\mathcal{W}^{\RR-trop}_{(g,\Delta,\mathbf{A},\mathbf{P})}\,.
\end{equation}
\end{theorem}
\begin{proof}
Let $\underline{\varphi}_0 \colon \underline{C}_0 \rightarrow X_0$
be a maximally degenerate real stable map. Let 
$h \colon \Gamma \rightarrow M_{\RR}$
be the corresponding tropical curve. The proof is divided into two cases:
either $\Gamma$ contains one bounded edge of even weight, or all the bounded edges of $\Gamma$ have odd weight.

Let us assume that $\Gamma$ contains a bounded edge $E$ of even weight $w(E)$.
By Theorem \ref{Thm: lifting md curves to log maps}, in trying to lift 
$\underline{\varphi}_0 \colon \underline{C}_0 \rightarrow X_0$
to a real stable log map 
$\varphi_0 \colon C_0 \rightarrow X_0$, we have two choices for each bounded edge of 
$\Gamma$ of even weight, labeled by $\zeta=1$ or $\zeta=-1$.
Therefore, we can divide all real log lifts of $\underline{\varphi}_0$
into pairs, where two elements $\varphi_0^+$ and $\varphi_0^-$ of a pair only differ by the sign of 
$\zeta$ specifying the log structure at the node corresponding to the edge $E$.
By Theorem \ref{thm_signs_edges}, the nodes of the image curves 
of the deformations $\varphi_t^+$ and $\varphi_t^-$ 
of $\varphi_0^+$ and $\varphi_0^-$ are in a one-to-one correspondence preserving the 
real type, except for the nodes created by the smoothing of the node $p_E$ of $\underline{C}_0$ corresponding to the edge $E$. 
For either $\varphi_t^+$ and $\varphi_t^-$, the $w(E)-1$ nodes of the image curve created  by the smoothing of $p_E$ are all real elliptic, 
whereas for the other, the $w(E)-1$ nodes of the image curve created by the smoothing of $p_E$ split into one real hyperbolic node and $w(E)-2$ non-real nodes. 
It follows that the contributions of $\varphi_+$ and $\varphi_-$
to $\mathcal{W}^{\RR-log}_{(g,\Delta,\mathbf{A},\mathbf{P}),t}$ differ by 
$(-1)^{w(E)-1}=-1$ and so they cancel. 
Therefore, the total contribution to  $\mathcal{W}^{\RR-log}_{(g,\Delta,\mathbf{A},\mathbf{P}),t}$ 
of all the real log lifts of $\underline{\varphi_0}$ is zero. 
On the other hand, the contribution of the
tropical curve $h \colon \Gamma \rightarrow M_{\R}$ to $\mathcal{W}^{\RR-trop}_{(g,\Delta,\mathbf{A},\mathbf{P})}$ is also zero
by Definition \ref{Def:tropical W mult}.

Let us assume that all edges of $\Gamma$ have odd weight.
By Theorem \ref{Thm: lifting md curves to log maps}, there is a unique lift
of $\underline{\varphi}_0$ to a real stable log map 
$\varphi_0 \colon C_0 \rightarrow X_0$. 
Let $\varphi_t \colon C_t \rightarrow X_t$ be the deformation of $\varphi_0$.
By Definition \ref{Def: log Welschinger sign}, the contribution of 
$\varphi_t$ to $\mathcal{W}^{\RR-log}_{(g,\Delta,\mathbf{A},\mathbf{P}),t}$
is $\mathcal{W}(\varphi_t)=(-1)^{m(\varphi_t)}$, where $m(\varphi_t)$
is the number of elliptic nodes of $\varphi_t(C_t)$. 
Nodes of $\varphi_t(C_t)$ fall into three categories.

First, there are nodes of $\varphi_t(C_t)$ produced by the smoothing of a node 
$p_E$ of $\underline{\varphi}_0$ corresponding to an edge $E$ of 
$\Gamma$ of weight $w(E)$.
By Theorem \ref{Thm: lifting md curves to log maps}, for every edge $E$
of $\Gamma$, there are $w(E)-1$ such nodes, all real elliptic. As $w(E)$ is odd, 
$w(E)-1$ is even, so the total factor contribution of these nodes to 
$\mathcal{W}(\varphi_t)$ is $1$.

Second, there are nodes of $\varphi_t(C_t)$ which are deformation of nodes of $\varphi_0(C_0)$ 
where the images by $\varphi_0$ of two irreducible components of $C_0$ intersect. 
All these nodes are real hyperbolic by Lemma \ref{lem:4-valent}
and so their factor contribution to $\mathcal{W}(\varphi_t)$ is $1$.

Third, there are nodes of $\varphi_t(C_t)$ which are deformation of nodes
of $\varphi_0(C_0)$ where the image by $\varphi_0$ of an irreducible
component $C_V$ of $C_0$ self-intersects. Let $V$ be the vertex of $\Gamma$
corresponding to $C_V$.
By Proposition \ref{Prop: all interior points give elliptic nodes}, all real nodes of
$\varphi_0(C_V)$ are elliptic. On the other hand, the total 
number of nodes of $\varphi_0(C_V)$ is the number $I_{\Delta_V}$ 
of integral points in the interior of the triangle $\Delta_V$
dual to $V$. 
As the total numbers of nodes and 
the number of real nodes have the same parity (because differing by the numbers of non-real nodes which come in complex conjugated pairs), 
we deduce that the factor contribution of the nodes of $\varphi_t(c_t)$ which are deformation of the nodes of 
$\varphi_0(C_V)$ is $(-1)^{I_{\Delta_V}}$.

In conclusion, we get $\mathcal{W}(\varphi_t)=\prod_V (-1)^{I_{\Delta_V}}$, 
which is equal to the contribution $\mathrm{Mult}_{\RR}(h)$ of $h \colon \Gamma 
\rightarrow M_{\R}$ to $\mathcal{W}^{\RR-trop}_{(g,\Delta,\mathbf{A},\mathbf{P})}$
by Definition \ref{Def:tropical W mult}.
\end{proof}

\subsection{Welschinger invariants}
\label{Sec:W invariants}
Let $(X,\omega, \iota_X)$ be a compact $4$-manifold endowed with
a symplectic form $\omega$ and an anti-symplectic involution $\iota_X$, that is, such
that $\iota^{*}_X \omega = - \omega$. 
Fix $d \in H_2(X;\ZZ)$ such that 
$c_1(X) \cdot d >0$, and let $x \subset X$ be a generic configuration of $c_1(X)\cdot d - 1$ distinct real 
($\iota_X$-fixed) points in $X$. For a generic almost complex structure $J$ tamed by $\omega$,
there are finitely many $J$-holomorphic rational curves in $X$ of degree $d$ and passing through $x$, all nodal. 
In particular, finitely many of these curves are real 
(that is $\iota_X$-invariant).
The count of $J$-holomorphic rational curves in $X$ of degree $d$ and passing through $x$ does not depend on $J$ and $x$: it is a genus $0$
degree $d$ Gromov-Witten invariant.
By contrast, the count of real $J$-holomorphic rational curves in $X$ of degree $d$ and passing through $x$ is not invariant in general under variation of 
$J$ and $x$.
Welschinger \cite{We,We1} understood how to construct an invariant by considering a signed count of real $J$-holomorphic rational curves in $X$.

The following definition  of Welschinger signs can be found in \cite[\S2.1]{We1}.
\begin{definition} 
\label{Def: Welschinger sign}
Let $C$ be a nodal real irreducible rational pseudo-holomorphic curve in a real rational symplectic $4$-manifold $X$ with homology class 
$d \in H_2(X,\Z)$. 
The \emph{Welschinger sign} of $C$ is defined as
\begin{equation}
    \label{Eq: mass}
    \mathcal{W}(C)\coloneqq (-1)^{m(C)}
\end{equation}
where $m(C)$, referred to as the \emph{mass of $C$}, is the total number of elliptic nodes of $C$.
\end{definition}
Note that by the adjunction formula, a nodal
rational curve of degree $d$ in a symplectic $4$-manifold has a total number of 
nodes equal to
\begin{equation}
    \label{delta}
  \delta=\frac{d \cdot d - c_1(X) \cdot d + 2}{2}  
\end{equation}
The following result is due to Welschinger.\\
\begin{theorem}(\cite[Thm.~$2.1$]{We})
For every integer $m$ ranging from $0$ to $\delta$, 
denote by $n_d(m)$ the total number of real $J$-holomorphic rational curves of mass $m$ in $X$ passing through $x$ and realizing the homology class $d$. 
Then, the number
\begin{equation}
    \label{Eq: Welschinger number}
    \mathcal{W}^{\R-symp}_{X,d}(x,J) = \sum_{m=0}^{\delta} (-1)^m n_d(m)
\end{equation}
is invariant; it neither depends on the choice of $J$ nor on the choice
of $x$ and we denote it by $\mathcal{W}^{\R-symp}_{X,d}$.
\end{theorem}

Let $X$ be a smooth toric del Pezzo surface and $d \in H_2(X,\Z)$ such that $c_1(X) \cdot d -1 >0$. 
Let $\mathcal{W}^{\R-symp}_{X,d}$ be
the corresponding Welschinger invariant.
Let 
$\Delta_v \colon M \setminus \{0\} \rightarrow \NN$ be the 
tropical degree defined by $\Delta_d(v)=d \cdot D_v$ 
if $v$ is the primitive generator of a ray of the fan of $X$
corresponding to the toric divisor $D_v$, and 
$\Delta_d(v)=0$ else.

\begin{corollary} \label{Cor:W_tropical} 
We have 
\[\mathcal{W}^{\R-symp}_{X,d}= \mathcal{W}^{\R-trop}_{(0,\Delta_d,\mathbf{A},\mathbf{P})} \,.\]
\end{corollary}

\begin{proof}
By \cite{IKS17}, for a toric del Pezzo surface $X$, Welschinger invariants are counts
with Welschinger signs of real rational curves in $X$ for the standard integrable complex structure of $X$ 
and passing through a generic configuration of real points.
Furthermore every such curve is torically transverse.
Therefore, considering a toric degeneration as in the setting of Theorem \ref{Thm: log Welschinger equals tropical}, 
we have 
\[ \mathcal{W}^{\R-symp}_{X,d}
= \mathcal{W}^{\RR-log}_{(g,\Delta,\mathbf{A},\mathbf{P}),t} \]
for every $t \in \A^1(\R) \setminus \{0\} \simeq \R^{\times}$ general, and so the 
result follows from Theorem \ref{Thm: log Welschinger equals tropical}.
\end{proof}

\bibliographystyle{plain}
\bibliography{bibliography}

\end{document}